\documentclass[11pt,reqno]{amsart}

\usepackage{amsmath,amssymb,amsthm}
\usepackage{mathtools}
\usepackage{bm}
\usepackage{natbib}
\usepackage{graphicx}
\usepackage{multirow}
\usepackage{here}
\usepackage{fullpage}
\usepackage{xcolor}
\usepackage{bm}
\usepackage{url}
\usepackage{enumitem}

\makeatletter

\@addtoreset{equation}{section}
\makeatletter

\newtheorem{theorem}{Theorem}[section]
\newtheorem{corollary}{Corollary}[section]
\newtheorem{lemma}{Lemma}[section]
\newtheorem{proposition}{Proposition}[section]
\newtheorem{assumption}{Assumption}[section]
\theoremstyle{definition}

\newtheorem{remark}{Remark}[section]
\newtheorem{example}{Example}[section]

\newcommand{\R}{\mathbb{R}}

\renewcommand{\hat}{\widehat}
\newcommand{\bY}{\bm{Y}}
\newcommand{\bs}{\bm{s}}
\newcommand{\bZ}{\bm{Z}}
\newcommand{\bX}{\bm{X}}

\DeclareMathOperator{\Var}{Var}
\DeclareMathOperator{\Cov}{Cov}

\DeclareMathOperator{\E}{E}

\newcommand{\dk}[1]{{ \color{blue}[DK: #1]}}

\begin{document}

\title[]{Gaussian approximation and spatially dependent wild bootstrap for high-dimensional spatial data}
\thanks{D. Kurisu is partially supported by JSPS KAKENHI Grant Number 20K13468. K. Kato is partially supported by NSF grants DMS-1952306 and DMS-2014636. X. Shao is partially supported by NSF grants DMS-1807032 and DMS-2014018. We would like to thank Adam Kashlak, Yasumasa Matsuda, Taisuke Otsu and Yoshihiro Yajima for their helpful comments and suggestions.} 

\author[D. Kurisu]{Daisuke Kurisu}
\author[K. Kato]{Kengo Kato}
\author[X. Shao]{Xiaofeng Shao}

\date{First version: May 16, 2019. This version: \today}

\address[D. Kurisu]{Department of Industrial Engineering and Economics, School of Engineering, Tokyo Institute of Technology\\
2-12-1 Ookayama, Meguro-ku, Tokyo 152-8552, Japan.
}
\email{kurisu.d.aa@m.titech.ac.jp}

\address[K. Kato]{
Department of Statistics and Data Science, Cornell University \\
1194 Comstock Hall, Ithaca, NY 14853, USA.
}
\email{kk976@cornell.edu}

\address[X. Shao]{
Department of Statistics, University of Illinois at Urbana-Champaign \\
Champaign, IL, 61820, USA.
}
\email{xshao@illinois.edu}

\begin{abstract}

In this paper, we establish  a high-dimensional CLT for the sample mean of $p$-dimensional spatial data observed over irregularly spaced sampling sites in $\R^d$, allowing the dimension $p$ to be much larger than the sample size $n$. We adopt a stochastic sampling scheme that can generate irregularly spaced sampling sites in a flexible manner and  include both pure increasing domain and mixed increasing domain frameworks. To facilitate statistical inference, we develop the  spatially dependent wild bootstrap (SDWB) and justify its asymptotic validity in high dimensions by deriving error bounds that hold almost surely conditionally on the stochastic sampling sites. Our dependence conditions on the underlying random field cover a wide class of random fields such as Gaussian random fields and continuous autoregressive moving average  random fields. Through numerical simulations and a real data analysis, we demonstrate the usefulness of  our bootstrap-based inference in several applications, including joint confidence interval construction for high-dimensional spatial data and change-point detection for spatio-temporal data. 

\end{abstract}


\maketitle

\section{Introduction}

Spatial data analysis plays an important role in many fields, such as atmospheric science, climate studies, ecology, hydrology and seismology. There are many classic textbooks and monographs devoted to modeling and inference of spatial data, see, e.g., \cite{Cr93}, \cite{St99}, \cite{MoWa04},  \cite{GaGu10}, and \cite{BaCaGe15}, among others. This paper aims to advance high-dimensional Gaussian approximation theory and bootstrap-based methodology related to the analysis of multivariate (and possibly high-dimensional) spatial data. Specifically, we assume that our data are from a multivariate random field $\bY = \{\bY(\bs):\bs \in \R^{d}\}$ with $\bY(\bs) = (Y_{1}(\bs),\dots,Y_{p}(\bs))'$, where $d\ge 2$ is the dimension of the spatial domain and $p\ge 2$ stands for the dimension of multivariate measurements at any location $\bs\in \R^d$. 

With recent technological advances and remote sensing technology, multivariate spatial data are becoming more prevalent. For example, levels of multiple air pollutants (e.g., ozone, PM$_{2.5}$, PM$_{10}$, nitric oxide, carbon monoxide) are monitored at many stations in many countries. To understand spatial distributions of carbon intake and emissions as well as their seasonal and annual evolutions, the total-column carbon dioxide (CO2) mole fractions (in units of parts per million) are measured using remote sensing instruments, which produce estimates of CO2 concentration, called profiles, at 20 different pressure levels; see \cite{NgCrBr17}. The latter authors treated 20 measurements at different profiles as a 20-dimensional vector and proposed a modified spatial random effect model to capture spatial dependence and multivariate dependence across profiles. For an early literature  on the modeling and inference of multivariate spatial data, we refer to \cite{GeVo03}, \cite{GeVoScBa04}, and \cite{GeBa10}.

Motivated by the increasing availability of multivariate spatial data with increasing dimensions, we shall study a fundamental problem at the intersection of spatial statistics and high-dimensional statistics: central limit theorem (CLT) for the sample mean of high-dimensional spatial data observed at irregularly spaced sampling sites. When the dimension  $p$ is low and fixed, CLTs for weighted sums of spatial data have been derived when the sampling sites lie on the $d$-dimensional integer lattice, see \cite{BuZh76}, \cite{Ne80}, \cite{Na80}, \cite{Bo82} and \cite{GuRi84}. To accommodate irregularly spaced sampling sites, which is the norm rather than the exception in spatial statistics,  \cite{La03a} introduced a novel stochastic sampling design and derived CLTs under both pure increasing domain and mixed increasing domain settings. However, so far, all these results are restricted to the case when the dimension $p$ is fixed, and there seem no CLT results that allow for the growing dimension in the literature.

To address the high-dimensional CLT for spatial data, we face the following challenges:  (1) when the dimension $p$ exceeds the sample size $n$, even for i.i.d. data, it is usually not known whether the distribution of normalized sample mean (or its norm, such as the $\ell^\infty$-norm) converges to a fixed limit, unless under very stringent assumptions; (2) spatial data have no natural ordering and sampling sites are often irregularly spaced.  In the low-dimensional setting, \cite{La03a} showed that the asymptotic variance depends on the sampling density, and the convergence rate for the sample mean depends on which asymptotic regime we adopt (pure increasing-domain versus mixed increasing domain). To meet the first challenge, we shall build on the celebrated high-dimensional Gaussian approximation techniques that have undergone recent rapid development (see a literature review below) and establish the asymptotic equivalence between the distribution of the normalized sample mean and that of its Gaussian counterpart in high dimensions. To tackle the challenge from the irregular spatial spacing,  we shall adopt the stochastic sampling scheme of \cite{La03a}, which allows the sampling sites to have a nonuniform density across the sampling region and enables the number of sampling sites $n$ to grow at a different rate compared with the volume of the sampling region $\lambda_n^d$. This scheme accommodates both \textit{the pure increasing domain} case ($\lim_{n \to \infty}n\lambda_{n}^{-d} = \kappa \in (0,\infty)$) and \textit{the mixed increasing domain} case ($\lim_{n \to \infty}n\lambda_{n}^{-d} = \infty$).  From a theoretical viewpoint, this scheme covers all possible asymptotic regimes since it is well-known that the sample mean is not consistent under the infill asymptotics \citep{La96}. See  \cite{La03b}, \cite{LaZh06} and \cite{BaLaNo15} for some detailed discussions on the stochastic spatial sampling design.

Specifically, we establish a CLT for the sample mean of high-dimensional spatial data over the rectangles when $p = p_{n} \to \infty$ as $n \to \infty$ and possibly $p \gg n$ under a weak dependence condition, where the random field is observed  at a finite number of discrete locations $\bs_{1},\dots,\bs_{n}$ in a sampling region $R_{n}$ whose volume scales as  $\lambda_{n}^{d}$, where $\lambda_{n} \to \infty$ as $n \to \infty$. To facilitate statistical inference, we propose and develop the spatially dependent wild bootstrap (SDWB, hereafter), which is an extension of the dependent wild bootstrap of \cite{Sh10} to the spatial setting,  and justify its asymptotic validity in high dimensions. 
Notably, we will show that the SDWB  works for a wide class of random fields on $\R^{d}$ that includes multivariate L\'evy-driven moving average (MA) random fields (see \cite{Ku20} for a detailed discussion of such  random fields).  L\'evy-driven MA random fields constitute a rich class of models for spatial data and include both Gaussian and non-Gaussian random fields such as continuous autoregressive moving-average (CARMA) random fields \citep{BrMa17,MaYa18}. 
To illustrate the usefulness of our theory and SDWB, we describe several applications, including  (i) simultaneous inference for the mean vector of high-dimensional spatial data; (ii) construction of confidence bands for the mean function of spatio-temporal data, and (iii) multiple change-point detection for spatio-temporal data.

\subsection*{ Contributions and Connections to the literature}
To put our contributions in perspective, we shall review two lines of research that have inspired our work. The first line is related to {\it gaussian approximation for both high-dimensional independent data and high-dimensional time series}.  There is now a large and still rapidly growing literature on high-dimensional CLTs  over the rectangles and related bootstrap theory when the dimension of the data is possibly much larger than the sample size.  For the sample mean of  independent random vectors, we refer to \cite{ChChKa13, ChChKa14, ChChKa15, ChChKa16, ChChKa17}, \cite{ChChKaKo19}, \cite{DeZh20}, \cite{KuMuBa20}, \cite{FaKo20} and \cite{ChChKo20}. For high-dimensional $U$-statistics and $U$-processes,  see \cite{Ch18} and \cite{ChKa19}. 
In the time series setting, \cite{ZhWu17} developed Gaussian approximation for the maximum of the sample mean of  high-dimensional stationary time series with equidistant observations under the physical dependence measures developed by \cite{Wu05}. Based on a nonparametric estimator for the long-run covariance matrix of the sample mean, they used a simulation-based approach to constructing simultaneous confidence intervals for the mean vector.  \cite{ZhCh18} also developed high-dimensional CLTs for the maximum of the sample mean of  high-dimensional time series under the physical/functional dependence measures and used non-overlapping block bootstrap to perform inference.
\cite{ChChKa19} studied high-dimensional CLTs for the maximum of the sum of $\beta$-mixing and possibly non-stationary time series and showed the asymptotic validity of a block multiplier bootstrap method. Also see \cite{ChYaZh17}, \cite{Ko19}, and \cite{YuCh20} among others for the use of Gaussian approximation or variants in high-dimensional testing problems. 

To the best of our knowledge, our work is the first paper that establishes a high-dimensional CLT for spatial data and rigorously justifies the asymptotic validity of a bootstrap method for high-dimensional data in the spatial setting.  From a technical point of view, the present paper builds on \cite{ChChKa17},  \cite{ChChKa19}, \cite{ZhWu17} and \cite{ZhCh18}, but our theoretical analysis differs substantially from those references in several important aspects. Specifically, (i) we establish a high-dimensional CLT and the asymptotic validity of SDWB that hold almost surely conditionally on the stochastic sampling sites. Precisely, we show that the conditional distribution of the sample mean (or its SDWB counterpart) given the sampling sites can be approximated by a (conditionally) Gaussian distribution.   The randomness of the sampling sites yields additional technical complications in high dimensions; see e.g. Lemma \ref{variance-rate}. (ii) We extend the coupling technique used in \cite{Yu94} to irregularly spatial data to prove the high-dimensional CLT.  This extension is nontrivial since there is no natural ordering for spatial data and the number of observations in each block constructed is random.  Our approach to the blocking construction is also quite different from those in \cite{La03b} and \cite{LaZh06} whose proofs essentially rely on approximating the characteristic function of the weighted sample mean by that of independent blocks; see also Remark \ref{rem: lahiri}. (iii) We explore in detail concrete random fields that satisfy our weak dependence condition and other regularity conditions. Indeed, we show that our regularity conditions can be satisfied for a wide class of multivariate L\'evy-driven MA random fields  that constitute a rich class of models for spatial data \citep[cf.][]{BrMa17,MaYa18,MaYu20} but whose mixing properties have not been investigated so far. Verification of our regularity conditions to L\'{e}vy-driven MA fields is indeed nontrivial and relies on several probabilistic techniques from L\'{e}vy process theory and theory of infinitely divisible random measures \citep{Be96,Sa99,RaRo89}.

 Our work also builds on the literature of {\it bootstrap methods for time series and spatial data}. 
For both time series and spatial data, the block-based bootstrap (BBB) has been fairly well studied since the introduction of moving block bootstrap (MBB) by \cite{Ku89} and \cite{LiSi92}. Among many variants of the MBB, we mention \cite{Ca86} for the non-overlapping block bootstrap, \cite{PoRo92} for the circular bootstrap, \cite{PoRo94} for the stationary bootstrap, and \cite{PaPo01, PaPo02} for the tapered block bootstrap. See \cite{La03b} for a book-length treatment.  The BBB methods have been extended to spatial framework for both regular lattice and irregularly spaced non-lattice data. See e.g., \cite{PoRo93}, \cite{PoPaRo99}, \cite{LaZh06}, and \cite{ZhLa07}.

As we mentioned before, the proposed SDWB  is an extension of the dependent wild bootstrap in \cite{Sh10}, which was developed for time series data. The main difference between SDWB and DWB  is that the SDWB observations in $\R^d$ were generated by simulating an auxiliary random field with suitable covariance function on $\R^d$ to mimic the spatial dependence. In contrast, the DWB in \cite{Sh10} aims to capture temporal dependence when $d=1$. Thus the multipliers (or external variables) in SDWB are spatially dependent, hence the name SDWB. 
As argued in \cite{Sh10}, the DWB/SDWB is much easier to implement for irregularly spaced data than BBB, as the latter requires partitioning the sampling region into blocks and can be less convenient to implement due to incomplete blocks. The SDWB is also different from the block multiplier bootstrap (BMB) proposed in \cite{ChChKa19} since the multipliers of the BMB are i.i.d. Gaussian random variables, while the multipliers of SDWB are dependent Gaussian random variables generated from a stationary Gaussian random field on the irregular spaced sampling sites.

   \cite{Sh10} also demonstrated favorable theoretical properties of DWB for both equally spaced and unequally spaced time series when $d=1$.
Our work is distinct from \cite{Sh10} in several important aspects. (i) Since we are dealing with random fields, our technical assumptions are based on $\beta$-mixing conditions for random fields, which is a nontrivial extension of $\beta$-mixing condition for time series; (ii) all of our results are stated conditional on the underlying (randomly sampled) sampling sites, which is stronger than the unconditional version obtained in \cite{Sh10} for the case $d=1$; (iii)   most importantly, the analysis of \cite{Sh10} is focused on the case where the dimension $p$ is fixed and relies on the explicit limit distribution of the normalized sample mean, while in the high-dimensional case, there are no explicit limit distributions, and the asymptotic analysis of the SDWB is substantially more involved than that of \cite{Sh10}. Overall, the technical assumptions and probabilistic tools we use are considerably different due to our focus on high-dimensional Gaussian approximation for random fields. 
Compared to other bootstrap methods and associated theory developed for spatial data \citep{LaZh06,ZhLa07}, our bootstrap-based inference is targeted at a high-dimensional parameter and our theoretical argument is substantially different. 

The rest of the paper is organized as follows. In Section \ref{sec: sampling}, we introduce the asymptotic framework for the sampling region, stochastic design of sampling locations, and dependence structure of the random field. In Section \ref{sec: SDWB}, we introduce the spatially dependent wild bootstrap and describe its implementation.  In Section \ref{sec: Main-results}, we present a high-dimensional CLT for the sample mean of high-dimensional spatial data and derive the asymptotic validity of  SDWB. In Section \ref{sec: examples}, we discuss a class of random fields that satisfy our regularity conditions. We present some applications of SDWB  in Section \ref{sec: applications}. In Section \ref{sec: simulations}, we investigate finite sample properties of the SDWB via numerical simulations. All the proofs and a real data illustration are included in the supplement.  

\subsection{Notation}
For any vector $\bm{x} = (x_{1},\dots, x_{q})' \in \R^{q}$, let $|\bm{x}| =\sum_{j=1}^{q}|x_{j}|$ and $\|\bm{x}\| =\sqrt{\sum_{j=1}^{q}x_j^2}$ denote the $\ell^{1}$ and $\ell^{2}$-norms of $\bm{x}$, respectively. For two vectors $\bm{x} = (x_{1},\dots, x_{q})'$ and $\bm{y} = (y_{1},\dots, y_{q})' \in \R^{q}$, the notation $\bm{x} \leq \bm{y}$ means that $x_{j} \leq y_{j}$ for all $j=1,\dots,q$. For any set $A \subset \R^{q}$, let $|A|$ denote the Lebesgue measure of $A$, and let $[\![A]\!]$ denote the number of elements in $A$. For any positive sequences $a_{n}, b_{n}$, we write $a_{n} \lesssim b_{n}$ if there is a constant $C >0$ independent of $n$ such that $a_{n} \leq Cb_{n}$ for all $n$,  $a_{n} \sim b_{n}$ if $a_{n} \lesssim b_{n}$ and $b_{n} \lesssim a_{n}$, and $a_{n} \ll b_{n}$ if $a_{n}/b_{n} \to 0$ as $n \to \infty$. For any $a,b \in \R$, let $a \vee b = \max\{a,b\}$ and $a \wedge b = \min\{a,b\}$. For $a \in \R$ and $b>0$, we use the shorthand notation $[a\pm b] = [a-b,a+b]$. Let $\|X\|_{\psi_{1}} = \inf\left\{c>0 : E\left[\exp(|X|/c)-1\right] \leq 1\right\}$ denote the $\psi_{1}$-Orlicz norm for a real-valued random variable $X$. For random variables $X$ and $Y$, we write $X \stackrel{d}{=} Y$ if they have the same distribution.

\section{Settings}\label{sec: sampling}
In this section, we discuss mathematical settings of our sampling design and spatial dependence structure. We observe discrete samples $\bY(\bs_1),\dots,\bY(\bs_n)$ from a random field $\bY = \{ \bY(\bs) : \bs \in \R^d \}$ with $\bY (\bs) = (Y_1(\bs),\dots,Y_p(\bs))' \in \R^p$ and are interested in approximating the distribution of the sample mean $\overline{\bY}_n = n^{-1}\sum_{i=1}^n \bY(\bs_i)$ when $p = p_n \to \infty$ as $n \to \infty$ and possibly $p \gg n$. The sampling sites $\bs_1,\dots,\bs_n \in \R^d$ are stochastic and obtained by rescaling i.i.d. random vectors $\bZ_1,\dots,\bZ_n$; see below for details.   Let $(\Omega^{(j)}, \mathcal{F}^{(j)},P^{(j)})$, $j=1,2,3$ be probability spaces on which the random field $\bY$, a sequence of i.i.d. random vectors $\{\bZ_{i}\}_{i \geq 1}$ with values in $\R^{d}$, and an auxiliary real-valued Gaussian random field $W = \{W(\bs): \bs \in \R^{d}\}$ are defined, respectively. The auxiliary Gaussian random field $W$ will be used in the construction of SDWB. Consider the product probability space  $(\Omega, \mathcal{F},P)$ where $\Omega = \Omega^{(1)} \times \Omega^{(2)} \times \Omega^{(3)}$, $\mathcal{F} = \mathcal{F}^{(1)} \otimes \mathcal{F}^{(2)} \otimes \mathcal{F}^{(3)}$, and $P = P^{(1)} \times P^{(2)} \times P^{(3)}$. Then $\bY$, $\{\bZ_{i}\}_{i \geq 1}$, and $W$ are independent by construction. Let $P_{\bZ}$ denote the joint  distribution of the sequence of i.i.d. random vectors $\{\bZ_{i}\}_{i \geq 1}$ and let $P_{\cdot \mid \bZ}$ denote the conditional  probability given $\sigma (\{\bZ_{i}\}_{i \geq 1})$, the $\sigma$-field generated by $\{\bZ_{i}\}_{i \geq 1}$. Let $E_{\bZ}$ denote the expectation with respect to $\{\bZ_{i}\}_{i \geq 1}$ and let $E_{\cdot \mid \bZ}$ and $\Var_{\cdot \mid \bZ}$ denote the conditional expectation and variance given $\sigma(\{\bZ_{i}\}_{i \geq 1})$, respectively. Finally, let $P_{\cdot|\bY,\bZ}$ and $\Var_{\cdot|\bY, \bZ}$ denote the conditional probability and variance given $\sigma(\{\bY(\bs): \bs \in \R^{d}\}\cup \{\bZ_{i}\}_{i \geq 1})$, respectively.

\subsection{Sampling design}
We follow the setting considered in \cite{La03a} and define the sampling region $R_{n}$ as follows. Let $R_{0}^{\ast}$ be an open connected subset of $(-1/2,1/2]^{d}$ containing the origin and let $R_{0}$ be a Borel set satisfying $R_{0}^{\ast} \subset R_{0} \subset \overline{R}_{0}^{\ast}$, where for any set $A \subset \R^{d}$, $\overline{A}$ denotes its closure. Let $\{\lambda_{n}\}_{n \geq 1}$ be a sequence of positive numbers such that $\lambda_{n} \to \infty$ as $n \to \infty$.  We consider the following set as the sampling region 
\[
R_{n} = \lambda_{n}R_{0}. 
\]
To avoid pathological cases, we also assume that for any sequence of positive numbers $\{a_{n}\}_{n \geq 1}$ with $a_{n} \to 0$ as $n \to \infty$, the number of cubes of the form $a_{n}(\bm{i} +[0,1)^{d})$, $\bm{i} \in \mathbb{Z}^{d}$ with their lower left corner $a_{n}\bm{i}$ on the lattice $a_{n}\mathbb{Z}^{d}$ that intersect both $R_{0}$ and $R_{0}^{c}$ is $O(a_{n}^{-d+1})$ as $n \to \infty$. 

Next we introduce our (stochastic) sampling designs. Let $f$ be a continuous, everywhere positive probability density function on $R_{0}$, and let $\{\bZ_{i}\}_{i \geq 1}$ be a sequence of i.i.d. random vectors with  density $f$. Recall that $\{\bZ_{i}\}_{i \geq 1}$ and $\bY$ are independent from the construction of the probability space $(\Omega, \mathcal{F}, P)$.  We assume that the sampling sites $\bs_{1},\dots, \bs_{n}$ are obtained from realizations $\bm{z}_1,\dots, \bm{z}_n$ of the random vectors $\bZ_{1},\dots, \bZ_{n}$ and the relation
\[
\bs_{i} = \lambda_{n}\bm{z}_{i},\ i=1,\dots, n. 
\]  

In practice, $\lambda_{n}$ can be determined by the diameter of a sampling region. See e.g., \cite{HaPa94} and \cite{MaYa09}. The boundary condition on the prototype set $R_{0}$ holds in many practical situations, including many convex subsets in $\R^{d}$ such as spheres, ellipsoids, polyhedrons, as well as many non-convex sets in $\R^{d}$. See also \cite{La03a} and Chapter 12 in \cite{La03b} for more discussions.

\subsection{Dependence structure}
In what follows, we  assume that the  random field $\bY = \{ \bY(\bs) : \bs \in \R^d \}$ can be decomposed as
\begin{align}
\bY(\bs) &= \bX(\bs) + \bm{\Upsilon}(\bs), \quad \bs \in \R^d, \label{beta-decomp-approx}
\end{align}
where $\bX =\{ \bX(\bs) : \bs \in \R^{d} \}$ with $\bX(\bs) = (X_{1}(\bs),\dots, X_{p}(\bs))'$ is a strictly stationary random field and $\bm{\Upsilon} = \{ \bm{\Upsilon}(\bs) : \bs \in \R^{d} \}$ with $\bm{\Upsilon}(\bs) = (\Upsilon_{1}(\bs),\dots, \Upsilon_{p}(\bs))'$ is a ``residual'' random field such that 
for some $\zeta>0$,
\begin{align}\label{AN-error}
P_{\cdot \mid \bZ}\left(\max_{1 \leq j \leq p}\left|{1 \over \sqrt{n^{2}\lambda_{n}^{-d}}}\sum_{i=1}^{n}\Upsilon_{j}(\bs_{i})\right| > n^{-\zeta}\log^{-1/2} p\right) &=O(n^{-\zeta}) \quad \text{with $P_{\bZ}$-probability one}.
\end{align}
The decomposition (\ref{beta-decomp-approx}) may (and in general does) depend on $n$, i.e., $\bX = \bX^{(n)}$ and $\bm{\Upsilon} = \bm{\Upsilon}^{(n)}$, but the dependence on $n$ is suppressed for  notational convenience. Throughout the paper, we assume that $E[\bm{\Upsilon}(\bs)]=0$ for any $\bs\in \R^d$. Then $\bY$
is approximately stationary with constant mean.

We also  assume that the random field $\bX$ satisfies a certain mixing condition. Let $\sigma_{\bX}(T) = \sigma(\{\bX(\bs): \bs \in T\})$ denote the $\sigma$-field generated by $\{\bX(\bs): \bs \in T\}$ for $T \subset \R^{d}$. For any subsets $T_{1}$ and $T_{2}$ of $\R^{d}$, the $\beta$-mixing coefficient between $\sigma_{\bX}(T_1)$ and $\sigma_{\bX}(T_2)$ is defined by
\[
\check{\beta}(T_{1}, T_{2}) = \sup \frac{1}{2} \sum_{j=1}^{J}\sum_{k=1}^{K}|P(A_{j}\cap B_{k}) - P(A_{j})P(B_{k})|,
\]
where the supremum is taken over all partitions $\{ A_j\}_{j=1}^{J} \subset \sigma_{\bX}(T_{1})$ and $\{ B_k\}_{k=1}^{K} \subset \sigma_{\bX}(T_{2})$ of $\R^d$. 
 Let $\mathcal{R}(b)$ denote the collection of all finite disjoint unions of cubes in $\R^{d}$ with  total volume not exceeding $b$.
Then, we define
\begin{align}\label{beta-mix-def}
\beta(a;b) = \sup \big\{ \check{\beta}(T_{1},T_{2}): d(T_{1},T_{2}) \geq a, T_{1}, T_{2} \in \mathcal{R}(b)\big\}, \ a,b > 0,
\end{align}
where $d(T_{1},T_{2}) = \inf\{|\bm{x} - \bm{y}|: \bm{x} \in T_{1}, \bm{y} \in T_{2}\}$. 
We assume that there exist a nonincreasing function $\beta_{1}$ with $\lim_{a \to \infty}\beta_{1}(a) = 0$ and a nondecreasing function $g$ (that may be unbounded) such that 
\begin{equation}
\beta(a;b) \leq \beta_{1}(a)g(b),\ a,b > 0.
\label{eq: mixing}
\end{equation}

Our mixing condition (\ref{eq: mixing}) is a $\beta$-mixing version of the $\alpha$-mixing condition considered in \cite{La03a}, \cite{LaZh06}, and \cite{BaLaNo15}. In general, the function $\beta_{1}$ may depend on $n$ since the random field $\bm{X}$ that appears in (\ref{beta-decomp-approx}) depends on $n$. Here we assume that $g$ does not depend on $n$ for simplicity, but the extension to the general case that $g$ changes with $n$ is not difficult. The random field $\bY$ itself may not satisfy the mixing condition (\ref{eq: mixing}), since  the mixing condition (\ref{eq: mixing}) is assumed on $\bX$. With the decomposition (\ref{beta-decomp-approx}), we allow $\bY$ to have a flexible dependence structure since the residual random field $\bm{\Upsilon}$ can accommodate a complex dependence structure. In particular, we will show in Section \ref{sec: examples} that a wide class of L\'{e}vy-driven MA random fields admit the decomposition satisfying Condition (\ref{AN-error}).

\begin{remark}
\label{rem: lahiri}
\cite{La03b}, \cite{LaZh06}, and \cite{BaLaNo15} assume the $\alpha$-mixing version of Condition (\ref{eq: mixing}) to prove limit theorems for spatial data in the fixed dimensional case (i.e., $p$ is fixed).  \cite{La03b} established CLTs for weighted sample means of spatial data under an $\alpha$-mixing condition in the univariate case. 
Lahiri's proof relies essentially on approximating the characteristic function of the weighted sample mean by that of independent blocks using the Volkonskii-Rozanov inequality \citep[cf. Proposition 2.6 in][] {FaYa03} and then showing that the characteristic function corresponding to the independent blocks converges to the  characteristic function of its Gaussian limit. 
However, in the high-dimensional case ($p_{n} \to \infty$ as $n \to \infty$),  characteristic functions are difficult to capture the effect of dimensionality in approximation theorems, so  we rely on a different argument than that of \cite{La03b}. Indeed,  we use a stronger blocking argument tailored to $\beta$-mixing sequences; cf. Lemma 4.1 in \cite{Yu94}. As discussed in the latter paper, the corresponding results would not hold for the $\alpha$-mixing case; see Remark (ii) right after the proof of Lemma 4.1 in \cite{Yu94}. Hence we assume Condition (\ref{eq: mixing}) in the present paper. 
\end{remark}

\begin{remark}
It is important to restrict the size of index sets $T_{1}$ and $T_{2}$ in the definition of $\beta(a;b)$. Define the $\beta$-mixing coefficient of a random field $\bX$ similarly to the time series as follows: Let $\mathcal{O}_{1}$ and $\mathcal{O}_{2}$ be half-planes with boundaries $L_{1}$ and $L_{2}$, respectively. For each $a>0$, define $\beta(a) = \sup\left\{ \check{\beta}(\mathcal{O}_{1},\mathcal{O}_{2}) : d(\mathcal{O}_{1},\mathcal{O}_{2}) \geq a\right\}$. 
According to Theorem 1 in \cite{Br89}, if $\{\bX(s): \bs \in \R^{2}\}$ is  strictly stationary, then $\beta(a) = 0$ or $1$ for $a >0$. This implies that if a random field $\bX$ is $\beta$-mixing ($\lim_{a \to \infty}\beta(a)=0$), then it is automatically $m$-dependent, i.e., $\beta(a)=0$ for some $a>m$, where $m$ is a positive constant. To allow for certain flexibility, we restrict the size of $T_1$ and $T_2$ in the definition of $\beta(a;b)$. We refer to \cite{Br93} and \cite{Do94} for more details on mixing coefficients for random fields. 
\end{remark}


\section{Spatially dependent wild bootstrap}\label{sec: SDWB}

In this section, we introduce  the spatially dependent wild bootstrap (SDWB) method for the construction of joint confidence intervals for the mean vector $\bm{\mu} = E[\bY(\bs)] = (\mu_{1},\dots,\mu_{p})'$.  Let $\overline{\bY}_{n} = n^{-1}\sum_{i=1}^{n}\bY(\bs_{i}) = (\overline{Y}_{1,n},\dots, \overline{Y}_{p,n})'$ denote the sample mean. In Section \ref{sec: GA-mean-result}, we will  show that under certain regularity conditions, as $n \to \infty$, 
\[
\sup_{A \in \mathcal{A}}\left|P_{\cdot \mid \bZ}\left(\sqrt{\lambda_{n}^{d}}(\overline{\bY}_{n} - \bm{\mu}) \in A\right) - P_{\cdot \mid \bZ}\left(\bm{V}_{n} \in A\right)\right| \to 0 \quad \text{$P_{\bZ}$-a.s.},
\]
provided that $p = O(n^{\alpha}$) for some $\alpha > 0$, where $\mathcal{A}$ is the class of closed rectangles in $\R^{p}$. Here $\bm{V}_n = (V_{1,n},\dots, V_{p,n})'$ is a centered Gaussian random vector in $\R^{p}$ under $P_{\cdot \mid \bZ}$ with (conditional) covariance matrix $\Sigma^{\bm{V}_{n}} = (\Sigma_{j,k}^{\bm{V}_{n}})_{1 \leq j,k \leq p} = E_{\cdot \mid \bZ}[\bm{V}_{n}\bm{V}'_{n}]$, the form of which is specified later.  This high-dimensional CLT implies that 
a joint $100(1-\tau)\%$ confidence interval for the mean vector $\bm{\mu}$ with $\tau \in (0,1)$ is given by $\hat{C}_{1-\tau} = \prod_{j=1}^{p}\hat{C}_{j,1-\tau}$, where
\begin{align*}
\hat{C}_{j,1-\tau} &= \left[\overline{Y}_{j,n} \pm \lambda_{n}^{-d/2}\sqrt{\Sigma_{j,j}^{\bm{V}_{n}}}q_{n}(1-\tau)\right],
\end{align*}
and  $q_{n}(1-\tau)$ is the ($1-\tau$)-quantile of $\max_{1 \leq j \leq p}\left|V_{j,n}/\sqrt{\Sigma_{j,j}^{\bm{V}_{n}}}\right|$. Indeed, we have 
\[
P_{\cdot \mid \bZ}\left( \bm{\mu} \in \hat{C}_{1-\tau}\right) = P_{\cdot \mid \bZ}\left(\max_{1 \leq j \leq p}\left|V_{j,n}/\sqrt{\Sigma_{j,j}^{\bm{V}_{n}}}\right| \leq q_{n}(1-\tau)\right)  + o(1) \geq 1-\tau + o(1)
\]
with $P_{\bZ}$-probability one, so that $\hat{C}_{1-\tau}$ is a valid joint confidence interval for $\bm{\mu}$ with level approximately $1-\tau$. 

In practice, we have to  estimate the quantile $q_{n}(1-\tau)$, in addition to the coordinatewise variances $\Sigma_{j,j}^{\bm{V}_{n}}$. To this end, we develop  the spatially dependent wild bootstrap (SDWB), as an extension of DWB proposed by \cite{Sh10} to the spatial setting. Given the observations $\{\bY(\bs_{i})\}_{i=1}^{n}$, we define the SDWB pseudo-observations as 
\[
\bY^{\ast}(\bs_{i}) = \overline{\bY}_{n} + (\bY(\bs_{i}) - \overline{\bY}_{n})W(\bs_{i}),\ i = 1,\dots,n,
\]
where 
$\{W(\bs_{i})\}_{i=1}^{n}$ are $n$ discrete samples from a real-valued stationary Gaussian random field $W = \{W(\bs): \bs \in \R^{d}\}$ such that $E[W(\bm{s})]=0$, $\Var(W(\bm{s}))=1$ and $\Cov(W(\bs_{1}), W(\bs_{2})) = a(\|\bs_{2} - \bs_{1}\|/b_{n})$. Here $a(\cdot): \R \to [0,1]$ is a continuous kernel function supported in $[-1,1]$ and $b_{n}$ is a bandwidth such that $b_{n} \to \infty$ as $n \to \infty$. 

To estimate the covariance matrix $\Sigma^{\bm{V}_{n}}$, we use the classical
lag-window type estimator defined as 
\begin{align}
\hat{\Sigma}^{\bm{V}_{n}} &= {1 \over n^{2}\lambda_{n}^{-d}}\sum_{\ell_{1},\ell_{2}=1}^{n}(\bY(\bs_{\ell_{1}}) - \overline{\bY}_{n})(\bY(\bs_{\ell_{2}}) - \overline{\bY}_{n})'\Cov(W(\bs_{\ell_{1}}), W(\bs_{\ell_{2}})) \nonumber \\
&= {1 \over n^{2}\lambda_{n}^{-d}}\sum_{\ell_{1},\ell_{2}=1}^{n}(\bY(\bs_{\ell_{1}}) - \overline{\bY}_{n})(\bY(\bs_{\ell_{2}}) - \overline{\bY}_{n})'a(\|\bs_{\ell_{1}} - \bs_{\ell_{2}}\|/b_{n}). \label{SDWB-cov-est-def}
\end{align} 
Denote by $\hat{\Sigma}_{j,k}^{\bm{V}_n}$ the $(j,k)$-th component of $\hat{\Sigma}^{\bm{V}_{n}}$. Let $\overline{\bY}^{\ast}_{n} = n^{-1}\sum_{i=1}^{n}\bY^{\ast}(\bs_{i}) = (\overline{Y}^{\ast}_{1,n},\dots,\overline{Y}^{\ast}_{p,n})'$. It is not difficult to see that $\hat{\Sigma}^{\bm{V}_{n}} = \lambda_n^d \Var_{\cdot|\bY,\bZ}(\overline{\bY}^{\ast}_{n}) $.
That is, the SDWB variance estimator coincides with the lag window estimator provided that the covariance function and bandwidth used in SDWB match the kernel and bandwidth in the above expression.

Then we can estimate the quantile $q_{n}(1-\tau)$ by 
\begin{align}\label{SDWB-quantile-est-def}
\hat{q}_{n}(1-\tau) = \inf\left\{t \in \R : P_{\cdot|\bY,\bZ}\left(\sqrt{\lambda_{n}^{d}}\max_{1 \leq j \leq p}\left|{\overline{Y}^{\ast}_{j,n} - \overline{Y}_{j,n} \over \sqrt{\hat{\Sigma}_{j,j}^{\bm{V}_{n}}}} \right|  \leq t\right) \geq 1-\tau\right\}.
\end{align}  
We will show in Sections \ref{sec: SDWB-result} and \ref{joint-CI-mean} that the plug-in joint confidence interval $\hat{C}_{1-\tau}$ with $\Sigma_{j,j}^{\bm{V}_{n}}$ and $q_{n}(1-\tau)$ replaced by $\hat{\Sigma}_{j,j}^{\bm{V}_{n}}$ and $\hat{q}_{n}(1-\tau)$ will have asymptotically correct coverage probability under regularity conditions.

\begin{remark}[Comparison with DWB]
Since the introduction of DWB for time series inference, there have been 
quite a bit further extensions  in the time series literature. For example, its validity has been justified for degenerate $U$- and $V$-statistics by \cite{LeNe13} and \cite{ChSeGr14}, and for empirical processes by \cite{DoLaLeNe15}. It has also been used in several testing problems to cope with weak temporal dependence; see \cite{BuWe17}, \cite{RhSh19}, and \cite{HiMo20}  among others.  To widen the scope of applicability of DWB while preserving   its adaptiveness to irregular configuration of time series or spatial data, \cite{SeShWa15} developed  the dependent random weighting, which can be viewed as an extension of traditional random weighting to  both time series and random fields. The theoretical justification in \cite{SeShWa15} is restricted to the $d=1$ case. Although conceptually simple, the theory associated with SDWB is considerably more involved than DWB and our proof techniques are substantially different from the above-mentioned papers due to our focus on  its validity in the high-dimensional setting.  
\end{remark}

\section{Main results}\label{sec: Main-results}

In this section, we first derive a high-dimensional CLT for the sample mean over the rectangles in Section~\ref{sec: GA-mean-result}. Building on the high-dimensional CLT, we establish the asymptotic validity of the SDWB over the rectangles in high dimensions in Section~\ref{sec: SDWB-result}. In what follows, we  maintain the baseline assumption discussed in Section \ref{sec: sampling}. 

\subsection{High-dimensional CLT}
\label{sec: GA-mean-result}

To state the high-dimensional CLT, we shall  consider the two cases separately:  (i) the coordinates of $\bX$ are sub-exponential and (ii) have finite polynomial moments.

\subsubsection{High-dimensional CLT under sub-exponential condition}
We make the following assumption. 
\begin{assumption}\label{ass: design}
Suppose that  $p = O(n^{\alpha})$ for some $\alpha>0$. Let $\{\lambda_{1,n}\}_{n \geq 1}$ and $\{\lambda_{2,n}\}_{n \geq 1}$ be two sequences of positive numbers such that $\lambda_{1,n},
\lambda_{2,n} \to \infty$, $\lambda_{2,n} =o (\lambda_{1,n})$, and $\lambda_{1,n} = o(\lambda_{n})$.

(i) The random field $\bX$ has zero mean, i.e., $E[\bX(\bs)] = \bm{0}$ and the residual random field $\bm{\Upsilon}$ satisfies Condition (\ref{AN-error}). There exist two sequences of  positive constants $\{D_{n}\}_{n \geq 1}$ with $D_{n} \geq 1$ and $\{\delta_{n,\Upsilon}\}_{n \geq 1}$ with $\delta_{n,\Upsilon} \to 0$ such that
\begin{equation}
\begin{split}\label{Ass-moment-exp-HD}
&\max_{1 \leq j \leq p}\|X_{j}(\bs)\|_{\psi_{1}} \leq D_{n} \quad \text{and}\\
&E_{\cdot \mid \bZ}\left[\max_{1 \leq i \leq n}\max_{1 \leq j \leq p}\left|\Upsilon_{j}(\bs_{i})\right|^{q}\right] \leq \delta_{n,\Upsilon}^{q} \quad \text{with $P_{\bZ}$-probability one for some $q \in [8,\infty)$}.
\end{split}
\end{equation}

(ii) The probability density function $f$ is continuous, everywhere positive with support $\overline{R}_{0}$.

(iii) We have $\lim_{n \to \infty}n\lambda_{n}^{-d} = \kappa \in (0,\infty]$ with $\lambda_{n} \geq n^{\bar{\kappa}}$ for some $\bar{\kappa}>0$.

(iv) There exists a constant $0 < c < 1/2$ such that
\begin{align*}
&\max\left\{{1 \over \sqrt{\lambda_{2,n}}},\ D_{n}\left({\log n \over n\lambda_{n}^{-d}} +1\right)\sqrt{{\overline{\beta}_{q}\lambda_{2,n} \over \lambda_{1,n}}},  \left(\lambda_{1,n}^{d/2}\lambda_{2,n}^{d}D_{n}^{2} + \lambda_{1,n}^{d} D_{n}\right)\lambda_{n}^{-d/2},\ {D_{n}^{6}\lambda_{1,n}^{3d}\over n^{2}\lambda_{n}^{-d}} \right\}n^{c} = O(1)
\end{align*}
as $n \to \infty$, where $\overline{\beta}_{q} =\overline{\beta}_{q}(n) :=  1 + \sum_{k=1}^{\lambda_{1,n}}k^{d-1}\beta_{1}^{1-2/q}(k)$.  Further, there exists a constant $0<c'<c$ such that
\begin{align}\label{asy-neg-Upsilon}
\lambda_{n}^{d/2}n^{c'}\delta_{n,\Upsilon} = O(1).
\end{align}

(v) We have $\lim_{n \to \infty}\lambda_{n}^{d}\lambda_{1,n}^{-d}\beta(\lambda_{2,n} ;\lambda_{n}^{d}) = 0$, and there exist some constants $0<\underline{c}<\overline{C}<\infty$ such that 
\begin{align}
\Sigma_{j,j}(\bm{0}) \leq \overline{C},\quad   \int_{\R^{d}} \Sigma_{j,j}(\bs)d\bs \ge \underline{c} \quad \text{and} \quad \int_{\mathbb{R}^{d}}|\Sigma_{j,j}(\bm{s})|d\bm{s} \leq \overline{C} \quad \text{for all $1 \leq j \leq p$,} \label{HDGA_cov_ass} 
\end{align}
where $\Sigma(\bs) = (\Sigma_{j,k}(\bs))_{1\leq j,k \leq p}  = \Cov(\bX(\bs), \bX(\bm{0}))$.
\end{assumption}

A discussion about the above assumptions is warranted. The sequences $\{\lambda_{1,n}\}$ and $\{\lambda_{2,n}\}$ will be used in the large-block-small-block argument, which is commonly used in proving CLTs for sums of mixing random variables; see \cite{La03b}.  Specifically, $\lambda_{1,n}$ corresponds to the side length of large blocks, while $\lambda_{2,n}$ corresponds to the side length of small blocks. 
The first part of Condition (\ref{Ass-moment-exp-HD}) requires the coordinates of $\bX(\bs)$ to be (uniformly) sub-exponential, while the second part of Condition (\ref{Ass-moment-exp-HD}) partially ensures the asymptotic negligibility of the residual random field, along with 
the condition (\ref{AN-error}). 
Condition (ii) is concerned with the distribution of  irregularly spaced sampling sites and allows a nonuniform density across the sampling region.  Condition (iii) implies that our sampling design allows the pure increasing domain case ($\lim_{n \to \infty}n\lambda_{n}^{-d} = \kappa \in (0,\infty)$) and the mixed increasing domain case ($\lim_{n \to \infty}n\lambda_{n}^{-d} = \infty$). 
Condition (\ref{asy-neg-Upsilon}) is used to guarantee the asymptotic negligibility of $\bm{\Upsilon}(\bs)$ for the asymptotic validity of the SDWB.  
For random fields to be discussed in Section \ref{sec: examples}, $\delta_{n,\Upsilon}$ decays exponentially fast as $n \to \infty$, so that Condition (\ref{asy-neg-Upsilon}) is satisfied. 
Condition (v) is a technical condition on the covariance function of $\bX$. Condition (\ref{HDGA_cov_ass}) is used to guarantee that the (conditional) coordinatewise variances of the normalized sample mean $\sqrt{\lambda_n^d} \overline{\bX}_n$ are bounded away from zero almost surely.

Let us briefly compare our conditions with Condition (S.5) in \cite{La03a}, who established CLTs for weighted sums of spatial data in the univariate case (i.e., $p=1$). Condition (iv) corresponds to Lahiri's Condition (S.5) Part (i), and the condition $\lim_{n \to \infty}\lambda_{n}^{d}\lambda_{1,n}^{-d}\beta(\lambda_{2,n} ;\lambda_{n}^{d}) = 0$ corresponds to the $\beta$-mixing version of Lahiri's Condition (S.5) Part (iii).  In particular, $\sqrt{\overline{\beta}_{q}\lambda_{2,n}\lambda_{1,n}^{-1}} = O(n^{-c})$ and $\lambda_{1,n}^{d}D_{n}/\lambda_{n}^{-d/2} = O(n^{-c})$ imply Lahiri's Condition (S.5) Part (i). Although our conditions are slightly more restrictive than his, they enable us to obtain error bounds for the the high-dimensional CLT for the sum of large blocks of high-dimensional spatial data with the dimension growing polynomially fast in the sample size.

We are ready to state the main results. 
Let $\mathcal{A} = \{\prod_{j=1}^{p}[a_{j},b_{j}]: -\infty \leq a_{j}\leq b_{j} \leq \infty, 1\leq j \leq p\}$ denote the collection of closed rectangles in $\R^{p}$. For $\bm{\ell} = (\ell_{1}, \dots, \ell_{d})' \in \mathbb{Z}^{d}$, we let $\Gamma_{n}(\bm{\ell};\bm{0}) = (\bm{\ell} + (0,1]^{d})\lambda_{3,n}$ with $\lambda_{3,n} = \lambda_{1,n} + \lambda_{2,n}$, and define the following hypercubes, 
\[
\Gamma_{n}(\bm{\ell};\bm{1}) = \prod_{j=1}^{d}(\ell_{j}\lambda_{3,n}, \ell_{j}\lambda_{3,n} + \lambda_{1,n}].
\]
Intuitively, $\Gamma_{n}(\bm{\ell};\bm{0})$ is a complete block of indices in $\R^d$ that contains a large block $\Gamma_{n}(\bm{\ell};\bm{1})$ and many small blocks $\Gamma_{n}(\bm{\ell};\bm{0}) \setminus \Gamma_{n}(\bm{\ell};\bm{1})$.  
Let $L_{n} = \{\bm{\ell} \in \mathbb{Z}^{d}: \Gamma_{n}(\bm{\ell},\bm{0}) \cap R_{n} \neq \emptyset\}$ denote the index set of all hypercubes $ \Gamma_{n}(\bm{\ell},\bm{0})$ that are contained in or intersects with the boundary of $R_{n}$.
Define
\[
S_{n}(\bm{\ell};\bm{1}) = \sum_{i:\bs_{i}\in \Gamma_{n}(\bm{\ell};\bm{1}) \cap R_{n}}\bX(\bs_{i}).
\]
If $[\![\{i: \bs_{i} \in \Gamma_{n}(\bm{\ell};\bm{1}) \cap R_{n}\}]\!] = 0$, we set $S_{n}(\bm{\ell};\bm{1}) = \bm{0}$. 

\begin{theorem}[High-dimensional CLT]\label{Thm: GA_hyperrec}
Under Assumption \ref{ass: design}, the following result holds $P_{\bZ}$-almost surely: 
\begin{align}
\sup_{A \in \mathcal{A}}\left|P_{\cdot \mid \bZ}\left(\sqrt{\lambda_{n}^{d}}\overline{\bY}_{n} \in A\right) - P_{\cdot \mid \bZ}\left(\bm{V}_{n} \in A\right)\right| \leq C\left({\lambda_{n} \over \lambda_{1,n}}\right)^{d}\beta(\lambda_{2,n}; \lambda_{n}^{d}) + O(n^{-\left({c' \over 6} \wedge \zeta \right)}), \label{GA-HR}
\end{align}
where $C$ is a positive constant that does not depend on $n$, and $\bm{V}_{n} = (V_{1,n},\dots, V_{p,n})'$ is a centered Gaussian random vector under $P_{\cdot \mid \bZ}$ with (conditional) covariance 
\begin{equation}
E_{\cdot \mid \bZ}\left[\bm{V}_{n}\bm{V}'_{n}\right] = {1 \over n^{2}\lambda_{n}^{-d}}\sum_{\bm{\ell} \in L_{n}}E_{\cdot \mid \bZ}\left[S_{n}(\bm{\ell};\bm{1})S_{n}(\bm{\ell};\bm{1})'\right].
\label{eq: covariance matrix}
\end{equation}
\end{theorem}


The proof of Theorem \ref{Thm: GA_hyperrec} relies on an extension of the coupling technique in \cite{Yu94} to irregularly spatial data.  The proof proceeds by first approximating the sample mean by the sum of independent large blocks and then showing the high-dimensional CLT for the sum of independent large blocks. The terms $S_{n}(\bm{\ell};\bm{1})$ that appear in the representation of $E_{\cdot \mid \bZ}[\bm{V}_{n}\bm{V}'_{n}]$ is the (conditional) covariance matrix of independent couplings for the large blocks.  The first term $\left(\lambda_{n}/\lambda_{1,n}\right)^{d}\beta(\lambda_{2,n}; \lambda_{n}^{d})$ in the error bound (\ref{GA-HR}) comes from the blocking argument and reflects a bound on the contribution from small blocks, while the second term corresponds to the error bound of the high-dimensional CLT for the sum of independent large blocks. 

The covariance matrix (\ref{eq: covariance matrix}) of the (conditionally) Gaussian vector $\bm{V}_{n}$ depends on the block construction. While the result of Theorem \ref{Thm: GA_hyperrec} is sufficient to establish the asymptotic validity of the SDWB, it is possible to replace the approximating Gaussian vector by that with covariance matrix independent of the block construction, as shown in the following corollary. 

\begin{corollary}\label{cor41}
If, in addition to Assumption \ref{ass: design}, (i) $(n\lambda_n^{-d})^{-1} = \kappa^{-1} + o((\log n)^{-2})$, where $\kappa^{-1} = 0$ if $\kappa = \infty$; (ii) the density function $f$ is Lipschitz continuous inside $R_{0}$; and (iii) $
\int_{\|\bm{s}\| \geq \lambda_{2,n}}|\Sigma_{j,j}(\bm{s})|d\bm{s} = O(n^{-c'/2})$ uniformly over $1 \le j \le p$, then we have
with $P_{\bZ}$-probability one, 
\begin{align*}
\sup_{A \in \mathcal{A}}\left|P_{\cdot \mid \bZ}\left(\sqrt{\lambda_{n}^{d}}\overline{\bY}_{n} \in A\right) - P\left(\breve{\bm{V}}_{n} \in A\right)\right|= o(1),
\end{align*}
where $\breve{\bm{V}}_{n} = (\breve{V}_{1,n},\dots,\breve{V}_{p,n})'$ is a centered Gaussian random vector with covariance 
\[
\E[\breve{\bm{V}}_{n}\breve{\bm{V}}_{n}'] = \left ( \int_{\R^{d}} \Sigma_{j,k}(\bm{x})d\bm{x}\int_{R_{0}} f^{2}(\bm{z})d\bm{z} + \kappa^{-1}\Sigma_{j,k}(\bm{0}) \right )_{1 \le j,k \le p}.
\]
\end{corollary}

Corollary \ref{cor41} is a high-dimensional extension of Theorem 3.1 in \cite{La03a} when $\omega_{n}(s) = 1$ in his notation. 
The conclusion of Corollary \ref{cor41} follows from Theorem \ref{Thm: GA_hyperrec} combined with the  Gaussian comparison inequality. Indeed, under the assumption of Corollary \ref{cor41}, we will show that $\max_{1 \leq j,k \leq p}|E_{\cdot \mid \bZ}[V_{j,n}V_{k,n}] - E[\breve{V}_{j,n}\breve{V}_{k,n}]| = o((\log n)^{-2})$ with $P_{\bZ}$-probability one, which implies the conclusion of Corollary \ref{cor41} via the Gaussian comparison.

\subsubsection{High-dimensional CLT under polynomial moment condition}

Next, we consider the case where $\bX$ has finite polynomial moments. We make the following assumption.

\begin{assumption}\label{ass: design2}
Suppose that  $p = O(n^{\alpha})$ for some $\alpha>0$. Let $\{\lambda_{1,n}\}_{n \geq 1}$ and $\{\lambda_{2,n}\}_{n \geq 1}$ be two sequences of positive numbers such that $\lambda_{1,n},\lambda_{2,n} \to \infty$, $\lambda_{2,n} =o (\lambda_{1,n})$, and $\lambda_{1,n} = o(\lambda_{n})$.
We replace Condition (i) in Assumption \ref{ass: design} with the following Condition (i'). 

(i') The random field $\bX$ has zero mean $E[\bX(\bs)] = \bm{0}$ and the residual random field $\bm{\Upsilon}$ satisfies (\ref{AN-error}). There exist two sequences of positive constants $\{M_{n}\}_{n \geq 1}$ with $M_{n} \geq 1$ and $\{\phi_{n,\Upsilon}\}_{n \geq 1}$ with $\phi_{n,\Upsilon} \to 0$ such that 
\begin{align}
\max_{1 \leq j \leq p}E\left[|X_{j}(\bs)|^{2+k}\right] &\leq M_{n}^{k},\ k=1,2,  \label{moment-ass1}\\
E\left[\max_{1\leq j \leq p}|X_{j}(\bs)|^{q}\right] &\leq 2M_{n}^{q},\ E_{\cdot \mid \bZ}\left[\max_{1 \leq i \leq n}\max_{1\leq j \leq p}|\Upsilon_{j}(\bs_{i})|^{q}\right] \leq \phi_{n,\Upsilon}^{q}\quad \text{$P_{\bZ}$-a.s} \label{moment-ass2}
\end{align}
for some $q \in [8,\infty)$.

In addition, we maintain Conditions (ii), (iii), and (v) in Assumption \ref{ass: design}, but replace Condition (iv) in Assumption \ref{ass: design} with the following Condition (iv'). 

(iv') There exists a constant $0 < c< 1/2$ such that
\begin{equation}
\label{high-level-condi-HDCLT}
\begin{split}
&\max \Bigg \{  {1 \over \sqrt{\lambda_{2,n}}}, M_{n}\! \left({\log n \over n\lambda_{n}^{-d}} + 1\right)\sqrt{{\overline{\beta}_{q}\lambda_{2,n} \over \lambda_{1,n}}}, (\lambda_{2,n}^{3d/2}+\lambda_{1,n}^{d})\lambda_{n}^{-d/2}M_{n}^{2}, {M_{n}^{2}\lambda_{1,n}^{3d} \over (n^{2}\lambda_{n}^{-d})^{1-2/q}} \Bigg\}n^{c} =O(1) 
\end{split}
\end{equation}
as $n \to \infty$, where $\overline{\beta}_{q} = 1 + \sum_{k=1}^{\lambda_{1,n}}k^{d-1}\beta_{1}^{1-2/q}(k)$. Further, there exists a constant $0<c'<c$ such that
\begin{align}\label{asy-neg-Upsilon2}
\lambda_{n}^{d/2}n^{c'}\phi_{n,\Upsilon} = O(1).
\end{align}
\end{assumption}

Condition (i') allows the coordinates of $\bX$ to have polynomial moments, so it is weaker than Condition (i). As a trade-off,  Condition (iv') is more restrictive than Condition (iv) in Assumption  \ref{ass: design}, as we impose more restrictions on the weak dependence structure of the random field.


\begin{theorem}[High-dimensional CLT]\label{Thm: GA_hyperrec-m}
Under Assumption \ref{ass: design2}, 
the high-dimensional CLT (\ref{GA-HR}) continues to hold $P_{\bZ}$-almost surely.
\end{theorem}


Differences in the proofs of Theorems \ref{Thm: GA_hyperrec} and \ref{Thm: GA_hyperrec-m} arise when we approximate the normalized sample mean by the sum of independent  large blocks using the large-block and small-block argument. Indeed, the conditions on the moments of random fields have a direct impact when we apply maximal inequalities to control the order  of small blocks.

\subsection{Asymptotic validity of the SDWB}
\label{sec: SDWB-result}

In this section, we establish the asymptotic validity of  SDWB in high dimensions. Recall that, given the observations $\{\bY(\bs_{i})\}_{i=1}^{n}$, the SDWB pseudo-observations are given by
\[
\bY^{\ast}(\bs_{i}) = \overline{\bY}_{n} + (\bY(\bs_{i}) - \overline{\bY}_{n})W(\bs_{i}),\ i = 1,\dots,n,
\]
where
$\{W(\bs_{i})\}_{i=1}^{n}$ are $n$ discrete samples from a real-valued stationary Gaussian random field $W = \{W(\bs): \bs \in \R^{d}\}$ independent of $\bY$ and $\{ \bZ_i \}_{i \ge 1}$. We make the following assumption on $W$. 
\begin{assumption}\label{ass: SDWB}
The random field $W$ is a stationary 
Gaussian random field with mean zero and covariance function $\Cov(W(\bs_{1}), W(\bs_{2})) = a(\|\bs_{2} - \bs_{1}\|/b_{n})$, where $a(\cdot): \R \to [0,1]$ is a continuous kernel function and $b_n$ is a bandwidth parameter.
The kernel function satisfies that $a(0) = 1$ and $a(x) = 0$ for $|x| \geq 1$. There exist positive constants $c_W$ and $L_W$ such that $|1 - a(x)| \leq L_W|x|$ for $|x| \leq c_W$. Further, with $\mathrm{i} = \sqrt{-1}$,
\begin{align}\label{psd_condition}
K_{a}(x) = \int_{\R}a(u)e^{-\mathrm{i}ux}du \geq 0 \quad \text{for all  $x \in \R$}. 
\end{align}

\end{assumption}

 Condition (\ref{psd_condition}) guarantees the positive semi-definiteness of the covariance matrix of $\{ W(\bs_i)\}_{i=1}^n$. Assumption \ref{ass: SDWB} is satisfied by many commonly used kernel functions in the literature of spectral density estimation, in particular,  Bartlett and Parzen kernels. See  \cite{Pr81} and \cite{An91} for details.

\begin{remark}[Comments on the auxiliary random field $W$]
The covariance function of the Gaussian random field $W$ defined in Assumption \ref{ass: SDWB} implies that the random field $W$ is isotropic. We assume this condition for technical convenience, and it is not difficult to see from the proof that the conclusion of the following theorem holds for the following class of (possibly) non-isotropic covariance functions.  Consider a function $\breve{a}: \mathbb{R}^{d} \to [0,1]$ with $\breve{a}(\bm{0}) = 1$, $\breve{a}(\bm{x}) = 0$ for $\|\bm{x}\| \geq 1$, and assume that there exist positive constants $c_W$ and $L_W$ such that $|1-\breve{a}(\bm{x})| \leq L_W\|\bm{x}\|$ for $\|\bm{x}\| \leq c_W$. Further, assume that the function $\mathfrak{a}: \mathbb{R}^{d} \times \mathbb{R}^{d} \to [0,1]$ defined by $\mathfrak{a}(\bm{x}_{1}, \bm{x}_{2}) = \breve{a}(\bm{x}_{1}-\bm{x}_{2})$ is positive semidefinite. For example, these conditions are satisfied for product kernels of the form $\breve{a}(\bm{x}) = \prod_{j=1}^{d}a_{j}(\sqrt{d}|x_{j}|)$ where $a_{j}$ are one-dimensional kernel functions that satisfy Assumption \ref{ass: SDWB}. In addition, the Gaussian random field assumption can also be relaxed but at the expense of additional technical complications; see Example 4.1 of \cite{Sh10} for an example of non-Gaussian distribution for external random variables
$\{ W(\bs_i)\}_{i=1}^n$.
\end{remark}

\begin{theorem}[Asymptotic validity of SDWB in high dimensions]
\label{DWBvalidity: HR}
Suppose that Assumptions \ref{ass: design} (or \ref{ass: design2}) and \ref{ass: SDWB} hold with $b_n\sim \lambda_{2,n}$. In addition, assume that 
\begin{equation}
\sum_{n=1}^{\infty}n^{c'}(\log n)^{2}\lambda_{n}^{d}\lambda_{1,n}^{-d}\max_{1 \leq j,k \leq p}\int_{\|\bs\|>\sqrt{\lambda_{2,n}}}|\Sigma_{j,k}(\bs)|d\bs < \infty. \label{cov_matrix_tail}
\end{equation}
Then the following result holds $P_{\bZ}$-almost surely: with $P_{\cdot \mid \bZ}$-probability at least $1 - O(n^{-\left({c' \over 6} \wedge \zeta \right)}) - C(\lambda_{n}/\lambda_{1,n})^{d}\beta(\lambda_{2,n};\lambda_{n}^{d})$,
\begin{align}
\sup_{A \in \mathcal{A}}\left|P_{\cdot|\bY,\bZ}\left( \sqrt{\lambda_{n}^{d}}(\overline{\bY}^{\ast}_{n} - \overline{\bY}_{n} )  \in A \right) - P_{\cdot \mid \bZ} \left(\bm{V}_{n} \in A\right)\right| = O(n^{-c'/6}) \label{DWB-HR}
\end{align}
where $C$ is a positive constant that does not depend on $n$ and $p$.
\end{theorem}

From the definition of $\overline{\bY}^{\ast}_{n}$, $\sqrt{\lambda_{n}^{d}}(\overline{\bY}^{\ast}_{n} - \overline{\bY}_{n})$ can be decomposed into the following three terms: 
\begin{align*}
\sqrt{\lambda_{n}^{d}}(\overline{\bY}^{\ast}_{n} - \overline{\bY}_{n}) &= n^{-1}\lambda_{n}^{d/2}\left(\sum_{i=1}^{n}W(\bs_{i})\bX(\bs_{i}) - \sum_{i=1}^{n}W(\bs_{i})\overline{\bY}_{n} + \sum_{i=1}^{n}W(\bs_{i})\bm{\Upsilon}(\bs_{i})\right)\\
& =: n^{-1}\lambda_{n}^{d/2}\left(U_{1,n} + U_{2,n} + U_{3,n}\right). 
\end{align*}
The proof of Theorem  \ref{DWBvalidity: HR}  proceeds with (i) showing asymptotic negligibility of $n^{-1}\lambda_{n}^{d/2}(U_{2,n} + U_{3,n})$ and (ii) approximating  $n^{-1}\lambda_{n}^{d/2}U_{1,n}$ by $\bm{V}_{n}$.
%

\begin{remark}[Comparisons with block-based subsampling and resampling methods]
There have been substantial efforts in extending subsampling \citep{PoRoWo99} and  block-based bootstrap (BBB) methods \citep{La03b} from time series (i.e., $d=1$) to random fields (i.e., $d\ge 2$). For example,  \cite{PoPaRo98} proposed a subsampling method for irregularly spaced spatial data generated by a homogeneous Poisson process. \cite{PoPaRo99} proposed a version of the spatial block bootstrap under the same framework. \cite{LaZh06} developed a grid-based block bootstrap  for irregularly spaced spatial data with nonuniform stochastic sampling designs. 
While subsampling and BBB methods are able to capture spatial dependence nonparametrically, their implementation can be inconvenient when applied to irregularly spaced spatial data, as both require partitioning the sampling region into complete and incomplete blocks, and the implementation details can be highly dependent on  spatial configuration of sampling region. By contrast, the implementation of SDWB only requires the generation of an auxiliary random field $W(\cdot)$ and irregularity of  sampling sites brings no additional difficulty.

On the theory front, 
 \cite{Sh10} showed that DWB and BBB (especially TBB) are often comparable in terms of theoretical properties in the time series setting with a proper choice of kernel function and bandwidth, but all theoretical results developed so far for BBB seem exclusively for low-dimensional time series/random fields. To the best of our knowledge, our work is the first attempt in the literature to show the validity of a
 bootstrap method for high-dimensional spatial data. 
 
\end{remark}


\section{Examples}
\label{sec: examples}

In this section, we discuss some examples of random fields to which our theoretical results can be applied. To this end, we consider multivariate L\'evy-driven MA random fields and discuss their dependence structure.  L\'evy-driven MA random fields include many Gaussian and non-Gaussian random fields and constitute a flexible class of models for multivariate spatial data. 
We show that a broad class of multivariate L\'evy-driven MA random fields, which include CARMA random fields as special cases, satisfies our assumptions.

We first introduce a multivariate L\'{e}vy-driven MA random field; see \cite{Be96} and \cite{Sa99} for standard references on L\'evy processes and \cite{RaRo89} for details on the theory of infinitely divisible measures and fields.
\label{ex: Levy-MARF}
Let $\bm{L}(\cdot) = (L_{1}(\cdot),\dots,L_{p}(\cdot))'$ be an $\R^{p}$-valued L\'evy random measure on the Borel subsets $\mathcal{B}(\R^{d})$ of $\R^{d}$ that is infinitely divisible in the sense that 
\begin{itemize}
\item[(a)] If $A$ and $B$ are disjoint Borel subsets of $\R^{d}$, then $\bm{L}(A)$ and $\bm{L}(B)$ are independent.
\item[(b)] For every Borel subset $A$ of $\R^{d}$ with finite Lebesgue measure $|A|$, 
\begin{align}\label{Levy-exp}
E[\exp(\mathrm{i}\bm{\theta}'\bm{L}(A))] &= \exp(|A|\psi(\bm{\theta})),\ \bm{\theta} \in \R^{p},
\end{align}
where $\mathrm{i} = \sqrt{-1}$ and $\psi$ is the logarithm of the characteristic function of an $\R^{p}$-valued infinitely divisible distribution, which is given by 
\begin{align*}
\psi(\bm{\theta}) &= \mathrm{i}\langle \bm{\theta}, \bm{\gamma}_{0}\rangle - {1 \over 2}\bm{\theta}'\Sigma_{0}\bm{\theta} + \int_{\R^{p}}\left\{e^{\mathrm{i}\langle \bm{\theta}, \bm{x} \rangle }-1-\mathrm{i}\langle \bm{\theta}, \bm{x} \rangle I(\|\bm{x}\| \leq 1)\right\}\nu_{0}(d\bm{x}),
\end{align*}
where $\bm{\gamma}_{0} = (\gamma_{0,1},\dots, \gamma_{0,p})' \in \R^{p}$, $\Sigma_{0} = (\sigma_{0,j,k})_{1 \leq j,k \leq p}$ is a $p\times p$ positive definite matrix, and $\nu_{0}$ is a L\'evy measure with $\int_{\R^{p}}\min\{1,\|\bm{x}\|^{2}\}\nu_{0}(d\bm{x})<\infty$. If $\nu_{0}(d\bm{x})$ has a Lebesgue density, i.e., $\nu_{0}(d\bm{x}) = \nu_{0}(\bm{x})d\bm{x}$, we call $\nu_{0}(\bm{x})$ the L\'evy density. The triplet $(\bm{\gamma}_{0}, \Sigma_{0}, \nu_{0})$ is called the L\'evy characteristic of $\bm{L}$ and uniquely determines the distribution of $\bm{L}$. 
\end{itemize}

By equation (\ref{Levy-exp}), the first and second moments of the random measure $\bm{L}$ are determined by
\[
E[L_{j}(A)] = \mu_{j}^{(\bm{L})}|A|,\ \Cov(L_{j}(A), L_{k}(A)) = \sigma_{j,k}^{(\bm{L})}|A|,\ 1 \leq j,k \leq p,
\] 
where $\mu_{j}^{(\bm{L})} = -i\partial \psi(\bm{0})/\partial \theta_{j}$ and $\sigma_{j,k}^{(\bm{L})} = -\partial^{2}\psi(\bm{0})/\partial \theta_{j}\partial \theta_{k}$. 

The following are a couple of examples of L\'evy random measures. 
\begin{itemize}
\item If $\psi(\bm{\theta}) = -\bm{\theta}'\Sigma_{0} \bm{\theta}/2$ with a $p\times p$ positive semi-definite matrix $\Sigma_{0}$, then $\bm{L}$ is a Gaussian random measure. 

\item If $\psi(\bm{\theta}) = \lambda \int_{\R^{p}}(\exp(\mathrm{i}\langle \bm{\theta}, \bm{x}\rangle) - 1)F(d\bm{x})$, where $\lambda>0$ and $F$ is a probability distribution function with no jump at the origin, then $\bm{L}$ is a compound Poisson random measure with intensity $\lambda$ and jump size distribution $F$. More specifically, 
\begin{align*}
\bm{L}(A) &= \sum_{i=1}^{\infty}\bm{J}_{i}1_{\bm{x}_{i}}(A),\ A \in \mathcal{B}(\R^{d}),
\end{align*}
where $\bm{x}_{i}$ denotes the location of the $i$th unit point mass of a Poisson random measure on $\R^{d}$ with intensity $\lambda>0$ and $\{\bm{J}_{i}\}$ is a sequence of  i.i.d. random vectors in $\R^{p}$ with distribution function $F$ independent of $\{\bm{x}_{i}\}$. 
\end{itemize}

Let $\bm{g} = (g_{j,k})_{1 \leq j,k \leq p}$ be a measurable function on $\R^{d}$ with $g_{j,k} \in L^{1}(\R^{d}) \cap L^{\infty}(\R^{d})$.  A multivariate L\'evy-driven MA random field with kernel $\bm{g}$ driven by a L\'evy random measure $\bm{L}$ is defined by
\begin{align}\label{multi-CARMA}
\bY(\bs) &= (Y_{1}(\bs),\dots,Y_{p}(\bs))' = \int_{\R^{d}}\bm{g}(\bs - \bm{u})\bm{L}(d\bm{u}),\ \bs \in \R^{d}. 
\end{align}
Define $\bm{\mu}_{\bm{L}} = (\mu_{1}^{(\bm{L})},\dots \mu_{p}^{(\bm{L})})'$ and $\Sigma_{\bm{L}} = (\sigma_{j,k}^{(\bm{L})})_{1 \leq j,k \leq p}$. The first and second moments of $\bY(\bm{s})$ satisfy
\begin{align*}
E[\bY(\bm{s})] &= \bm{\mu}_{\bm{L}}\int_{\mathbb{R}^{d}}\bm{g}(\bm{u})d\bm{u},\ \Cov(\bY(\bm{s}), \bY(\bm{v})) = \int_{\mathbb{R}^{d}}\bm{g}(\bm{s}-\bm{u})\Sigma_{\bm{L}}\bm{g}'(\bm{v}-\bm{u})d\bm{u}.
\end{align*}
We refer to \cite{BrMa17} and \cite{MaSt07} for more details on the computation of moments of L\'evy-driven MA processes.

Before discussing theoretical results, we look at some examples of random fields defined by (\ref{multi-CARMA}).  When $p=1$, the random fields defined by (\ref{multi-CARMA}) include  CARMA random fields  \citep{BrMa17}.  For example, if the L\'evy random measure of a CARMA random field is compound Poisson, then the resulting random field is called a compound Poisson-driven CARMA random field. In particular, when
\[
g(\bs) = (1-\varsigma)\exp(\lambda_{1}\|\bs\|) + \varsigma\exp(\lambda_{2}\|\bs\|),
\]
where $\varsigma$ is a parameter that satisfies 
\[
-{\lambda_{2}^{2} - \xi^{2}\lambda_{1} \over \lambda_{1}^{2} - \xi^{2}\lambda_{2}} = {\varsigma \over 1-\varsigma},\ \lambda_{1}<\lambda_{2}<0,\ \xi \leq 0,
\]
then the random field (\ref{multi-CARMA}) is called a CARMA($2,1$) random field. This random field includes normalized CAR($1$) (when $\varsigma=0$) and CAR($2$) (when $\varsigma = -\lambda_{1}/(\lambda_{2} - \lambda_{1})$) as special cases. See \cite{BrMa17} for more details. We also refer to \cite{MaYu20} for a multivariate extension of univariate CARMA random fields.  

\begin{remark}[Connections to  Mat\'ern covariance functions]
In spatial statistics, Gaussian random fields with the following Mat\'ern covariance functions play an important role \cite[cf.][]{Ma86,St99,GuGn06}:
\[
M(\bm{s}; \nu, a, \sigma) = \sigma^{2}(\|a\bm{s}\|)^{\nu}K_{\nu}(\|a\bm{s}\|),\ \nu>0, a>0, \sigma>0,
\]
where $K_{\nu}$ denotes the modified Bessel function of the second kind of order $\nu$ (we call $\nu$ the index of Mat\'ern covariance function). \cite{BrMa17} showed that in the univariate case, when the kernel function is $g(\bm{s}) = (\|a\bm{s}\|)^{\nu}K_{\nu}(\|a\bm{s}\|)$, which they call a Mat\'ern kernel with index $\nu$,  then the Levy-driven MA random field has a Mat\'ern covariance function with index $d/2 + \nu$. For example, a normalized CAR(1) random field has a Mat\'ern covariance function since its kernel function is given by $g(\bm{s}) =  \exp(-\|\lambda_{1}\bm{s}\|) = \sqrt{(2/\pi)}\|\lambda_{1}\bm{s}\|^{1/2}K_{1/2}(\|\lambda_{1}\bm{s}\|)$ for some $\lambda_{1}<0$.

It will turn out that a class of Gaussian (and non-Gaussian) random fields with Mat\'ern covariance functions satisfies our assumptions on $\bY$. For example, a class of multivariate L\'evy-driven MA fields with diagonal kernels of the form $\bm{g}(\bm{s}) = \text{diag}(g_{j,j}(\bm{s}))$ where $g_{j,j}(\bm{s}) = (\|a_{j}\bm{s}\|)^{\nu_{j}}K_{\nu_{j}}(\|a_{j}\bm{s}\|)$ includes Gaussian and non-Gaussian random fields of which each component has a Mat\'ern covariance function. When the kernel function $\bm{g}(\bm{s}) = (g_{j,k}(\bm{s}))$ is non-diagonal, each component of $\bY(\bm{s})$ is a sum of Gaussian or non-Gaussian Mat\'ern random fields if $g_{j,k}$ satisfy the assumptions in Proposition \ref{m-approx-CARMA} below. Since the Mat\'ern kernel decays exponentially fast as $\|\bm{s}\| \to \infty$, our assumptions (in Proposition \ref{m-approx-CARMA}) cover a wide class of (multivariate) Mat\'ern families.


\end{remark}

In general, if $\bm{g}$ depends only on $\|\bm{s}\|$, i.e., $\bm{g}(\bs) = \bm{g}(\|\bs\|)$, then $\bY$ is a strictly stationary isotropic random field whose characteristic function is given by
\begin{equation}
\begin{split}
E[e^{\mathrm{i}\langle \bm{t}, \bY(\bs) \rangle}]
&= \exp\left\{\mathrm{i} \left \langle \bm{t}, \int_{\R^{d}}\bm{g}(\|\bm{v}\|)\bm{\gamma}_{0}d\bm{v} \right \rangle - {1 \over 2}\bm{t}'\left(\int_{\R^{d}}\bm{g}(\|\bm{v}\|)\Sigma_{0}\bm{g}'(\|\bm{v}\|)d\bm{v}\right)\bm{t} \right.  \\ 
&\left. \quad +  \int_{\R^{d}}\left(\int_{\R^{p}}(e^{\mathrm{i}\langle \bm{t}, \bm{g}(\|\bm{v}\|)\bm{x}\rangle } - 1 - \mathrm{i}\langle \bm{t}, \bm{g}(\| \bm{v}\|)\bm{x} \rangle I(\|\bm{g}(\|\bm{v}\|)\bm{x}\| \leq 1))\nu_{0}(d\bm{x})\right)d\bm{v}\right\}. \label{CF-CARMA-RF}
\end{split}
\end{equation}
This implies that the law of $\bY(\bs)$ is infinitely divisible. The second moment of $\bY(\bm{s})$ satisfies
\begin{align*}
\Sigma_{\bY}(\bm{s}) &=\Cov(\bY(\bm{s}), \bY(\bm{0})) = \int_{\mathbb{R}^{d}}\bm{g}(\|\bm{s}-\bm{u}\|)\Sigma_{\bm{L}}\bm{g}'(\|\bm{u}\|)d\bm{u}=:(\sigma^{(\bY)}_{j,k}(\bm{s}))_{1 \leq j,k \leq p}.
\end{align*}

Consider the following decomposition: 
\begin{align*}
\bY(\bs) &= \int_{\R^{d}}\bm{g}(\bs - \bm{u})\psi_{0}\left(\|\bs - \bm{u}\| : m_{n}\right)\bm{L}(d\bm{u}) + \int_{\R^{d}}\bm{g}(\bs - \bm{u})\left(1 - \psi_{0}\left(\|\bs - \bm{u}\| : m_{n}\right)\right)\bm{L}(d\bm{u})\\
&=: \bX^{(m_{n})}(\bs) + \bm{\Upsilon}^{(m_{n})}(\bs),
\end{align*}
where $m_{n}$ is a sequence of positive constants with $m_{n} \to \infty$ as $n \to \infty$ and $\psi_{0}(\cdot:c) : \R \to [0,1]$ is a truncation function defined by
\begin{align*}
\psi_{0}(x:c) = 
\begin{cases}
1 & \text{if $|x| \leq c/4$},\\
-{4 \over c}\left(x-{c \over 2}\right) & \text{if $c/4 < |x| \leq c/2$},\\
0 & \text{if $x>c/2$}.
\end{cases}
\end{align*} 
The random field $\bX^{(m_n)} = \{ \bX^{(m_{n})}(\bs)  : \bs \in \R^d \}$ is $m_{n}$-dependent (with respect to the $\ell^{2}$-norm), i.e., $\bX^{(m_n)}(\bs_{1})$ and $\bX^{(m_n)}(\bs_{2})$ are independent if $\| \bs_1-\bs_2\| \geq m_{n}$. Also, if the tail of the kernel function $\bm{g}$ decays sufficiently fast, then $\bm{\Upsilon}^{(m_n)}=\{\bm{\Upsilon}^{(m_{n})}(\bs)  : \bs \in \R^{d}\}$ is asymptotically negligible. In such cases, we can approximate $\bY$ by an $m_{n}$-dependent process and verify the negligibility condition (\ref{AN-error}) for the residual random field $\bm{\Upsilon} = \bm{\Upsilon}^{(m_{n})}$, as shown in the following proposition.

\begin{proposition}\label{m-approx-CARMA}
Consider a multivariate L\'evy-driven MA random field $\bY$ defined by (\ref{multi-CARMA}).

(i)  Let $\alpha$, $\xi_{0}$, $r_{0}$ and $\overline{M}$ be finite positive constants independent of $n$. Assume that $g_{j,k}(\bs) = \xi_{j,k}e^{-r_{j,k}\|\bs\|}$ where $\xi_{j,j} \neq 0$, $\max_{1 \leq j \leq p}[\![\{ k  : \xi_{j,k} \neq 0,k \neq j\}]\!] \leq \overline{M}$, $|\xi_{j,k}| \leq \xi_{0}$, and $r_{0} \leq r_{j,k} \leq \xi_{0}$ for all $1 \leq j,k \leq p$.
Additionally, assume that $\max_{1 \leq j \leq p}E[\left|L_{j}([0,1]^{d})\right|^{q}] \leq \overline{M}$ for some even integer $q \geq 8$ and $p = O(n^{\alpha})$. Set $m_{n} \geq \delta \log n$ with $\delta > {2(2\alpha+q+2) \over r_{0}q}$. Then, there exists $\zeta>0$ such that Condition (\ref{AN-error}) holds with $\bm{\Upsilon} = \bm{\Upsilon}^{(m_n)}$. 

(ii) Let $0<\underline{M}<\overline{M}<\infty$ and $-1< M_{\beta}<1$ be constants independent of $n$. Further, suppose that $\min_{1 \leq j \leq p}\sigma^{(\bY)}_{j,j}(\bm{0})\geq \underline{M}, \min_{1 \leq j \leq p}\int_{\mathbb{R}^{d}}\sigma^{(\bY)}_{j,j}(\bm{u})d\bm{u} \geq \underline{M}$ and
\begin{itemize}
\item[(a)] the random measure $\bm{L}(\cdot)$ is Gaussian with triplet $(\bm{\gamma}_{0}, \Sigma_{0} , 0)$ such that $\underline{M} \leq \sigma_{0,j,j} \leq \overline{M}$ and $|\gamma_{0,j}| \leq \overline{M}$ for all $1 \leq j \leq p$, or
\item[(b)] the random measure $\bm{L}(\cdot)$ is non-Gaussian with triplet $(\bm{\gamma}_{0}, 0 , \nu_{0})$ with the marginal L\'evy density $\nu_{0,j}(x)$ of $L_{j}(\cdot)$ given by
\begin{align}
\nu_{0,j}(x) &= C_{j}|x|^{-1-\beta_{j}}e^{-\overline{c}_{j}|x|^{\alpha_{0,j}}}1_{\R\backslash \{0\}}(x),\ j=1,\dots,p, \label{LS-Levy} 
\end{align}
where $|\gamma_{0,j}|\leq \overline{M}$, $1 \leq \alpha_{0,j} \leq \overline{M}$, $-1 \leq \beta_{j} \leq M_{\beta}$, and $\underline{M} \leq \overline{c}_{j}, \Var(L_{j}([0,1]^{d})) \leq \overline{M}$ for all $1 \leq j \leq p$.
\end{itemize}
Then $\bY$ satisfies Assumption \ref{ass: design2} and Condition (\ref{cov_matrix_tail}) with $p = O(n^{\alpha})$, $\lambda_{1,n} = \lambda_{n}^{c_{0}/2}$, $\lambda_{2,n} = \lambda_{1,n}^{c_{1}}$, $n^{c_{2}} \lesssim \lambda_{n}^{d} \lesssim n$, $m_{n} = \lambda_{2,n}^{1/(\theta d)}$, $M_{n} = n^{M_{0}}$ and $\phi_{n,\Upsilon} = e^{-M_{0,1}\lambda_{2,n}^{1/(\theta d)}}$ where $M_{0}$, $M_{0,1}$, $\theta$, $c_{0}$, $c_{1}$ and $c_{2}$ are some positive constants independent of $n$ such that $\theta\geq 3$, $c_{0} \in (0, {2 \over 3}(1-{2 \over q}))$, $c_{1} \in (0, {\theta \over 1+ \theta})$ and $c_{2} \in (0,1)$. 
\end{proposition}

The condition on $\sigma_{j,j}^{(\bY)}$ in Proposition \ref{m-approx-CARMA} (ii) is typically satisfied when $\bm{g}$ is diagonal. Proposition \ref{m-approx-CARMA} implies that the approximation error $\sum_{i=1}^{n}\bm{\Upsilon}^{(m_{n})}(\bs_{i})/\sqrt{n^2\lambda_n^{-d}}$ is asymptotically negligible for both the high-dimensional CLT and the asymptotic validity of SDWB, implying that the conclusions of Theorems \ref{Thm: GA_hyperrec-m} and \ref{DWBvalidity: HR} hold for the CARMA-type random field $\bY$. More precisely, a L\'evy-driven MA random field with the kernel function $\bm{g}$ in Proposition \ref{m-approx-CARMA} satisfies Condition (\ref{moment-ass2}) with $\phi_{n,\Upsilon} \lesssim p^{2/q}m_{n}^{(d-1)/q}e^{-r_{0}m_{n}/4}$. See Appendix \ref{Appendix: CARMA} for details. 

\begin{remark}
The L\'evy density of the form (\ref{LS-Levy}) corresponds to a compound Poisson random measure if $\beta_{j} \in [-1,0)$, a Variance Gamma random measure if $\alpha_{0,j} = 1, \beta_{j} = 0$, and a tempered stable random measure if $\beta_{j} \in (0,1)$ (cf. Section 2 in \cite{Ma19} and Section 5 in \cite{KaKu20}).
\end{remark}

Proposition \ref{m-approx-CARMA} (i) also implies that a wide class of CARMA random fields are approximately ($\log n$)-dependent with respect to the $\ell^{2}$-norm. For example, high-dimensional CARMA($p_{0},q_{0}$) random fields with $p_{0}>q_{0}$  are approximately $(\log n)$-dependent random fields: Let $a(z) = z^{p_{0}} + a_{1}z^{p_{0}-1} + \cdots+a_{p_{0}} = \prod_{i=1}^{p_{0}}(z - \lambda_{i})$ be a polynomial of degree $p_{0}$ with real coefficients and distinct negative zeros $\lambda_{1},\dots,\lambda_{p_{0}}$ and let $b(z) = b_{0} + b_{1}z + \cdots + b_{q_{0}}z^{q_{0}} = \prod_{i=1}^{q_{0}}(z - \xi_{i})$ be a polynomial with real coefficients and real zeros $\xi_{1},\dots, \xi_{q_{0}}$ such that $b_{q_{0}}=1$ and $0\leq q_{0} < p_{0}$ and $\lambda_{i}^{2} \neq \xi_{j}^{2}$ for all $i$ and $j$. Define 
\[
a(z) = \prod_{i=1}^{p_{0}}(z^{2} - \lambda_{i}^{2}),\ b(z) = \prod_{i=1}^{q_{0}}(z^{2} - \xi_{i}^{2}).
\]
The kernel function of the univariate isotropic CARMA($p_{0},q_{0}$) random field with $p_{0}>q_{0}$ is given by 
\begin{align}\label{CARMA_kernel_general}
g(\bs) = \sum_{i=1}^{p_{0}}{b(\lambda_{i}) \over a'(\lambda_{i})}e^{\lambda_{i}\|\bs\|},
\end{align}
where $a'$ denotes the derivative of the polynomial $a$. See also \cite{MaYu20} for more detailed discussion on (multivariate) CARMA random fields. It is  straightforward to extend Proposition \ref{m-approx-CARMA} to the case where $g_{j,k}$ is a finite sum of kernel functions with exponential decay. Therefore, our results can be applied to a wide class of CARMA($p_{0},q_{0}$) random fields.

\begin{remark}[CARMA random fields satisfying Assumption \ref{ass: design}]\label{Ass4.1CARMA} 
We can also verify that a class of CARMA random fields satisfies Assumption \ref{ass: design}. Let $\theta \geq 3$ be a given constant and let $c_{0}, c_{1}$ and $c_{2}$ be positive constants with $c_{0} \in (0, {2 \over 3}(1-{2 \over q}))$, $c_{1} \in (0, {\theta \over 1+\theta})$ and $c_{2} \in (0,1)$. Further, let $\alpha > 0$ be any positive constant. Assume that $\bY$ is a multivariate L\'evy-driven MA random field with a kernel function $\bm{g}$ that satisfies the conditions of Proposition \ref{m-approx-CARMA} and $\bY$ is driven by (i) a Gaussian random measure or (ii) a compound Poisson random measure with bounded jumps. Then, Assumption \ref{ass: design} is satisfied with $p = O(n^{\alpha})$, $\lambda_{1,n} = \lambda_{n}^{c_{0}/2}$, $\lambda_{2,n} = \lambda_{1,n}^{c_{1}}$, $n^{c_{2}} \lesssim \lambda_{n}^{d} \lesssim n$, $D_{n} = D_{0}$ and $\delta_{n,\Upsilon} = e^{-D_{0,1}\lambda_{2,n}^{1/(\theta d)}}$ where $D_{0}$ and $D_{0,1}$ are positive constants independent of $n$. See the proof of Proposition \ref{m-approx-CARMA} and Remark \ref{bdd-supp-jump} in Appendix \ref{Appendix: CARMA}.
\end{remark}

\section{Applications}\label{sec: applications}

In this section, we discuss some applications of our results to inference for multivariate spatial and spatio-temporal data. For a spatio-temporal process $\{Y(\bm{s}, t): \bm{s} \in \mathbb{R}^{d}, t \in T\}$, suppose we obtain the observations  at a finite number of sampling sites $\{\bm{s}_{i}\}_{i=1}^{n}$ and at (possibly non-equidistant) discrete time points $\{t_{j}\}_{j=1}^{p} \subset T=\mathbb{R}$ or $\mathbb{Z}$ with $t_{1}<\cdots<t_{p}$.
In our application, we shall convert the spatio-temporal data into a form of multivariate spatial data  by stacking the time series at each location into a vector, i.e., we define
$\bY = \{\bY(\bm{s}) = (Y(\bm{s}, t_{1}),\dots, Y(\bm{s}, t_{p}))': \bm{s} \in \mathbb{R}^{d}\}$. This conversion does not result in any loss of data information but is convenient when the parameter of interest  can be estimated by spatial averaging. Also we do not require temporal stationarity or regular spacing for $\{t_j\}_{j=1}^{p}$, either or both of which are typically required in many statistical methods for spatio-temporal data analysis. More discussions about the implication of this conversion in the context of change-point analysis are offered in Remark~\ref{rem:change}.




\subsection{Joint confidence intervals for mean vectors}\label{joint-CI-mean}

Define $\Sigma^{\bm{V}_{n}} = (\Sigma_{j,k}^{\bm{V}_{n}})_{1 \leq j,k \leq p} = E_{\cdot \mid \bZ}[\bm{V}_{n}\bm{V}'_{n}]$. When we are interested in simultaneous inference on the mean vector $\bm{\mu} = E[\bY(\bs)] = (\mu_{1},\dots,\mu_{p})'$, Theorems \ref{Thm: GA_hyperrec} and \ref{Thm: GA_hyperrec-m} yield joint $100(1-\tau)\%$ confidence intervals for $\bm{\mu}$ with $\tau \in (0,1)$ of the form $\hat{C}_{1-\tau} = \prod_{j=1}^{p}\hat{C}_{j,1-\tau}$, where 
\begin{align*}
\hat{C}_{j,1-\tau} &= \left[\overline{Y}_{j,n} \pm \lambda_{n}^{-d/2}\sqrt{\Sigma_{j,j}^{\bm{V}_{n}}}q_{n}(1-\tau)\right],
\end{align*}
and  $q_{n}(1-\tau)$ is the ($1-\tau$)-quantile of $\max_{1 \leq j \leq p}\left|V_{j,n}/\sqrt{\Sigma_{j,j}^{\bm{V}_{n}}}\right|$. 
Define $\hat{C}^{\ast}_{1-\tau}$ as $\hat{C}_{1-\tau}$ by replacing $\Sigma^{\bm{V}_{n}}$ and $q_{n}(1-\tau)$ with $\hat{\Sigma}^{\bm{V}_{n}}$ and $\hat{q}_{n}(1-\tau)$, respectively.

\begin{proposition}[Asymptotic validity of SDWB joint confidence intervals in high dimensions]\label{asy_valid_SDWB_CI}
Under Assumptions \ref{ass: design} (or \ref{ass: design2}) and \ref{ass: SDWB}, the following result holds $P_{\bZ}$-almost surely: $P_{\cdot \mid \bZ}(\bm{\mu} \in \hat{C}^{\ast}_{1-\tau}) = 1-\tau + o(1)$.
\end{proposition}





\subsection{Inference on spatio-temporal data}

In this subsection we discuss some applications of our results to inference for spatio-temporal data. 

\subsubsection{Spatio-temporal compound Poisson-driven MA random fields}

To illustrate our approach to spatio-temporal data analysis, we begin with an introduction of a spatio-temporal model. Consider a multivariate compound Poisson-driven MA random field $\bY(\bs)$ defined by
\begin{align}
\bY(\bs) &= (Y_{1}(\bs),\dots,Y_{p}(\bs))' = \sum_{i=1}^{\infty}\bm{g}(\|\bs-\bm{x}_{i}\|)\bm{J}_{i}, \label{ST-CARMA} 
\end{align}
where $\bm{x}_{i}$ is the location of the $i$-th unit point mass of a Poisson random measure on $\R^{d}$, $\{\bm{J}_{i} = (J_{i,1},\dots, J_{i,p})'\}_{i \in \mathbb{Z}}$ is a sequence of i.i.d. random vectors in $\R^{p}$. Moreover, $\{\bm{J}_{i}\}_{i \geq 1}$ is independent of $\{\bm{x}_{i}\}$. Regarding $\{Y_{1}(\bm{s}),\dots, Y_{p}(\bs)\}$ as a time series observed at possibly non-equidistant time points $t_{1}<\cdots < t_{p}$, the model (\ref{ST-CARMA}) can be seen as a nonparametric spatio-temporal model $Y(\bs,t_{j})=Y_{j}(\bs)$ with $E[Y(\bs,t_{j})]=\mu(t_{j})$, $\bs \in \mathbb{R}^{d}$, $j=1,\dots,p$. 

\begin{example}[Inference on the time-varying mean of univariate spatio-temporal data]\label{Ex-ST-TVM}
We can also apply the method in Section \ref{joint-CI-mean} to inference for univariate spatio-temporal model $\{Y(\bs,t_{j}):\bs\in \mathbb{R}^{d},t_{1}<t_{2}<\dots<t_{p}\}$. For example, we can construct joint $100(1-\tau)\%$ confidence intervals for the time-varying mean $E[Y(\bs,t_{j})]=\mu(t_{j})(=E[Y_{j}(\bs)])$ of the model (\ref{ST-CARMA}).
\end{example}

\begin{example}[Simultaneous confidence bands for the mean of multivariate spatio-temporal data]
We can generalize the idea in Example \ref{Ex-ST-TVM} to construct joint $100(1-\tau)\%$ confidence bands for the mean of multivariate spatio-temporal data. Consider the following model: 
\begin{align*}
\bm{\breve{Y}}(\bs) = \left(
\begin{array}{ccc}
\breve{Y}_{1}(\bs, t_{1}) & \cdots & \breve{Y}_{1}(\bs, t_{p}) \\
 \vdots & \ddots & \vdots \\
\breve{Y}_{L}(\bs, t_{1}) & \cdots & \breve{Y}_{L}(\bs, t_{p})
\end{array}
\right) = (\bm{\breve{Y}}(\bs, t_{1}) \cdots \bm{\breve{Y}}(\bs, t_{p})),\ t_{1}< \cdots < t_{p},
\end{align*}
where $\bm{\breve{Y}}(\bs, t) = (\breve{Y}_{1}(\bs,t),\dots, \breve{Y}_{L}(\bs,t))'$ and 
\begin{align*}
\bm{\breve{Y}}(\bs, t) &= \bm{\mu}(t) + \bm{v}(\bs,t) + \bm{\Upsilon}(\bs,t) =: \bX(\bs, t) + \bm{\Upsilon}(\bs,t).
\end{align*}
Here  $\bm{\mu}(t) = (\mu_{1}(t),\dots,\mu_{L}(t))'$ is a deterministic function, $\bm{v}(\bs) = (\bm{v}'(\bs,t_{1}), \cdots, \bm{v}'(\bs,t_{p}))'$ is a $\beta$-mixing random field in $\mathbb{R}^{pL}$ with $E[\bm{v}(\bs)] = \bm{0}$, and $\bm{\Upsilon}(\bm{s})=(\bm{\Upsilon}'(\bm{s},t_{1}),\dots, \bm{\Upsilon}'(\bm{s},t_{p}))'$ is a mean zero residual random field that is asymptotically negligible, i.e. $\bm{\Upsilon}(\bm{s})$ satisfies Condition (\ref{AN-error}) by replacing $p$ with $pL$. Define $\bY(\bs) = (\bm{\breve{Y}}'(\bs, t_{1}),\dots,\bm{\breve{Y}}'(\bs, t_{p}))'$. Then we can construct joint confidence intervals for the $pL$ dimensional mean vector
\begin{align*}
E\left[\left(
\begin{array}{c}
\bX(\bs,t_{1})\\
\vdots \\
\bX(\bs,t_{p})
\end{array}
\right)
\right] &= 
\left(
\begin{array}{c}
\bm{\mu}(t_{1})\\
\vdots \\
\bm{\mu}(t_{p})
\end{array}
\right).
\end{align*}
Then we obtain joint $100(1-\tau)\%$ confidence bands for mean functions $\{\mu_{k}(t): t=t_{1},\dots,t_{p}\}_{k=1}^{L}$ by linear interpolation of joint confidence intervals for each $\mu_{k}$ of the form $\hat{C}_{1-\tau}^{(k)} = \prod_{j=1}^{p}\hat{C}_{j,1-\tau}^{(k)}$ where
\begin{align*}
\hat{C}_{j,1-\tau}^{(k)} &= \left[\overline{Y}_{k+L(j-1),n} \pm \lambda_{n}^{-d/2}\sqrt{\Sigma_{k + L(j-1),k+L(j-1)}^{\bm{V}_{n}}}q_{n}(1-\tau)\right],\ \tau \in (0,1),
\end{align*}
$\overline{Y}_{k+L(j-1),n} = {1 \over n}\sum_{i=1}^{n}\breve{Y}_{k}(\bs_{i},t_{j})$ and $q_{n}(1-\tau)$ is the ($1-\tau$)-quantile of $\max_{1 \leq j \leq pL}\left|V_{j,n}/\sqrt{\Sigma_{j,j}^{\bm{V}_{n}}}\right|$. Define $\hat{C}^{(k)\ast}_{1-\tau}$ as $\hat{C}^{(k)}_{1-\tau}$ by replacing $\Sigma_{k + L(j-1),k+L(j-1)}^{\bm{V}_{n}}$ and $q_{n}(1-\tau)$ with $\hat{\Sigma}_{k + L(j-1),k+L(j-1)}^{\bm{V}_{n}}$ and $\hat{q}_{n}(1-\tau)$, respectively. Theorems \ref{Thm: GA_hyperrec} and \ref{Thm: GA_hyperrec-m} yield that
\[
P_{\cdot \mid \bZ}\left(\left\{\mu_{1}(\bm{t}) \in \hat{C}_{1-\tau}^{(1) \ast}\right\} \cap\cdots \cap \left\{\mu_{L}(\bm{t}) \in \hat{C}_{1-\tau}^{(L) \ast}\right\}\right) \to 1-\tau,\ P_{\bZ}-a.s.,
\]
as $n \to \infty$ where $\mu_{k}(\bm{t}) = (\mu_{k}(t_{1}),\dots, \mu_{k}(t_{p}))'$, $k=1,\dots, L$. 



\end{example}

\begin{example}[Change-point analysis for spatio-temporal data]\label{cp-test-temporal}
Consider a spatio-temporal process $Y(\bs,t)$, $s\in \R^d$ and $t\in \mathbb{Z}$ or $\mathbb{R}$,  
with $E[Y(\bs,t)] = \mu(t)$. Given the spatio-temporal observations $\{Y(\bs_i,t_j),i=1,\cdots,n;t_1<t_2<\cdots<t_p\}$,   our interest is to understand whether there is a shift in mean at any time $t_{j+1}$, as compared to previous observation time $t_j$. Let $\mu_j=\mu(t_j)$. We can formulate this as a hypothesis testing problem,    i.e., we are interested in simultaneously testing the set of null hypotheses
\begin{align*}
H_{j}: \mu_{j+1} - \mu_j = 0,\ 1\leq j \leq p-1
\end{align*}
against the alternatives $H'_{j}: \mu_{j+1} - \mu_j \neq 0,\ 1\leq j \leq p-1$. 

Inspired by the idea in \cite{ChChKa13}, we combine a general stepdown procedure described in \cite{RoWo05} with the SDWB developed to construct a multiple change-point test. Formally, let $\mathfrak{D}$ be a set of all data generating processes, and $\mathfrak{d}$ be the true process. Each null hypothesis $H_{j}$ is equivalent to $\mathfrak{d} \in \mathfrak{D}_{j}$ for some subset $\mathfrak{D}_{j}$ of $\mathfrak{D}$. Let $\mathcal{W}:= \{1,\dots,p-1\}$ and for $w \subset \mathcal{W}$ denote $\mathfrak{D}^{w}:= (\cap_{j \in w}\mathfrak{D}_{j})\cap (\cap_{j \notin w}\mathfrak{D}_{j}^{c})$ where $\mathfrak{D}_{j}^{c}:= \mathfrak{D} \backslash \mathfrak{D}_{j}$. We are interested in a procedure with the strong control of the family-wise
error rate. In other words, we seek a procedure that would reject at least one true null hypothesis with probability not greater than $\tau+o(1)$ uniformly
over a large class of data-generating processes and, in particular, uniformly over the set of true null hypotheses. The strong control of the family-wise error rate means
\begin{align*}
\sup_{w \subset \mathcal{W}}\sup_{\mathfrak{d} \in \mathfrak{D}^{w}}P^{(\mathfrak{d})}_{\cdot \mid \bZ}\left(\text{reject at least one hypothesis among $H_{j}$, $j \in w$}\right) \leq \tau + o(1),
\end{align*}
where $P^{(\mathfrak{d})}_{\cdot \mid \bZ}$ denotes the conditional probability distribution under the data generating process $\mathfrak{d}$ given $\{\bZ_{i}\}_{i \geq 1}$. The step down procedure of \cite{RoWo05} is described as follows. For a subset $w \in \mathcal{W}$, let $\hat{q}_{n,w}(1-\tau)$ be the estimator of the $(1-\tau)$-quantile of $\sqrt{\lambda_{n}^{d}}\max_{j \in w}|\overline{Y}_{j,n}/\sqrt{\Sigma_{j,j}^{\bm{V}_{n}}}|$ under $P_{\cdot|\bY, \bZ}$ defined as follows: 
\begin{align}\label{SN-CP-test-critical-value}
\hat{q}_{n,w}(1-\tau) = \inf\left\{t \in \R : P_{\cdot|\bY,\bZ}\left(\sqrt{\lambda_{n}^{d}}\max_{j \in w}\left|{\overline{Y}^{\ast}_{j,n} - \overline{Y}_{j,n} \over \sqrt{\hat{\Sigma}_{j,j}^{\bm{V}_{n}}}} \right|  \leq t\right) \geq 1-\tau\right\},
\end{align}
where $\overline{Y}_{j,n} = {1 \over n}\sum_{i=1}^{n}(Y(\bs_{i},t_{j+1}) - Y(\bs_{i}, t_j))$. In the first step, let $w(1) = \mathcal{W}$. Reject all hypotheses $H_{j}$ satisfying $\sqrt{\lambda_{n}^{d}}|\overline{Y}_{j,n}/\sqrt{\hat{\Sigma}_{j,j}^{\bm{V}_{n}}} | > \hat{q}_{n,w(1)}(1-\tau)$. If no null hypothesis is rejected, then stop.  If some $H_{j}$ are rejected, let $w(2)$ be the set of all null hypotheses that were not rejected in the first step. In step $\ell \geq 2$, let $w(\ell) \subset \mathcal{W}$ be the subset of null hypotheses that were not rejected up to step $\ell$. Reject all hypotheses $H_{j}$, $j \in w(\ell)$, satisfying $\sqrt{\lambda_{n}^{d}}|\overline{Y}_{j,n}/\sqrt{\hat{\Sigma}_{j,j}^{\bm{V}_{n}}}| > \hat{q}_{n,w(\ell)}(1-\tau)$. If no null hypothesis is rejected, then stop. If some $H_{j}$ are rejected, let $w(\ell+1)$ be the subset of all null hypotheses among $j \in w(\ell)$ that were not rejected. Proceed in this way until the algorithm stops.  If null hypotheses $H_{k_{1}},\dots, H_{k_{m}}$ ($1\leq k_{1}<k_{2}<\dots<k_{m}\leq p-1$) are rejected, then they imply the following structure:  
\begin{align*}
E[Y(\bs, t_k)] &= 
\begin{cases}
\mu^{(1)} & \text{for $1 \leq k \leq k_{1}$},\\
\mu^{(2)} & \text{for $k_{1}+1 \leq k \leq k_{2}$},\\
\vdots & \\
\mu^{(m)} & \text{for $k_{m}+1 \leq k \leq p$},
\end{cases}
\end{align*}
where $\mu^{(1)} \neq \mu^{(2)}$, $\cdots$, $\mu^{(m-1)} \neq \mu^{(m)}$. Therefore the spatio-temporal observations are segmented into $m$ pieces with a constant mean within each piece. This is effective change-point estimation or segmentation for spatio-temporal data with potential mean shifts over time. Below we shall offer some discussion about the difference between our method and those developed for high-dimensional time series. 

\begin{remark}[Connection/difference from change-point testing/estimation for high-dimensional time-ordered data]
\label{rem:change}
Our change-point test is applied to spatio-temporal data, which can be viewed as a special kind of high-dimensional time series with spatial dependence in its high-dimensional vector.
There is a growing literature of change-point detection for the mean of high-dimensional time-ordered data with or without temporal dependence [\cite{ChFr15}, \cite{Ji15}, \cite{Ch16},  \cite{WaSa18}, \cite{WaVoSh19}, \cite{EnHa19}, \cite{LiXuZhLi19},  \cite{YuCh20} and \cite{ZhWaSh21}].
It pays to highlight the main difference between ours and the above-mentioned ones. For one, the stepdown procedure involves the one sample (multiple) testing of a mean vector of dimension $(p-1)$, i.e., $\theta_j=\mu_{j+1}-\mu_j$, $j=1,\cdots,p-1$, where $p$ is the length of time series at hand. 
Under the approximate spatial stationarity assumption, we are able to take sample average over space to gain detection power since $Y(s_i,t_j),~i=1,\cdots,n$ share the same mean $\mu_j$. By contrast, the framework in all the above-mentioned literature on  high-dimensional mean change detection is different, as their parameter of interest is $n$-dimensional, where $n$ is the number of components in their high-dimensional vector. Typically few structural assumptions on the components of the mean vector is imposed  but some temporal i.i.d. or stationarity with weak dependence assumption [see \cite{Ch16}, \cite{Ji15}, \cite{LiXuZhLi19}] needs to be assumed and the sample average is taken over the  equally spaced time points. In our setting, we convert spatio-temporal data into a high-dimensional spatial data by stacking the time series at each location into a high-dimensional vector, thus neither regular temporal spacing nor temporal stationarity is required.


Recently \cite{ZhMaNgYa19} modeled a non-stationary spatio-temporal process as a piecewise stationary spatio-temporal process and proposed a composite likelihood based approach to perform change point and parameter estimation. Their framework is parametric and can handle change in both mean and auto-covariance in both space and time, whereas ours is nonparametric and focuses on the mean change. Both their work and ours assume (approximate) spatial stationarity, impose suitable mixing assumptions on the random field, and leverage the averaging in space under the increasing domain asymptotics. In another related work, \cite{GrKoRe17} developed a procedure of  detecting a change point in the mean function of a spatio-temporal process using tools from functional data analysis, and their framework and methodology are very different from ours.



\end{remark}
\end{example}

\section{Simulation results}\label{sec: simulations}

In this section, we present some simulation results to evaluate the finite sample properties of the SDWB in constructing simultaneous confidence intervals for the  mean vector of  high-dimensional spatial data. Let the sampling region $R_{n} = \lambda_{n}(-1/2,1/2]^2 \subset \R^{2}$ with $\lambda_{n} \in \{15, 25\}$. We consider three data generating processes (DGPs). 

The first DGP (DGP1) is the following compound Poisson-driven CAR($1$) (CP-CAR($1$))-type random field: 
\begin{align}\label{CP-CAR-eq}
\bY(\bs) &= \sum_{i=1}^{\infty}\bm{g}(\|\bs - \bm{x}_{i}\|)\bm{J}_{i},
\end{align}
where $\bm{x}_{i}$ denotes the location of the $i$-th unit point mass  of a Poisson random measure on $\R^{2}$ with intensity $\lambda = 1$ and $\{\bm{J}_{i}\}_{i \geq 1}$ is a sequence of i.i.d. random variables in $\R^{p}$. The CAR($1$) random field is a spatial extension of (well-balanced) L\'evy-driven Ornstein-Uhlenbeck (OU) processes. See \cite{Ba97, Ba01} and \cite{ScWo11} for examples of non-Gaussian OU processes. In our simulation study, we set  $\bm{g}(\|\bm{x}\|) = e^{-3\|\bm{x}\|}I_{p}$ and $\bm{J}_{1} \sim N(0, I_{p})$, where $I_{p}$ denotes the $p \times p$ identity matrix. To simulate the CP-CAR($1$) random field, we follow the algorithm described in \cite{BrMa17}: 
\begin{enumerate}
\item[(i)] Take $R'_{n}$ to be a sufficiently large set containing $R_{n}$. In this simulation study, we take $R'_{n} =35  \cdot  (-1/2, 1/2]^{2}$. 
\item[(ii)] Simulate a Poisson random variable $n(R'_{n})$ with mean $\lambda |R'_{n}|$ and set it as the number of  knots contained in $R'_{n}$.
\item[(iii)] Simulate $n(R'_{n})$ independent and uniformly distributed points $\bm{x}_{1},\dots,\bm{x}_{n(R'_{n})}$ in $R'_{n}$. 
\item[(iv)] Compute the truncated version of (\ref{CP-CAR-eq}): $\bY(\bs) = \sum_{i=1}^{n(R'_{n})}\bm{g}(\|\bs - \bm{x}_{i}\|)\bm{J}_{i}$.
\end{enumerate}

The second DGP (DGP2)  is a $p$-variate Gaussian random field with mean zero and independent components, each of which admits the following Mat\'ern covariance function: 
\begin{align}\label{Matern_cov_def}
\Cov(Y_{j}(\bm{s}), Y_{j}(\bm{0})) = {2^{1-\nu} \over \Gamma(\nu)}\left(\sqrt{2\nu}\|\bm{s}\|/a\right)^{\nu}K_{\nu}\left(\sqrt{2\nu}\|\bm{s}\|/a\right),\ \nu>0, a>0, 1\leq j \leq p,
\end{align}
where $\Gamma(\cdot)$ is the gamma function. In our simulation study, we set $\nu = 3/2$ and $a = 1/\sqrt{3}$.

The third DGP (DPG3) is the following model:  
\begin{align}\label{space_factor_model}
\bY(\bm{s}_{i}) = \bm{A}\bm{F}(\bm{s}_{i}) + \bm{R}(\bm{s}_{i}),\ i = 1,\dots, n, 
\end{align}
where $\bm{A}$ is a $p \times 5$ matrix, $\bm{F}(\bm{s})$ is a $5$-variate mean zero Gaussian random field with independent components that have the Mat\'ern covariance function (\ref{Matern_cov_def}) with $\nu = 3/2$, $a = 1/\sqrt{3}$, and $\bm{R}(\bm{s}_{i})$ are $p$-variate i.i.d. standard Gaussian random vectors. For each combination of $(n,p,\lambda_{n}, b_{n})$, we generate all the $p \times 5$ elements of $\bm{A}$ independently from the uniform distribution on the interval $[-1,1]$ and fix it for all Monte Carlo replications. Compared with DGP2, the Gaussian random field $\bY$ from the model (\ref{space_factor_model}) adds strong componentwise dependence through a factor model structure. 

In our simulation, we also set $p \in \{10, 100, 400\}$ and generate i.i.d. sampling sites $\bs_{1},\dots, \bs_{n} \in R_{n}$ from the uniform distribution on $R_{n}$ with $n = 100$. The configuration is compatible with the one in real data analysis ($p=75$ and $n = 101$); see Section \ref{real-data-US-ppt} of the supplement. We use the Bartlett kernel for the covariance function of the Gaussian random field $\{W(\bs): \bs \in \R^{2}\}$ and examine the coverage accuracy for the bandwidth $b_{n} \in \{1,\cdots,10\}$.
Figure \ref{Fig: Cov-Prob-b-plot} shows coverage probabilities of joint $95$\% (left panel) and $99$\% (right panel) confidence intervals for DGP1 (top row), DGP2 (middle row) and DGP3 (bottom row) based on 1000 Monte Carlo repetitions.  To compute the critical value $\hat{q}_{n}(1-\tau)$, we generate 1,500 bootstrap samples for each run of the simulations.

A few remarks are in order. (a), Comparing the cases $\lambda_{n} = 15$ and $\lambda_{n}=25$, we observe that the latter corresponds to more accurate coverages for the same combination of $(n,p,b_n)$. This can be explained by the  fact that the convergence rate of the sample mean is $\lambda_n^{d/2}$, and $\lambda_n^d$ plays the role of effective sample size here. (b), for all settings, there is a range of $b_n$s that yield empirical coverage levels that are closest to the nominal one. This suggests that in practice it is not necessary to find the optimal $b_n$ that corresponds to the optimal coverage, but instead we only need to locate the range of $b_n$ for which the coverage accuracy is almost optimal. (c),  SDWB works for both low-dimensional (i.e., $p=10$) and high-dimensional cases (i.e., $p=100,400$), and seems to be robust to strong componentwise dependence in view of the results for DGP2 and DGP3.   Overall,  the results are quite encouraging as empirical coverage probabilities of joint confidence intervals are reasonably close to the nominal ones for suitably selected bandwidth $b_n$.



 \begin{figure}
  \begin{center}
    \begin{tabular}{cc}

      \begin{minipage}{0.5\hsize}
        \begin{center}
          \includegraphics[clip, width=7cm]{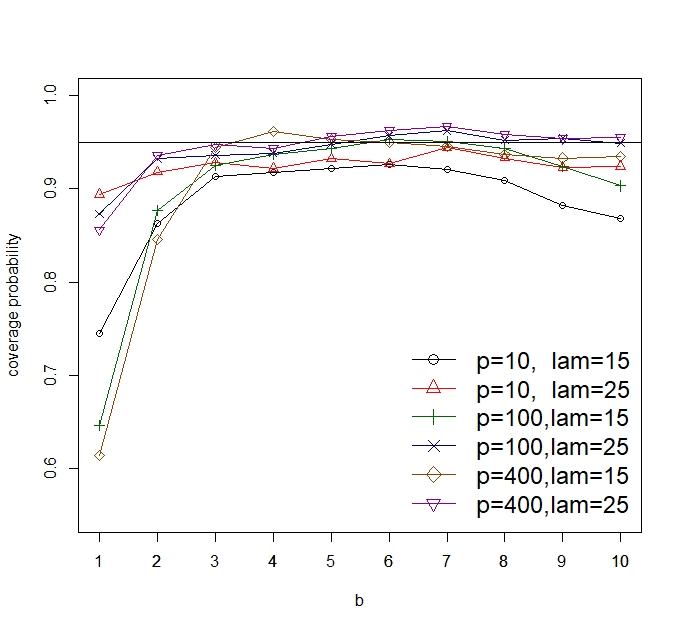}
        \end{center}
      \end{minipage}

       \begin{minipage}{0.5\hsize}
        \begin{center}
          \includegraphics[clip, width=7cm]{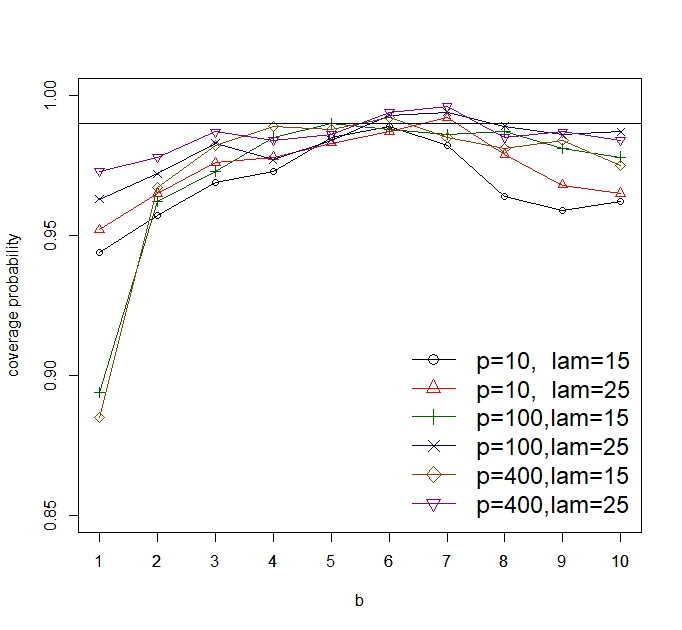}
        \end{center}
       \end{minipage} \\
       
       \begin{minipage}{0.5\hsize}
        \begin{center}
          \includegraphics[clip, width=7cm]{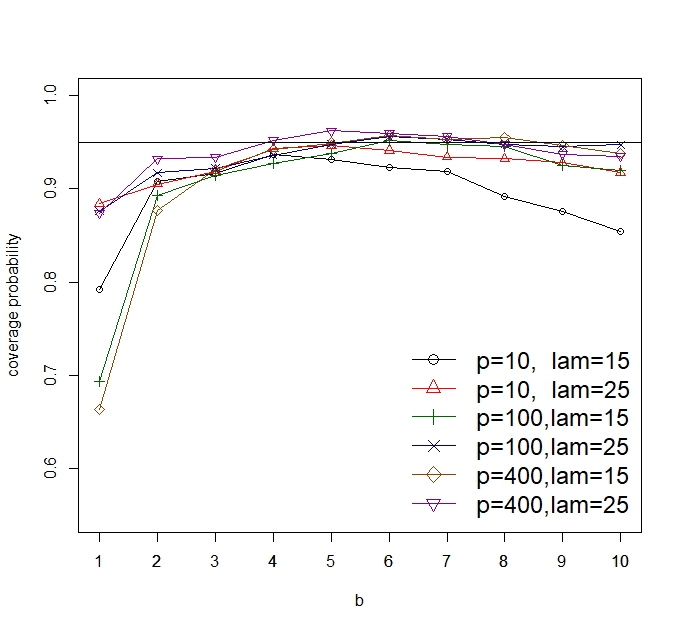}
        \end{center}
      \end{minipage}

       \begin{minipage}{0.5\hsize}
        \begin{center}
          \includegraphics[clip, width=7cm]{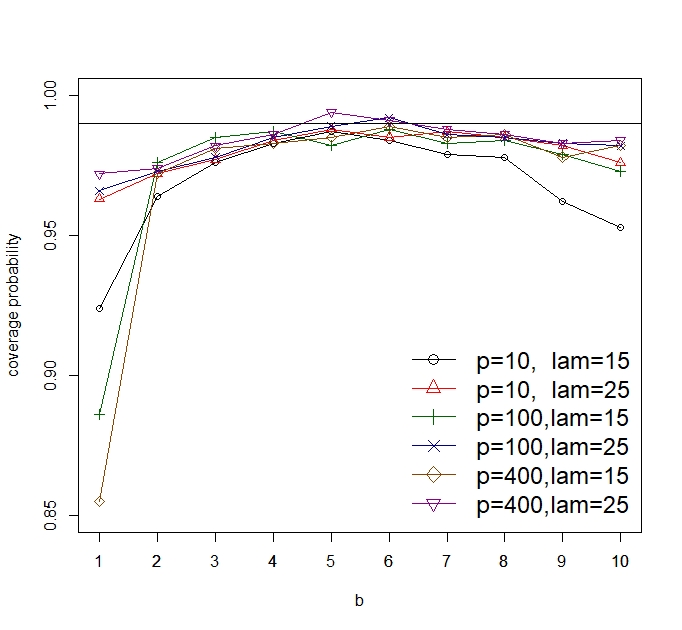}
        \end{center}
       \end{minipage} \\
       
       \begin{minipage}{0.5\hsize}
        \begin{center}
          \includegraphics[clip, width=7cm]{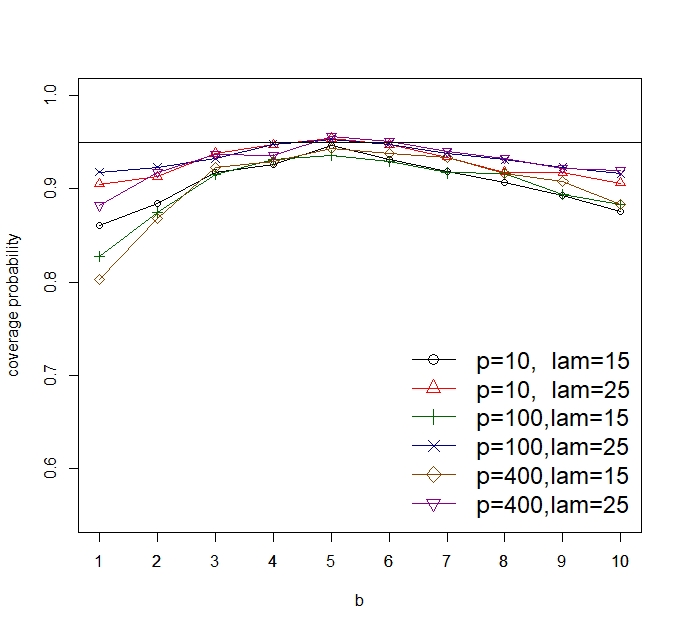}
        \end{center}
      \end{minipage}

       \begin{minipage}{0.5\hsize}
        \begin{center}
          \includegraphics[clip, width=7cm]{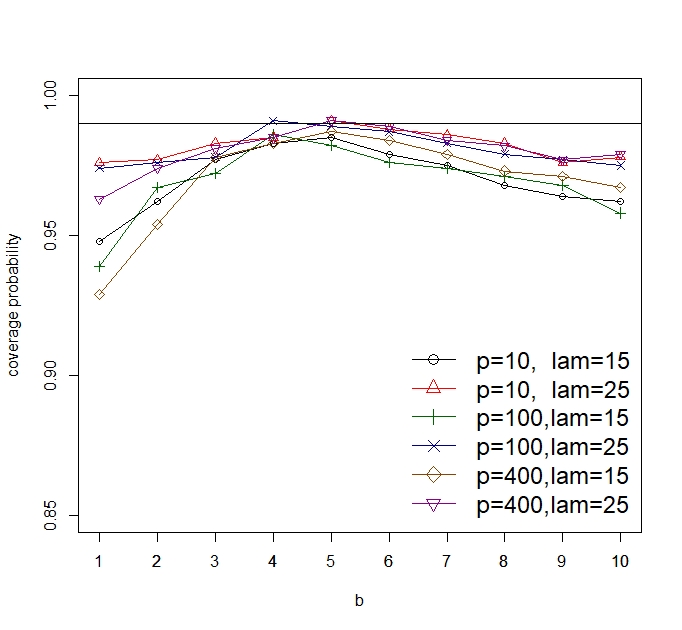}
        \end{center}
       \end{minipage} 

    \end{tabular}
    \caption{Coverage probabilities of joint $95$\% (left panel) and $99$\% (right panel) confidence intervals for DGP1 (top row), DGP2 (middle row) and DGP3 (bottom row).  \label{Fig: Cov-Prob-b-plot}}
  \end{center}
\end{figure}

\section{Conclusion}\label{sec: conclusion}

In this paper, we have advanced Gaussian approximation to high-dimensional spatial data observed at irregularly spaced sampling sites and proposed the spatially dependent wild bootstrap (SDWB) to allow feasible inference. We provide a rigorous theory for Gaussian approximation and bootstrap consistency under the stochastic sampling design in \cite{La03a}, which includes both pure increasing domain and mixed  increasing domain asymptotic frameworks.    SDWB is shown to be valid for a wide class of random fields that includes L\'evy-driven MA random fields and the popular Gaussian random field  as  special cases. We demonstrate the usefulness of SDWB by constructing joint confidence intervals of the mean of random field over time, and  performing change-point testing/estimation in the mean of spatio-temporal data.  The validity of our Gaussian approximation and associated bootstrap theory hinges on the approximate spatial stationarity, suitable mixing and moment assumptions. Both   irregularly temporal spacing and temporal nonstationarity can be accommodated in the application to  inference for  spatio-temporal  data. 



To conclude, we shall mention two important future research topics. First, an obvious one is to come up with a good data driven formula for the bandwidth parameter $b_n$, which plays an important role in the approximation accuracy of SDWB. Unfortunately, we are not aware of any results on the edgeworth expansion for the distribution of sample mean of spatial data, even in the increasing domain asymptotic framework and for the low-dimensional setting, let alone deriving the edgeworth expansion for the distribution of the  sample mean and its bootstrap counterpart in the high-dimensional spatial setting.  For Gaussian approximation of time series  and subsequent inference, a bandwidth parameter is often necessary; see \cite{ZhWu17}, \cite{ZhCh18}, \cite{ChYaZh17}, among others, and it seems difficult to extend their data-driven formula [see e.g., \cite{ChYaZh17}] to the spatial setting. One way out is to adopt the minimal volatility approach, as advocated by 
\cite{PoRoWo99} for subsampling and block bootstrap of low-dimensional time series, and it remains to see whether it works in our setting. Second, the inference problem we study is limited to the mean of random field since our Gaussian approximation result is stated for the mean of $p-$dimensional spatial data. We are hopeful that our theory can be extended to cover inference for the parameter related to second order properties of a random field, such as variogram at a particular lag, given the recent work by \cite{ChYaZh17} on testing white noise hypothesis for high-dimensional time series.
 We leave both topics for future investigation.


\newpage

\bibliography{Ref4}
\bibliographystyle{chicago}

\clearpage

\begin{center}
{\LARGE Supplementary Material - Gaussian approximation and spatially dependent wild bootstrap for high-dimensional spatial data \\
\vspace{3mm}
\Large by Daisuke Kurisu, Kengo Kato and Xiaofeng Shao}
\end{center}

\setcounter{section}{0}
\setcounter{figure}{0}
\setcounter{table}{0}
\def\thetable{S\arabic{table}}
\def\thefigure{S\arabic{figure}}
\def\thesection{S\arabic{section}}

\bigskip

The supplement contains all technical proofs and a real data illustration. In what follows, we set $r_{n} = n^{2}\lambda_{n}^{-d}$ and $\|\bm{x}\|_{\infty} = \max_{1 \leq j \leq p}|x_{j}|$ for $\bm{x} = (x_{1},\dots, x_{p})' \in \mathbb{R}^{p}$.
\appendix


\section{Proofs of Theorems \ref{Thm: GA_hyperrec} and  \ref{Thm: GA_hyperrec-m}}

\subsection{Preliminaries}
\label{sec: preliminaries}

Recall that $\{\lambda_{1,n}\}$ and $\{\lambda_{2.n}\}$ are sequences of positive numbers such that $\lambda_{1,n}/\lambda_{n} + \lambda_{2,n}/\lambda_{1,n} \to 0$ as $n \to \infty$. Let $\lambda_{3,n} = \lambda_{1,n} + \lambda_{2,n}$. We consider a partition of $\R^{d}$ by hypercubes of the form $\Gamma_{n}(\bm{\ell};\bm{0}) = (\bm{\ell} + (0,1]^{d})\lambda_{3,n}$, $\bm{\ell} = (\ell_{1}, \dots, \ell_{d})' \in \mathbb{Z}^{d}$ and divide $\Gamma_{n}(\bm{\ell};\bm{0})$ into $2^{d}$ hypercubes as follows: 
\[
\Gamma_{n}(\bm{\ell};\bm{\epsilon}) = \prod_{j=1}^{d}I_{j}(\epsilon_{j}),\ \bm{\epsilon} = (\epsilon_{1},\dots, \epsilon_{d})' \in \{1,2\}^{d}, 
\]
where for $j=1,\dots,d$,
\[
I_{j}(\epsilon_{j}) = 
\begin{cases}
(\ell_{j}\lambda_{3,n}, \ell_{j}\lambda_{3,n} + \lambda_{1,n}] & \text{if $\epsilon_{j} = 1$}, \\
(\ell_{j}\lambda_{3,n} + \lambda_{1,n}, (\ell_{j}+1)\lambda_{3,n}] & \text{if $\epsilon_{j} = 2$}.
\end{cases}
\]
We note that 
\begin{align}\label{partition volume}
|\Gamma_{n}(\bm{\ell};\bm{\epsilon})| = \lambda_{1,n}^{q(\bm{\epsilon})}\lambda_{2,n}^{d-q(\bm{\epsilon})}
\end{align}
for any $\bm{\ell} \in \mathbb{Z}^{d}$ and $\bm{\epsilon} \in \{1,2\}^{d}$, where $q(\bm{\epsilon}) = [\![\{1 \leq j \leq d: \epsilon_{j} = 1\}]\!]$. Let $\bm{\epsilon}_{0} = (1,\dots, 1)'$. The partitions $\Gamma_{n}(\bm{\ell};\bm{\epsilon}_{0})$ correspond to ``large blocks'' and the partitions $\Gamma(\bm{\ell};\bm{\epsilon})$ for $\bm{\epsilon} \neq \bm{\epsilon}_{0}$ correspond to ``small blocks''.

Let $L_{1,n} = \{\bm{\ell} \in \mathbb{Z}^{d}: \Gamma_{n}(\bm{\ell},\bm{0}) \subset R_{n}\}$ denote the index set of all hypercubes $ \Gamma_{n}(\bm{\ell},\bm{0})$ that are contained in $R_{n}$, and let $L_{2,n} = \{\bm{\ell} \in \mathbb{Z}^{d}:  \Gamma_{n}(\bm{\ell},\bm{0}) \cap R_{n} \neq 0,  \Gamma_{n}(\bm{\ell},\bm{0}) \cap R_{n}^{c} \neq \emptyset \}$ be the index set of boundary hypercubes. Define $L_{n} = L_{1,n} \cup L_{2,n}$ and 
\[
S_{n}(\bm{\ell};\bm{\epsilon}) = \sum_{i:\bs_{i}\in \Gamma_{n}(\bm{\ell};\bm{\epsilon}) \cap R_{n}}\bX(\bs_{i}) = (S_{n}^{(1)}(\bm{\ell};\bm{\epsilon}), \dots, S_{n}^{(p)}(\bm{\ell};\bm{\epsilon}))'.
\]
Then we have 
\begin{align*}
S_{n} &= (S_{n}^{(1)},\dots, S_{n}^{(p)})' = \sum_{i=1}^{n}\bX(\bs_{i}) \\
&= \sum_{\bm{\ell} \in L_{n}}S_{n}(\bm{\ell};\bm{\epsilon}_{0}) + \sum_{\bm{\epsilon} \neq \bm{\epsilon}_{0}}\underbrace{\sum_{\bm{\ell}\in L_{1,n}}S_{n}(\bm{\ell};\bm{\epsilon})}_{=: S_{2,n}(\bm{\epsilon})} + \sum_{\bm{\epsilon} \neq \bm{\epsilon}_{0}}\underbrace{\sum_{\bm{\ell} \in L_{2,n}}S_{n}(\bm{\ell};\bm{\epsilon})}_{=: S_{3,n}(\bm{\epsilon})}\\
&=: S_{1,n} + \sum_{\bm{\epsilon} \neq \bm{\epsilon}_{0}}S_{2,n}(\bm{\epsilon}) + \sum_{\bm{\epsilon} \neq \bm{\epsilon}_{0}}S_{3,n}(\bm{\epsilon}). 
\end{align*}
We first establish the high-dimensional CLT for $r_{n}^{-1/2}S_{n} = \sqrt{\lambda_{n}^{d}}\overline{\bm{X}}_{n}$ and then  for $\sqrt{\lambda_{n}^{d}}\overline{\bY}_{n}$.

We state some lemmas that will be used in the subsequent proofs. The following Lemmas \ref{n summands} and \ref{variance-rate} hold under the assumptions of Theorems \ref{Thm: GA_hyperrec} and \ref{Thm: GA_hyperrec-m}. 
\begin{lemma}\label{n summands}
Let $\mathcal{I}_{n} = \{\bm{i} \in \mathbb{Z}^{d}: (\bm{i} +(0,1]^{d})\cap  R_{n} \neq \emptyset \}$. Then we have 
\[
P_{\bZ}\left(\sum_{k=1}^{n}1\{\lambda_{n} \bZ_{k} \in (\bm{i} +(0,1]^{d})\cap  R_{n} \} > 2(\log n + n\lambda_{n}^{-d})\ \text{for some $\bm{i} \in \mathcal{I}_{n}$, i.o.}\right) = 0
\]
and 
\[
P_{\bZ}\left(\sum_{k=1}^{n}1\{\lambda_{n} \bZ_{k} \in \Gamma_{n}(\bm{\ell}; \bm{\epsilon})\} > C\lambda_{1,n}^{q(\bm{\epsilon})}\lambda_{2,n}^{d-q(\bm{\epsilon})}n\lambda_{n}^{-d}\ \text{for some $\bm{\ell} \in L_{1,n}$, i.o.}\right) = 0
\]
for any $\bm{\epsilon} \in \{1,2\}^{d}$. 
\end{lemma}
\begin{proof}
See the proofs of Lemmas A.1 and 5.1 in \cite{La03b} for each statement. 
\end{proof}

\begin{remark}
Since $[\![\{\bm{\epsilon} \in \{1,2\}^{d}: \bm{\epsilon} \neq \bm{\epsilon}_{0} \}]\!] = 2^{d}-1$, $[\![L_{1,n}]\!] \sim (\lambda_{n}/\lambda_{3,n})^{d} \sim (\lambda_{n}/\lambda_{1,n})^{d}$ and $[\![L_{2,n}]\!] \sim (\lambda_{n}/\lambda_{3,n})^{d-1} \sim (\lambda_{n}/\lambda_{1,n})^{d-1} \ll [\![L_{1,n}]\!]$, Lemma \ref{n summands} and Equation (\ref{partition volume}) imply that for sufficiently large $n$, the numbers of summands of $S_{2,n}$ and $S_{3,n}$ are at most $O\left(\lambda_{1,n}^{d-1}\lambda_{2,n}n\lambda_{n}^{-d}(\lambda_{n}/\lambda_{1,n})^{d}\right) = O\left({\lambda_{2,n} \over \lambda_{1,n}}n\right)$ and $O\left(\lambda_{1,n}^{d-1}\lambda_{2,n}n\lambda_{n}^{-d}(\lambda_{n}/\lambda_{1,n})^{d-1}\right) = O\left({\lambda_{2,n} \over \lambda_{n}}n\right)$, respectively. 
\end{remark}
Define $\bm{\sigma}_{n}^{2} = (\sigma_{1,n}^{2},\dots, \sigma_{p,n}^{2})'$ where 
\begin{align*}
\sigma^{2}_{j,n} &= \sum_{\bm{\ell} \in L_{n}}\sum_{k_{1}=1}^{n}\sum_{k_{2}=1}^{n}\Sigma_{j,j}(\lambda_{n}(\bZ_{k_{1}} - \bZ_{k_{2}}))1\{\lambda_{n}\bZ_{k_{1}}, \lambda_{n}\bZ_{k_{2}} \in \Gamma_{n}(\bm{\ell};\bm{\epsilon}_{0}) \cap R_{n}\}.
\end{align*}
Recall that we have assumed that $p = O(n^{\alpha})$. 
\begin{lemma}\label{variance-rate}
Let $E_{\bZ}$ denote the expectation with respect to $\{\bZ_{i}\}_{i \geq 1}$. Then we have
\begin{align}
E_{\bZ}\left[\sigma^{2}_{j,n}\right] &= n^{2}\lambda_{n}^{-d}\left(\left(\int_{\R^{d}} \Sigma_{j,j}(\bm{x})d\bm{x}\right)\left(\int_{R_{0}} f^{2}(\bm{z})d\bm{z}\right) + o(1)\right) + \Sigma_{j,j}(\bm{0})n(1 + o(1)) \label{order-sig0}
\end{align}
uniformly over $1 \leq j \leq p$, and 
\begin{align}\label{unif-sig0}
\left\|\bm{\sigma}_{n}^{2} - E_{\bZ}[\bm{\sigma}_{n}^{2}]\right\|_{\infty} = O\left(n^{2-{1 \over 2r_{0}}}\lambda_{n}^{-d}\right) \quad \text{$P_{\bZ}$-a.s.}
\end{align}
for any even integer $r_{0}$ such that $r_{0} > (2\alpha+3)\vee 4$. Additionally, if $f$ is Lipschitz continuous inside $R_{0}$ and
\begin{align}
    \max_{1 \leq j \leq p}\int_{\|\bm{x}\|>\lambda_{2,n}}|\Sigma_{j,j}(\bm{x})|d\bm{x} = O(n^{-c'/2}), \label{cov_matrix_decay_A2}
\end{align}
then we can replace $o(1)$ in (\ref{order-sig0}) with $O(n^{-c'/2})$. 
\end{lemma}

\begin{proof}
In this proof, the $o (\cdot)$ and $O(\cdot)$ terms should be understood to hold uniformly over $1 \le j \le p$. We first verify the expansion (\ref{order-sig0}). Let $R'_{n}(\bm{\ell}) = \{\bm{x} \in \R^{d}: \bm{x} = \bm{x}_{1} - \bm{x}_{2}\ \text{for some}\ \bm{x}_{1}, \bm{x}_{2} \in \Gamma_{n}(\bm{\ell};\bm{\epsilon}_{0}) \cap R_{n}\}$, $R''_{n}(\bm{x};\bm{\ell}) = \{\bm{z} \in R_{0} : \bm{z}, \bm{z} + \lambda_{n}^{-1}\bm{x} \in \lambda_{1,n}\lambda_{n}^{-1}(\bm{\ell} + (0,1]^{d}) \cap R_{0}\}$, and $R'''_{n}(\bm{\ell}) = \{\bm{z} \in R_{0} : \bm{z} \in \lambda_{1,n}\lambda_{n}^{-1}(\bm{\ell} + (0,1]^{d}) \cap R_{0}\}$. We note that  $R''(\bm{x};\bm{\ell}_{1}) \cap R''(\bm{x};\bm{\ell}_{2}) = \emptyset$ whenever $\bm{\ell}_{1} \neq \bm{\ell}_{2}$. Then, we can expand $E_{\bZ}[\sigma^{2}_{j,n}]$ as
\begin{align*}
E_{\bZ}\left[\sigma^{2}_{j,n}\right] &= n(n-1)\sum_{\bm{\ell} \in L_{n}}E_{\bZ}\left[\Sigma_{j,j}(\lambda_{n}(\bZ_{k_{1}} - \bZ_{k_{2}}))1\{\lambda_{n}\bZ_{k_{1}}, \lambda_{n}\bZ_{k_{2}} \in \Gamma_{n}(\bm{\ell};\bm{\epsilon}_{0}) \cap R_{n} \}\right] \nonumber \\
&\quad + n\Sigma_{j,j}(\bm{0})\sum_{\bm{\ell} \in L_{n}}E_{\bZ}\left[1\{\lambda_{n}\bZ_{k_{1}}\in \Gamma_{n}(\bm{\ell};\bm{\epsilon}_{0}) \cap R_{n}\}\right] \nonumber \\
&= n(n-1)\lambda_{n}^{-d}\sum_{\bm{\ell} \in L_{n}}\int_{R'_{n}(\bm{\ell})}\Sigma_{j,j}(\bm{x})\left(\int_{R''_{n}(\bm{x};\bm{\ell})} f(\bm{z}+\lambda_{n}^{-1}\bm{x})f(\bm{z})d\bm{z}\right)d\bm{x}\\
&\quad + n\Sigma_{j,j}(\bm{0})\sum_{\bm{\ell} \in L_{n}}\int_{R'''_{n}(\bm{\ell})}f(\bm{z})d\bm{z} \nonumber \\
&= n(n-1)\lambda_{n}^{-d}\left\{\int_{\lambda_{1,n}(-1,1]^{d}} \Sigma_{j,j}(\bm{x})\left(\sum_{\bm{\ell} \in L_{1,n}}\int_{R''_{n}(\bm{x};\bm{\ell})} f(\bm{z}+\lambda_{n}^{-1}\bm{x})f(\bm{z})d\bm{z}\right)d\bm{x} \right. \\
& \left. \quad  + \sum_{\bm{\ell} \in L_{2,n}}\!\!\int_{R'_{n}(\bm{\ell})}\!\!\Sigma_{j,j}(\bm{x})\!\left(\int_{R''_{n}(\bm{x};\bm{\ell})} \!\!f(\bm{z}+\lambda_{n}^{-1}\bm{x})f(\bm{z})d\bm{z}\right)\!d\bm{x}\right\} + n\Sigma_{j,j}(\bm{0})\!\!\sum_{\bm{\ell} \in L_{n}}\int_{R'''_{n}(\bm{\ell})}f(\bm{z})d\bm{z}.
\end{align*}
Observe that
\[
\sum_{\bm{\ell} \in L_{1,n}}\int_{R''_{n}(\bm{x};\bm{\ell})} f(\bm{z}+\lambda_{n}^{-1}\bm{x})f(\bm{z})d\bm{z} = \int_{\cup_{\bm{\ell} \in L_{1,n} }R''_{n}(\bm{x};\bm{\ell})} f(\bm{z}+\lambda_{n}^{-1}\bm{x})f(\bm{z})d\bm{z}
\]
and $f(\bm{z}+\lambda_{n}^{-1}\bm{x})f(\bm{z})1\left\{\bm{z} \in \cup_{\bm{\ell} \in L_{1,n}}R''_{n}(\bm{x};\bm{\ell})\right\} \leq \|f\|_{\infty}f(\bm{z})1\{\bm{z} \in R_{0}\}$ with $\|f\|_{\infty} = \sup_{\bm{z} \in R_{0}}f(\bm{z})$.
Let $|A|$ be the Lebesgue measure of a Borel subset $A$ of $\R^{d}$.  Since $[\![L_{2,n}]\!] \sim (\lambda_{n}/\lambda_{1,n})^{d-1}$ and $[\![L_{1,n}]\!] \sim (\lambda_{n}/\lambda_{1,n})^{d}$, we have 
\begin{align*}
\sup_{\bm{x} \in \lambda_{2,n}(-1,1]^{d}}\left|R_{0} \backslash \left(\cup_{\bm{\ell} \in L_{1,n}}R''_{n}(\bm{x};\bm{\ell})\right)\right| & \lesssim [\![L_{1,n}]\!]{1 \over \lambda_{n}}\left({\lambda_{1,n} \over \lambda_{n}}\right)^{d-1} + [\![L_{2,n}]\!]\left({\lambda_{1,n} \over \lambda_{n}}\right)^{d}\\
&\lesssim \left({\lambda_{n} \over \lambda_{1,n}}\right)^{d}{1 \over \lambda_{n}}\left({\lambda_{1,n} \over \lambda_{n}}\right)^{d-1} + \left({\lambda_{n} \over \lambda_{1,n}}\right)^{d-1}\left({\lambda_{1,n} \over \lambda_{n}}\right)^{d} \\
&= {1 \over \lambda_{1,n}} + {\lambda_{1,n} \over \lambda_{n}} \lesssim n^{-c'/2},\\
\sup_{\bm{x} \in \lambda_{2,n}(-1,1]^{d}}\left|\cup_{\bm{\ell} \in L_{2,n}}R''_{n}(\bm{x};\bm{\ell})\right|  &= \sum_{\bm{\ell} \in L_{2,n}}\sup_{\bm{x} \in \lambda_{2,n}(-1,1]^{d}}\left|R''_{n}(\bm{x};\bm{\ell})\right|\\ 
&\lesssim [\![L_{2,n}]\!]\left({\lambda_{1,n} \over \lambda_{n}}\right)^{d} \lesssim {\lambda_{1,n} \over \lambda_{n}} \lesssim n^{-c'/2},
\end{align*}
where the inequalities are up to constants independent of $n$ and $\bm{x}$. Likewise, we have 
\begin{align*}
\left|R_{0} \backslash \left(\cup_{\bm{\ell} \in L_{1,n}}R'''_{n}(\bm{\ell})\right)\right| &\lesssim n^{-c'/2} \quad \text{and} \quad \left|\cup_{\bm{\ell} \in L_{2,n}}R'''_{n}(\bm{\ell})\right| \lesssim n^{-c'/2}.
\end{align*}
Thus, the dominated convergence theorem yields that 
\begin{align}\label{L0n-bound-sig}
    \sum_{\bm{\ell} \in L_{n}}\int_{R'''_{n}(\bm{\ell})}f(\bm{z})d\bm{z} = \int_{R_{0}}f(\bm{z})d\bm{z} + O(n^{-c'/2}) = 1 + O(n^{-c'/2}),
\end{align}
\begin{align}
\label{L1n-bound-sig}
&\int_{\lambda_{1,n}(-1,1]^{d}} \Sigma_{j,j}(\bm{x})\left(\int_{\cup_{\bm{\ell} \in L_{1,n}}R''_{n}(\bm{x};\bm{\ell})}\!\!\!\! f(\bm{z}+\lambda_{n}^{-1}\bm{x})f(\bm{z})d\bm{z}\right)d\bm{x} \nonumber \\
&= \int_{\lambda_{2,n}(-1,1]^{d}} \Sigma_{j,j}(\bm{x})\left(\int_{\cup_{\bm{\ell} \in L_{1,n}}R''_{n}(\bm{x};\bm{\ell})}\!\!\!\! f(\bm{z}+\lambda_{n}^{-1}\bm{x})f(\bm{z})d\bm{z}\right)d\bm{x} \nonumber \\
&\quad + \int_{\lambda_{1,n}(-1,1]^{d}\backslash \lambda_{2,n}(-1,1]^{d}} \Sigma_{j,j}(\bm{x})\left(\int_{\cup_{\bm{\ell} \in L_{1,n}}R''_{n}(\bm{x};\bm{\ell})}\!\!\!\! f(\bm{z}+\lambda_{n}^{-1}\bm{x})f(\bm{z})d\bm{z}\right)d\bm{x} \nonumber \\
&= \int_{\lambda_{2,n}(-1,1]^{d}} \Sigma_{j,j}(\bm{x})d\bm{x}\left(\int_{R_{0}}f^{2}(\bm{z})d\bm{z}\right) + O\left(\int_{\|\bm{x}\| \geq \lambda_{2,n}}\Sigma_{j,j}(\bm{x})d\bm{x}\right) + o(1) \nonumber \\
&= \int_{\R^{d}} \Sigma_{j,j}(\bm{x})d\bm{x}\left(\int_{R_{0}}f^{2}(\bm{z})d\bm{z}\right) + o(1). 
\end{align}
Now we consider the case that $f$ is Lipschitz continuous inside $R_{0}$ and (\ref{cov_matrix_decay_A2}) holds. Decompose 
\begin{align*}
    &\int_{\lambda_{1,n}(-1,1]^{d}} \Sigma_{j,j}(\bm{x})\left(\int_{\cup_{\bm{\ell} \in L_{1,n}}R''_{n}(\bm{x};\bm{\ell})}\!\!\!\!f(\bm{z}+\lambda_{n}^{-1}\bm{x})f(\bm{z})d\bm{z}\right)d\bm{x} \nonumber \\
    & = \int_{\lambda_{2,n}(-1,1]^{d}} \Sigma_{j,j}(\bm{x})\left(\int_{\cup_{\bm{\ell} \in L_{1,n}}R''_{n}(\bm{x};\bm{\ell})}\!\!\!\!\left\{f(\bm{z}+\lambda_{n}^{-1}\bm{x}) - f(\bm{z})\right\}f(\bm{z})d\bm{z}\right)d\bm{x}  \nonumber \\
    &\quad + \int_{\lambda_{1,n}(-1,1]^{d} \backslash \lambda_{2,n}(-1,1]^d} \Sigma_{j,j}(\bm{x})\left(\int_{\cup_{\bm{\ell} \in L_{1,n}}R''_{n}(\bm{x};\bm{\ell})}\!\!\!\!\left\{f(\bm{z}+\lambda_{n}^{-1}\bm{x}) - f(\bm{z})\right\}f(\bm{z})d\bm{z}\right)d\bm{x}  \nonumber \\
    &\quad + \int_{\lambda_{2,n}(-1,1]^{d}} \Sigma_{j,j}(\bm{x})\left(\int_{\cup_{\bm{\ell} \in L_{1,n}}R''_{n}(\bm{x};\bm{\ell})}f^{2}(\bm{z})d\bm{z}\right)d\bm{x} \nonumber \\
    &\quad + \int_{\lambda_{1,n}(-1,1]^{d}\backslash \lambda_{2,n}(-1,1]^{d}} \Sigma_{j,j}(\bm{x})\left(\int_{\cup_{\bm{\ell} \in L_{1,n}}R''_{n}(\bm{x};\bm{\ell})}f^{2}(\bm{z})d\bm{z}\right)d\bm{x} \nonumber \\
    &=: I_{1,n} + I_{2,n} + I_{3,n} + I_{4,n}.
\end{align*}
Then we have 
\begin{align}
    |I_{1,n}| &\lesssim (\lambda_{2,n}/\lambda_{n}) \int_{\lambda_{2,n}(-1,1]^{d}}|\Sigma_{j,j}(\bm{x})|d\bm{x} = O(\lambda_{2,n}/\lambda_{n}) = O(n^{-c'/2}), \label{I_1n_ineq} \\
    |I_{2,n}| &\lesssim \int_{\|\bm{x}\| \geq \lambda_{2,n}}|\Sigma_{j,j}(\bm{x})|d\bm{x} = O(n^{-c'/2}),\label{I_2n_ineq} \\
    I_{3,n} &= \int_{\lambda_{2,n}(-1,1]^{d}}\Sigma_{j,j}(\bm{x})d\bm{x}\left(\int_{R_{0}}f^{2}(\bm{z})d\bm{z}\right) + O(n^{-c'/2}), \label{I_3n_eq}\\
    |I_{4,n}| &\lesssim \int_{\|\bm{x}\| \geq \lambda_{2,n}}|\Sigma_{j,j}(\bm{x})|d\bm{x} = O(n^{-c'/2}). \label{I_4n_ineq}
\end{align}
Combining (\ref{I_1n_ineq})--(\ref{I_4n_ineq}), we have 
\begin{align}
    &\int_{\lambda_{1,n}(-1,1]^{d}} \Sigma_{j,j}(\bm{x})\left(\int_{\cup_{\bm{\ell} \in L_{1,n}}R''_{n}(\bm{x};\bm{\ell})}\!\!\!\!f(\bm{z}+\lambda_{n}^{-1}\bm{x})f(\bm{z})d\bm{z}\right)d\bm{x} \nonumber \\
    &= \int_{\mathbb{R}^{d}}\Sigma_{j,j}(\bm{x})d\bm{x}\left(\int_{R_{0}}f^{2}(\bm{z})d\bm{z}\right) + O(n^{-c'/2}). \label{I0n_sig_eq}
\end{align}
Likewise, we have
\begin{align}
\label{L2n-bound-sig}
&\left|\sum_{\bm{\ell} \in L_{2,n}}\int_{R'_{n}(\bm{\ell})}\Sigma_{j,j}(\bm{x})\left(\int_{R''_{n}(\bm{x};\bm{\ell})} f(\bm{z}+\lambda_{n}^{-1}\bm{x})f(\bm{z})d\bm{z}\right)d\bm{x}\right| \nonumber \\
&\leq \int_{\lambda_{1,n}(-1,1]^{d}} \left|\Sigma_{j,j}(\bm{x})\right|\left(\sum_{\bm{\ell} \in L_{2,n}}\int_{R''_{n}(\bm{x};\bm{\ell})} f(\bm{z}+\lambda_{n}^{-1}\bm{x})f(\bm{z})d\bm{z}\right)d\bm{x} \nonumber \\
&\leq \|f\|_{\infty}\int_{\lambda_{1,n}(-1,1]^{d}} \left|\Sigma_{j,j}(\bm{x})\right|\left(\int_{\cup_{\bm{\ell} \in L_{2,n}}R''_{n}(\bm{x};\bm{\ell})} f(\bm{z})d\bm{z}\right)d\bm{x} = O(n^{-c'/2}).
\end{align}
Combining (\ref{L0n-bound-sig}), (\ref{L1n-bound-sig}) (or (\ref{I0n_sig_eq})) and (\ref{L2n-bound-sig}) leads to the expansion (\ref{order-sig0}).

Second, we shall verify (\ref{unif-sig0}). Since $\Gamma_{n}(\bm{\ell}_{1};\bm{\epsilon}_{0}) \cap \Gamma_{n}(\bm{\ell}_{2};\bm{\epsilon}_{0}) = \emptyset$ for $\bm{\ell}_{1} \neq \bm{\ell}_{2}$, we have 
\begin{align*}
\sum_{\bm{\ell} \in L_{n}}1\{\lambda_{n}\bZ_{k_{1}}, \lambda_{n}\bZ_{k_{2}} \in \Gamma_{n}(\bm{\ell};\bm{\epsilon}_{0}) \cap R_{n}\} \leq 1.
\end{align*}
Define $h_{n}(\bm{x}_{1},\bm{x}_{2}) = \Sigma_{j,j}(\lambda_{n}(\bm{x}_{1} - \bm{x}_{2}))\sum_{\bm{\ell} \in L_{n}}1\{\lambda_{n}\bm{x}_{1}, \lambda_{n}\bm{x}_{2} \in \Gamma_{n}(\bm{\ell};\bm{\epsilon}_{0}) \cap R_{n}\}$. Then $\sigma_{j,n}^{2} = \sum_{k_{1}=1}^{n}\sum_{k_{2}=1}^{n}h_{n}(\bZ_{k_{1}}, \bZ_{k_{2}})$ is a $V$-statistics of order $2$ with kernel $h_{n}$. For any even integer $r\geq 4$ and for any $\bm{z} \in \R^{d}$, we have 
\begin{align}
\left|E_{\bZ}\left[\left(h_{n}(\bm{z}, \bZ_{1})\right)^{r}\right]\right| 
&\leq \left|\int |\Sigma_{j,j}(\lambda_{n}(\bm{z} - \bs))|^{r}f(\bs)d\bs\right| \nonumber \\
&\leq \lambda_{n}^{-d}\|f\|_{\infty}\int \left|\Sigma_{j,j}(\bs)\right|^{r}d\bs, \label{r-m-bound1} \\
\left|E_{\bZ}\left[\left(h_{n}(\bZ_{1}, \bZ_{2})\right)^{r}\right]\right| 
&\leq E_{\bZ}\left[\left|E_{\bZ}\left[\left. \left(h_{n}(\bZ_{1},\bZ_{2})\right)^{r} \right|\bZ_{1}\right]\right|\right] \nonumber \\
&\leq \lambda_{n}^{-d}\|f\|_{\infty}\int \left|\Sigma_{j,j}(\bs)\right|^{r}d\bs. \label{r-m-bound2}
\end{align}
Define $h_{1,n}(\bm{z}) = E_{\bZ}[h_{n}(\bm{z},\bZ_{1})]$ for $\bm{z} \in \R^{d}$ and 
\begin{align*}
D_{1,n} &= \sum_{k=1}^{n}\left(h_{n}(\bZ_{k},\bZ_{k}) - E_{\bZ}[h_{n}(\bZ_{1},\bZ_{1})]\right),\\ 
D_{2,n} &= \sum_{k=1}^{n-1}(n-k)(h_{1,n}(\bZ_{k}) - E_{\bZ}[h_{1,n}(\bZ_{k})]),\\
D_{3,n} &= \sum_{k=2}^{n}U_{k},\ U_{k} = \sum_{i=1}^{k-1}(h_{n}(\bZ_{k}, \bZ_{i}) - h_{1,n}(\bZ_{i})),\ k=2,\dots,n.
\end{align*}
It follows that $\sigma_{j,n}^{2} - E_{\bZ}[\sigma_{j,n}^{2}] = D_{1,n} + D_{2,n} + D_{3,n}$. Using (\ref{r-m-bound1}) and (\ref{r-m-bound2}), we have
\begin{align}
E_{\bZ}[D_{1,n}^{r}] &\lesssim \Sigma_{j,j}^{r}(\bm{0})n^{r/2}, \label{r-m-D1}\\
E_{\bZ}\left[D_{2,n}^{r}\right] &\lesssim \left\{\sum_{k=1}^{n-1}(n-k)^{2}E_{\bZ}\left[h_{1,n}^{2}(\bZ_{1})\right]\right\}^{r/2} \nonumber  \\
& \lesssim \|f\|_{\infty}^{r/2}\left(\int |\Sigma_{j,j}(\bs)|^{2}d\bs\right)^{r/2}n^{3r/2}\lambda_{n}^{-rd},\label{r-m-D2}.
\end{align}
Since 
\[
E_{\bZ}[U_{k} \mid \bZ_{1},\hdots, \bZ_{k-1}] = 0, \quad k=2,\dots,n,
\]
the sequence $\left\{\sum_{j=1}^{i}U_{j}, \mathcal{F}_{i}^{\bZ}\right\}_{i=1}^{n}$ is a martingale, where $\mathcal{F}_{i}^{\bZ} = \sigma\left(\bZ_{1},\hdots, \bZ_{i}\right)$ for $i=1,\dots,n$. Thus, applying Rosenthal's inequality (Theorem 2.12 in \cite{HaHe80}) 
and noting that  $U_{i}$ is the sum of $(i-1)$ i.i.d. random variables conditionally on $\bZ_{i}$, 
we have
\begin{align*}
E_{\bZ}[D_{3,n}^{r}] &\lesssim \left\{E_{\bZ}\left[\left(\sum_{k=2}^{n}E_{\bZ}[U_{k}^{2}\mid\mathcal{F}_{k-1}^{\bZ}]\right)^{r/2}\right] + \sum_{k=2}^{n}E_{\bZ}[U_{k}^{r}]\right\}.
\end{align*}
Observe that
\begin{align}
E_{\bZ}\left[\left(\sum_{k=2}^{n}E_{\bZ}[U_{k}^{2}\mid\mathcal{F}_{k-1}^{\bZ}]\right)^{r/2}\right] &\leq E_{\bZ}\left[n^{r/2-1}\sum_{k=2}^{n}\left(E_{\bZ}[U_{k}^{2}\mid\mathcal{F}_{k-1}^{\bZ}]\right)^{r/2}\right] \nonumber \\
&\leq n^{r/2-1}E_{\bZ}\left[\sum_{k=2}^{n}\left(E_{\bZ}[U_{k}^{r}\mid\mathcal{F}_{k-1}^{\bZ}]\right)\right] \nonumber  \\
&= n^{r/2-1}\sum_{k=2}^{n}E_{\bZ}[U_{k}^{r}]. \label{r-m-bound12-D3}
\end{align}
Moreover, 
\begin{align}
&\sum_{k=2}^{n}E_{\bZ}[U_{k}^{r}] \lesssim \sum_{k=2}^{n}E_{\bZ}\left[E_{\bZ}\left[(U_{k} - E_{\bZ}[U_{k} \mid \bZ_{k}])^{r}\mid\bZ_{k}\right] + \left(E_{\bZ}[U_{k}\mid\bZ_{k}]\right)^{r}\right] \nonumber \\
& \lesssim \sum_{k=2}^{n}E_{\bZ}\left[\left\{(k-1)E_{\bZ}\left[(h_{n}(\bZ_{k}, \bZ_{1}) - h_{1,n}(\bZ_{1}))^{2} \mid \bZ_{k}\right]\right\}^{r/2} \right. \nonumber \\ 
&\left. \quad \quad + \left((k-1)(h_{1,n}(\bZ_{k}) - E_{\bZ}[h_{1,n}(\bZ_{1})])\right)^{r}\right] \nonumber \\
&\lesssim \left\{\|f\|_{\infty}^{r/2}\left(\int |\Sigma_{j,j}(\bs)|^{2}d\bs\right)^{r/2}\lambda_{n}^{-rd/2}\sum_{k=2}^{n}(k-1)^{r/2} \right. \nonumber \\ 
& \left. \qquad \qquad  + \|f\|_{\infty}^{r}\left(\int |\Sigma_{j,j}(\bs)|d\bs\right)^{r}\lambda_{n}^{-rd}\sum_{k=2}^{n}(k-1)^{r}\right\} \nonumber \\
&\lesssim \left\{\|f\|_{\infty}^{r/2}\left(\int |\Sigma_{j,j}(\bs)|^{2}d\bs\right)^{r/2}\lambda_{n}^{-rd/2}n^{r/2+1} + \|f\|_{\infty}^{r}\left(\int |\Sigma_{j,j}(\bs)|d\bs\right)^{r}\lambda_{n}^{-rd}n^{r+1}\right\}. \label{r-m-bound2-D3}
\end{align}
Combining (\ref{r-m-bound12-D3}) and (\ref{r-m-bound2-D3}), we have 
\begin{align}
&E_{\bZ}[D_{3,n}^{r}] \nonumber \\ 
&\quad \lesssim \left(\|f\|_{\infty}^{r/2}\left(\int |\Sigma_{j,j}(\bs)|^{2}d\bs\right)^{r/2} + \|f\|_{\infty}^{r}\left(\int |\Sigma_{j,j}(\bs)|d\bs\right)^{r}\right)\left(n^{r}\lambda_{n}^{-rd/2} + n^{3r/2}\lambda_{n}^{-rd}\right).\label{r-m-D3}
\end{align}

Now, from (\ref{r-m-D1}), (\ref{r-m-D2}), and (\ref{r-m-D3}), we have 
\begin{align}
\label{r-sig- BC-bound}
&{E_{\bZ}\left[(\sigma_{j,n}^{2} - E_{\bZ}[\sigma_{j,n}^{2}])^{r}\right] \over (n^{2}\lambda_{n}^{-d})^{r}} \nonumber \\ 
&\lesssim \left(\Sigma_{j,j}^{r}(\bm{0}) + \|f\|_{\infty}^{r/2}\left(\int |\Sigma_{j,j}(\bs)|^{2}d\bs\right)^{r/2} + \|f\|_{\infty}^{r}\left(\int |\Sigma_{j,j}(\bs)|d\bs\right)^{r}\right){n^{r/2} + n^{r}\lambda_{n}^{-rd/2} + n^{3r/2}\lambda_{n}^{-rd} \over n^{2r}\lambda_{n}^{-rd}} \nonumber \\
&= \left(\Sigma_{j,j}^{r}(\bm{0}) + \|f\|_{\infty}^{r/2}\left(\int |\Sigma_{j,j}(\bs)|^{2}d\bs\right)^{r/2} + \|f\|_{\infty}^{r}\left(\int |\Sigma_{j,j}(\bs)|d\bs\right)^{r}\right) \nonumber  \\
&\quad \times \left({1 \over n^{r/2}(n\lambda_{n}^{-d})^{r}} + {1 \over n^{r/2}(n\lambda_{n}^{-d})^{r/2}}+ {1 \over n^{r/2}}\right). 
\end{align}
Set $r$ sufficiently large such that $r/2 - \alpha- 3/2 >0$. Since $\|f\|_{\infty}<\infty$ from Condition (ii) in Assumption \ref{ass: design} and 
\[
\max_{1 \leq j \leq p} \left [ \Sigma_{j,j}(\bm{0}) \bigvee \int \left(|\Sigma_{j,j}(\bs)| + |\Sigma_{j,j}(\bs)|^{2}\right)d\bs \right]= O(1)\ 
\]
from Condition (\ref{HDGA_cov_ass}) in Assumption \ref{ass: design}, the estimate (\ref{r-sig- BC-bound}) yields that 
\begin{align*}
&\sum_{n=1}^{\infty}P_{\bZ}\left(\left\|\bm{\sigma}_{n}^{2} - E_{\bZ}[\bm{\sigma}_{n}^{2}]\right\|_{\infty} > n^{2-{1 \over 2r}}\lambda_{n}^{-d}\right) \\
&\leq \sum_{n=1}^{\infty}\sum_{j=1}^{p}P_{\bZ}\left(\left|\sigma_{j,n}^{2} - E_{\bZ}[\sigma_{j,n}^{2}]\right| > n^{2-{1 \over 2r}}\lambda_{n}^{-d}\right) \\
&\leq \sum_{n=1}^{\infty}\sum_{j=1}^{p}{E_{\bZ}\left[(\sigma_{j,n}^{2} - E_{\bZ}[\sigma_{j,n}^{2}])^{r}\right] \over (n^{2}\lambda_{n}^{-d})^{r}}n^{1/2}\\
&\lesssim \max_{1 \leq j \leq p}\left(\Sigma_{j,j}^{r}(\bm{0}) + \|f\|_{\infty}^{r/2}\left(\int |\Sigma_{j,j}(\bs)|^{2}d\bs\right)^{r/2} + \|f\|_{\infty}^{r}\left(\int |\Sigma_{j,j}(\bs)|d\bs\right)^{r}\right)\sum_{n=1}^{\infty}{pn^{1/2} \over n^{r/2}} \\
&\lesssim \sum_{n=1}^{\infty}{1 \over n^{r/2-\alpha-1/2}} =  \sum_{n=1}^{\infty}{1 \over n^{1 + (r/2-\alpha-3/2)}}<\infty.
\end{align*}
Combining the Borel-Cantelli lemma, we have 
\begin{align*}
\left\|\bm{\sigma}_{n}^{2} - E_{\bZ}[\bm{\sigma}_{n}^{2}]\right\|_{\infty} = O\left(n^{2-{1 \over 2r}}\lambda_{n}^{-d}\right)\quad \text{$P_{\bZ}$-a.s.}
\end{align*}
for any $r/2 - \alpha- 3/2 >0$. This completes the proof. 
\end{proof}

\subsection{Proof of Theorem \ref{Thm: GA_hyperrec}}\label{pf-Thm4-1}

The proof of the theorem relies on the blocking argument for $\beta$-mixing sequences; cf. \cite{Yu94}. The following lemma plays a crucial role in the blocking argument. 

\begin{lemma}\label{lem: indep_lemma}
Let $m \in \mathbb{N}$ and let $Q$ be a probability measure on a product space $(\prod_{i=1}^{m}\Omega_{i}, \prod_{i=1}^{m}\Sigma_{i})$ with marginal measures $Q_{i}$ on $(\Omega_{i}, \Sigma_{i})$. Suppose that $h$ is a bounded measurable function on the product probability space such that $|h|\leq M_{h}<\infty$.
For $1 \le a \le b \le m$, let $Q_{a}^{b}$  be the marginal measure on $(\prod_{i=a}^{b}\Omega_{i}, \prod_{i=a}^{b}\Sigma_{i})$. For a given $\tau > 0$, suppose that, for all $1 \leq k \leq m-1$, 
\begin{align}\label{prod_bound}
\|Q - Q_{1}^{k}\times Q_{k+1}^{m}\|_{\mathrm{TV}} \leq 2\tau,
\end{align}
where $Q_{1}^{k}\times Q_{k+1}^{m}$ is the product measure and $\| \cdot \|_{\mathrm{TV}}$ is the total variation. Then $| Qh - Ph | \leq 2M_{h}(m-1)\tau$, where $P=\prod_{i=1}^{m}Q_{i}$, $Qh=\int h dQ$ and $Ph=\int h dP$. 
\end{lemma}

\begin{proof}[Proof of Lemma \ref{lem: indep_lemma}]
See Lemma 2 in \cite{Eb84}; see also Corollary 2.7 in \cite{Yu94}.
\end{proof}

\begin{proof}[Proof of Theorem \ref{Thm: GA_hyperrec}]

We shall omit ``$P_{\bZ}$-a.s." throughout the proof for the sake of notational simplicity, having in mind that we are picking any realization of  $\{\bZ_{i}\}_{i \geq 1}$ that satisfies the conclusions of Lemmas \ref{n summands} and \ref{variance-rate}. The same comment applies to the other proofs without further mentioning. 
We divide the proof into several steps.

\underline{Step 1} (Reduction to independence). 
Recall the large blocks and small blocks that appeared in Section \ref{sec: preliminaries}, and  $S_{n}(\bm{\ell};\bm{\epsilon}) = \sum_{i:\bs_{i}\in \Gamma_{n}(\bm{\ell};\bm{\epsilon}) \cap R_{n}}\bX(\bs_{i})$.
For each $\bm{\epsilon} \in \{1,2\}^{d}$, let $\{\breve{S}_{n}(\bm{\ell};\bm{\epsilon}): \bm{\ell} \in L_{n} \}$ be a sequence of independent random vectors in $\R^{p}$ under $P_{\cdot \mid \bZ}$ such that 
\[
\breve{S}_{n}(\bm{\ell};\bm{\epsilon}) \stackrel{d}{=} S_{n}(\bm{\ell};\bm{\epsilon}), \ \text{under $P_{\cdot \mid \bZ}$}, \  \bm{\ell} \in L_{n}. 
\]
Define $\breve{S}_{1,n} = \sum_{\bm{\ell} \in L_{n}}\breve{S}_{n}(\bm{\ell};\bm{\epsilon}_{0}) = (\breve{S}^{(1)}_{1,n},\dots,\breve{S}^{(p)}_{1,n})'$ and for $\bm{\epsilon} \neq \bm{\epsilon}_{0}$, define $\breve{S}_{2,n}(\bm{\epsilon}) = \sum_{\bm{\ell}\in L_{1,n}}\breve{S}_{n}(\bm{\ell};\bm{\epsilon})$ and $\breve{S}_{3,n}(\bm{\epsilon}) = \sum_{\bm{\ell}\in L_{2,n}}\breve{S}_{n}(\bm{\ell};\bm{\epsilon})$. 

We begin with verifying the following results:
\begin{align}
&\sup_{A \in \mathcal{A}}\left|P_{\cdot \mid \bZ}\left(S_{1,n} \in A \right) - P_{\cdot \mid \bZ}\left(\breve{S}_{1,n} \in A\right)\right| \leq C\left({\lambda_{n} \over \lambda_{1,n}}\right)^{d}\beta(\lambda_{2,n}; \lambda_{n}^{d}),  \label{indep_block1} \\
&\sup_{t >0}\left|P_{\cdot \mid \bZ}\left(\|S_{2,n}(\bm{\epsilon})\|_{\infty} >t\right) - P_{\cdot \mid \bZ}\left(\|\breve{S}_{2,n}(\bm{\epsilon})\|_{\infty} >t\right)\right| \leq C\left({\lambda_{n} \over \lambda_{1,n}}\right)^{d}\beta(\lambda_{2,n}; \lambda_{n}^{d}) ,  \label{indep_block2} \\
&\sup_{t >0}\left|P_{\cdot \mid \bZ}\left(\|S_{3,n}(\bm{\epsilon})\|_{\infty} >t\right) - P_{\cdot \mid \bZ}\left(\|\breve{S}_{3,n}(\bm{\epsilon})\|_{\infty} >t\right)\right| \leq C\left({\lambda_{n} \over \lambda_{1,n}}\right)^{d-1}\!\!\!\! \beta(\lambda_{2,n}; \lambda_{n}^{d}).  \label{indep_block3}
\end{align}

Recall that $[\![L_{n}]\!] = O((\lambda_{n}/\lambda_{3,n})^{d}) \lesssim (\lambda_{n}/\lambda_{1,n})^{d}$. For $\bm{\epsilon}\in \{1,2\}^{d}$ and $\bm{\ell}_{1}, \bm{\ell}_{2} \in L_{n}$ with $\bm{\ell}_{1} \neq \bm{\ell}_{2}$, let 
\begin{align*}
\mathcal{J}_{1}(\bm{\epsilon}) &= \{1 \leq i_{1} \leq n: \bs_{i_{1}} \in \Gamma_{n}(\bm{\ell}_{1};\bm{\epsilon})\},\ \mathcal{J}_{2}(\bm{\epsilon}) = \{1 \leq i_{2} \leq n: \bs_{i_{2}} \in \Gamma_{n}(\bm{\ell}_{2};\bm{\epsilon})\}. 
\end{align*} 
For any $\bs_{i_k}=(s_{1,i_k},\dots,s_{d,i_k})$, $k=1,2$ such that $i_{1} \in \mathcal{J}_{1}(\bm{\epsilon})$ and $i_{2} \in \mathcal{J}_{2}(\bm{\epsilon})$, we have $\max_{1 \leq u \leq d}|s_{u,i_{1}} - s_{u,i_{2}}| \geq \lambda_{2,n}$ from the definition of $\Gamma(\bm{\ell};\bm{\epsilon})$. This implies $|\bs_{i_1}-\bs_{i_2}|\geq \lambda_{2,n}$.

For any $\bm{\epsilon} \in \{1,2\}^{d}$, let $S_{n}(\bm{\ell}_{1};\bm{\epsilon}), \dots, S_{n}(\bm{\ell}_{[\![L_{n}]\!]};\bm{\epsilon})$ be an arrangement of $\{S_{n}(\bm{\ell};\bm{\epsilon}): \bm{\ell} \in L_{n}\}$.
Let $P_{\cdot \mid \bZ}^{(a)}$ be the marginal distribution of $S_{n}(\bm{\ell}_{a};\bm{\epsilon})$ and let $P_{\cdot \mid \bZ}^{(a:b)}$ be the joint distribution of $\{S_{n}(\bm{\ell}_{k};\bm{\epsilon}): a \leq k \leq b\}$. The $\beta$-mixing property of $\bm{X}$ implies that for $1 \leq k \leq [\![L_{n}]\!]-1\!$,
\[
\left \|P_{\cdot \mid \bZ} - P_{\cdot \mid \bZ}^{(1:k)} \times P_{\cdot \mid \bZ}^{(k+1:[\![L_{n}]\!])} \right\|_{\mathrm{TV}} \lesssim \beta(\lambda_{2,n};\lambda_{n}^{d}).
\]
The inequality does not depend on the arrangement of $\{S_{n}(\bm{\ell};\bm{\epsilon}): \bm{\ell} \in L_{n}\}$. 
Hence, the assumption (\ref{prod_bound})  in Lemma \ref{lem: indep_lemma} is satisfied for $\{S_{n}(\bm{\ell};\bm{\epsilon}): \bm{\ell} \in L_{n}\}$  with $\tau \sim \beta(\lambda_{2,n};\lambda_{n}^{d})$ and $m \lesssim (\lambda_{n}/\lambda_{1,n})^{d}$. 
Combining Lemma \ref{lem: indep_lemma} and the boundary condition on $R_{n}$, we obtain (\ref{indep_block1})-(\ref{indep_block3}).

Pick any rectangle $A = \prod_{j=1}^{p}[a_{j},b_{j}]$ with $a = (a_{1},\dots,a_{p})'$ and $b = (b_{1},\dots,b_{p})'$.
For every $t>0$, (\ref{indep_block1})-(\ref{indep_block3}) yield that 
\begin{align}
&P_{\cdot \mid \bZ}\left(r_{n}^{-1/2}S_n \in A\right) = P_{\cdot \mid \bZ}\left(\left\{-r_{n}^{-1/2}S_{n} \leq -a\right\}\cap \left\{r_{n}^{-1/2}S_{n} \leq b\right\}\right) \nonumber \\
&\leq P_{\cdot \mid \bZ}\! \left(\left\{\! -r_{n}^{-1/2}S_n \leq -a \! \right\}\! \cap \! \left\{\! r_{n}^{-1/2}S_n \leq b\! \right\} \! \cap \! \left\{\!  r_{n}^{-1/2}\left\| {\textstyle \sum}_{\bm{\epsilon} \neq \bm{\epsilon}_{0}}S_{2,n}(\bm{\epsilon}) + S_{3,n}(\bm{\epsilon}) \right\|_{\infty} \leq 2(2^{d}-1)t \! \right\}\! \right) \nonumber  \\
&\quad + \sum_{\bm{\epsilon} \neq \bm{\epsilon}_{0}}P_{\cdot \mid \bZ}\left(r_{n}^{-1/2}\left\|S_{2,n}(\bm{\epsilon})\right\|_{\infty}> t\right) +  \sum_{\bm{\epsilon} \neq \bm{\epsilon}_{0}}P_{\cdot \mid \bZ}\left(r_{n}^{-1/2}\left\|S_{3,n}(\bm{\epsilon})\right\|_{\infty} > t\right) \nonumber \\
&\leq P_{\cdot \mid \bZ}\! \left(\! \left\{-r_{n}^{-1/2}S_{1,n} \leq -a + 2(2^{d}-1)t\right\}\cap \left\{r_{n}^{-1/2}S_{1,n} \leq b + 2(2^{d}-1)t\right\}\! \right) \! + \! 2C \! \left({\lambda_{n} \over \lambda_{1,n}}\right)^{d} \!\! \beta(\lambda_{2,n}; \lambda_{n}^{d}) \nonumber \\
&\quad + \sum_{\bm{\epsilon} \neq \bm{\epsilon}_{0}}P_{\cdot \mid \bZ}\left(r_{n}^{-1/2}\left\|\breve{S}_{2,n}(\bm{\epsilon}) \right\|_{\infty}>t\right) + \sum_{\bm{\epsilon} \neq \bm{\epsilon}_{0}}P_{\cdot \mid \bZ}\left(r_{n}^{-1/2}\left\|\breve{S}_{3,n}(\bm{\epsilon})\right\|_{\infty} > t\right) \nonumber \\
&\leq P_{\cdot \mid \bZ}\! \left( \! \left\{-r_{n}^{-1/2}\breve{S}_{1,n} \leq -a + 2(2^{d}-1)t\right\}\cap \left\{r_{n}^{-1/2}\breve{S}_{1,n} \leq b + 2(2^{d}-1)t\right\}\! \right) \! + 3C \! \left({\lambda_{n} \over \lambda_{1,n}}\right)^{d} \!\! \beta(\lambda_{2,n}; \lambda_{n}^{d}) \nonumber \\
&\quad + \sum_{\bm{\epsilon} \neq \bm{\epsilon}_{0}}P_{\cdot \mid \bZ}\left(r_{n}^{-1/2}\left\|\breve{S}_{2,n}(\bm{\epsilon}) \right\|_{\infty}>t\right) + \sum_{\bm{\epsilon} \neq \bm{\epsilon}_{0}}P_{\cdot \mid \bZ}\left(r_{n}^{-1/2}\left\|\breve{S}_{3,n}(\bm{\epsilon})\right\|_{\infty} > t\right) \label{asy_negligible} \\
&=: I + II + III + IV. \nonumber
\end{align}

Next, we will show below the asymptotic negligibility of $III$ and $IV$, i.e.,  
\begin{align}\label{asy-neg-S2-S3}
P_{\cdot \mid \bZ}\left(r_{n}^{-1/2}\left\|\breve{S}_{k,n}(\bm{\epsilon})\right\|_{\infty} > Cn^{-c'/2}\log^{-1/2} p\right) \lesssim n^{-c'/2}.
\end{align} 
for $k= 2,3$ and $\bm{\epsilon} \neq \bm{\epsilon}_{0}$.
Indeed, given (\ref{asy-neg-S2-S3}), (\ref{asy_negligible}) yields that
\begin{align}
P_{\cdot \mid \bZ}\left(r_{n}^{-1/2}S_n  \in A\right) &\leq P_{\cdot \mid \bZ}\! \left( \! \left\{-r_{n}^{-1/2}\breve{S}_{1,n} \leq -a + 2(2^{d}-1) Cn^{-c'/2}\log^{-1/2} p\right\} \right. \nonumber \\
&\left. \qquad \qquad \cap \left\{r_{n}^{-1/2}\breve{S}_{1,n} \leq b + 2(2^{d}-1)Cn^{-c'/2}\log^{-1/2} p\right\}\! \right) \nonumber \\ 
&\quad + 3C \! \left({\lambda_{n} \over \lambda_{1,n}}\right)^{d} \!\! \beta(\lambda_{2,n}; \lambda_{n}^{d}) + O(n^{-c'/2}) \label{S2-S3-asy-neg-ineq}
\end{align}
In Step 2 below, we will show that 
\begin{align}\label{S1-High-d-bound}
\sup_{A \in \mathcal{A}}\left|P_{\cdot \mid \bZ}\!\left(r_{n}^{-1/2}\breve{S}_{1,n}\in A\right)\! -\! P_{\cdot \mid \bZ}\!\left(\bm{V}_{n} \in A\right)\right| = O(n^{-c'/6}).
\end{align}
Together with (\ref{S2-S3-asy-neg-ineq}) and (\ref{S1-High-d-bound}), we have 
\begin{align*}
&P_{\cdot \mid \bZ}\left(r_{n}^{-1/2}S_n \in A\right)\\
&\leq P_{\cdot \mid \bZ}\! \left( \! \left\{\bm{V}_{n} \leq -a + 2(2^{d}-1) Cn^{-c'/2}\log^{-1/2} p\right\}\cap \left\{\bm{V}_{n} \leq b + 2(2^{d}-1)Cn^{-c'/2}\log^{-1/2} p\right\}\! \right) \nonumber \\ 
&\quad + 3C \! \left({\lambda_{n} \over \lambda_{1,n}}\right)^{d} \!\! \beta(\lambda_{2,n}; \lambda_{n}^{d}) + O(n^{-c'/6}).
\end{align*}
Combining Nazarov's inequality (see Lemma \ref{lem: Nazarov} ahead), we have 
\begin{align*}
P_{\cdot \mid \bZ}\left(r_{n}^{-1/2}S_n \in A\right) &\leq P_{\cdot \mid \bZ}\left(\bm{V}_{n} \in A\right) + O(n^{-c'/2}) + 3C\left({\lambda_{n} \over \lambda_{1,n}}\right)^{d}\beta(\lambda_{2,n}; \lambda_{n}^{d}) + O(n^{-c'/6}) \\
&=  P_{\cdot \mid \bZ}\left(\bm{V}_{n} \in A\right) + 3C\left({\lambda_{n} \over \lambda_{1,n}}\right)^{d}\beta(\lambda_{2,n}; \lambda_{n}^{d}) + O(n^{-c'/6}).
\end{align*}
Likewise, we have $P_{\cdot \mid \bZ}\left(r_{n}^{-1/2}S_n \in A\right) \geq P\left(\bm{V}_{n} \in A\right) - 3C\left({\lambda_{n} \over \lambda_{1,n}}\right)^{d}\beta(\lambda_{2,n}; \lambda_{n}^{d}) - O(n^{-c'/6})$.

Thus, (\ref{asy-neg-S2-S3}) implies that $III$ and $IV$ are asymptotically negligible, i.e., they do not contribute to the high-dimensional CLT for $\sqrt{\lambda_{n}^{d}}\overline{\bm{X}}_{n}$.

Now we move on to proving (\ref{asy-neg-S2-S3}). For every $\tau>0$, Markov's inequality yields that with $t = \tau^{-1}E_{\cdot \mid \bZ}[r_{n}^{-1/2}\|\breve{S}_{2,n}(\bm{\epsilon})\|_{\infty}]$, we have $III \leq \tau$. Then, to prove (\ref{asy-neg-S2-S3}), it suffices to bound the magnitude of $E_{\cdot \mid \bZ}[r_{n}^{-1/2}\|\breve{S}_{2,n}(\bm{\epsilon})\|_{\infty}]$. Note that $\breve{S}_{n}(\bm{\ell};\bm{\epsilon}) = (\breve{S}_{n}^{(1)}(\bm{\ell};\bm{\epsilon}),\dots,\breve{S}_{n}^{(p)}(\bm{\ell};\bm{\epsilon}))'$, $\bm{\ell} \in L_{1,n}$ are independent. We will show that 
\begin{align}
E_{\cdot \mid \bZ}\left[\max_{\bm{\ell} \in L_{1,n}}\|\breve{S}_{n}(\bm{\ell};\bm{\epsilon})\|_{\infty}^{2}\right] &\lesssim \lambda_{1,n}^{2d-2}\lambda_{2,n}^{2}n^{2}\lambda_{n}^{-2d}(\log n)^{2}(\log p)^{2}D_{n}^{2}, \label{mean_ineq_S2} \\
\Var_{\cdot \mid \bZ}(\breve{S}_{n}^{(j)}(\bm{\ell};\bm{\epsilon})) & \lesssim \overline{\beta}_{q}\lambda_{1,n}^{d-1}\lambda_{2,n}(\log n + n\lambda_{n}^{-d})^{2}(\log n)^{2} (\log p)^{2} D_{n}^{2} \label{Var_ineq_S2}
\end{align} 
for $\bm{\ell} \in L_{1,n}$, $\bm{\epsilon} \neq \bm{\epsilon}_{0}$, $1 \leq j \leq p$. 

Consider first (\ref{mean_ineq_S2}). Define $M = \max_{\bm{\ell} \in L_{1,n}}\|\breve{S}_{n}(\bm{\ell};\bm{\epsilon})\|_{\infty}$. By Lemma \ref{n summands}, we have 
\begin{align*}
E_{\cdot \mid \bZ}[M^{2}]  &\lesssim (\lambda_{1,n}^{d-1}\lambda_{2,n}n\lambda_{n}^{-d})^{2}E_{\cdot \mid \bZ}\left[\max_{1 \leq i \leq n}\|\bm{X}(\bs_{i})\|_{\infty}^{2}\right] \\
&\leq 2!\lambda_{1,n}^{2d-2}\lambda_{2,n}^{2}n^{2}\lambda_{n}^{-2d}\left\|\max_{1 \leq i \leq n}\|\bm{X}(\bs_{i})\|_{\infty}\right\|_{\psi_{1}}^{2} \\
&\lesssim \lambda_{1,n}^{2d-2}\lambda_{2,n}^{2}n^{2}\lambda_{n}^{-2d}(\log n)^{2}\left\|\|\bm{X}(\bs)\|_{\infty}\right\|_{\psi_{1}}^{2} \\
&\lesssim \lambda_{1,n}^{2d-2}\lambda_{2,n}^{2}n^{2}\lambda_{n}^{-2d}(\log n)^{2}(\log p)^{2}\!\!\max_{1 \leq j \leq p}\left\|X_{j}(\bs)\right\|_{\psi_{1}}^{2} \\
&\leq 2!\lambda_{1,n}^{2d-2}\lambda_{2,n}^{2}n^{2}\lambda_{n}^{-2d}(\log n)^{2}(\log p)^{2}D_{n}^{2},
\end{align*}
where we used an inequality for the Orlicz norm (see \cite{vaWe96}, p. 95) to obtain the second inequality.

Next consider (\ref{Var_ineq_S2}). Let $\mathcal{I}_{n}(\bm{\ell}) = \{\bm{i} \in \mathbb{Z}^{d} : (\bm{i} + (0,1]^{d}) \cap \Gamma_{n}(\bm{\ell};\bm{\epsilon})\}$ and define $H^{(j)}_{n}(\bm{i}) = \sum_{k: \bs_{k} \in (\bm{i}+(0,1]^{d})\cap R_{n}}X_{j}(\bs_{k})$. Then by Lemma \ref{n summands} and the mixing property of $\bX$, we have 
\begin{align*}
&\Var_{\cdot \mid \bZ}(\breve{S}_{n}^{(j)}(\bm{\ell};\bm{\epsilon})) = E_{\cdot \mid \bZ}\left[\left(\sum_{\bm{i} \in \mathcal{I}_{n}(\bm{\ell})}H_{n}^{(j)}(\bm{i})\right)^{2}\right]\\
& = \sum_{\bm{i} \in \mathcal{I}_{n}(\bm{\ell})}\!\!\!E_{\cdot \mid \bZ}\left[\left(H_{n}^{(j)}(\bm{i})\right)^{2}\right] +\!\!\! \sum_{\bm{i}_{1} \neq \bm{i}_{2}, \bm{i}_{1}, \bm{i}_{2} \in \mathcal{I}_{n}(\bm{\ell})}\!\!\!\!\!\!\!\!\!\!\!\!E_{\cdot \mid \bZ}\left[H_{n}^{(j)}(\bm{i}_{1})H_{n}^{(j)}(\bm{i}_{2})\right]\\
& \lesssim \lambda_{1,n}^{d-1}\lambda_{2,n}(\log n + n\lambda_{n}^{-d})^{2}E_{\cdot \mid \bZ}\left[\max_{1 \leq i \leq n}\|\bm{X}(\bs_{i})\|_{\infty}^{2}\right] \\
&\quad  + \sum_{\bm{i}_{1} \neq \bm{i}_{2}, \bm{i}_{1}, \bm{i}_{2} \in \mathcal{I}_{n}(\bm{\ell})}\!\!\!\!\!\!\!\!\!\!\!\!(\log n + n\lambda_{n}^{-d})E_{\cdot \mid \bZ}\!\! \left[\max_{\bs_{i_{k}} \in (\bm{i}_{1} + (0,1]^{d})\cap R_{n}}\|\bm{X}(\bs_{k})\|_{\infty}\max_{\bs_{\ell} \in (\bm{i}_{2} + (0,1]^{d})\cap R_{n}}\|\bm{X}(\bs_{\ell})\|_{\infty}\right]\\
&\lesssim  \lambda_{1,n}^{d-1}\lambda_{2,n}(\log n + n\lambda_{n}^{-d})^{2}\!\! \left(\!\! (\log n)^{2}(\log p)^{2}D_{n}^{2}  + \sum_{k=1}^{\lambda_{1,n}}k^{d-1}\beta^{1-{2 \over q}}(k;2)E_{\cdot \mid \bZ}\!\! \left[\max_{1 \leq i \leq n}\|\bm{X}(\bs_{i})\|_{\infty}^{q}\right]^{{2 \over q}}\right)\\
& \lesssim \overline{\beta}_{q}\lambda_{1,n}^{d-1}\lambda_{2,n}(\log n + n\lambda_{n}^{-d})^{2}(\log n)^{2}(\log p)^{2}D_{n}^{2},
\end{align*}
where the second inequality follows from similar arguments to prove (\ref{mean_ineq_S2}), and  the third inequality follows from applying similar arguments to the proof of (\ref{mean_ineq_S2}) and using the inequality for $\alpha$- (i.e. $\beta$-) mixing sequence (Proposition 2.5 in \cite{FaYa03}). 

Then, combining these estimates with Lemma \ref{lem: maximal inequality} ahead, we have 
\begin{align*}
&E_{\cdot \mid \bZ}\left[r_{n}^{-1/2}\|\breve{S}_{2,n}(\bm{\epsilon})\|_{\infty}\right]\\
&\leq \! K \! \left(\! D_{n}(\log n + n\lambda_{n}^{-d})\sqrt{{\overline{\beta}_{q}\lambda_{1,n}^{d-1}\lambda_{2,n}(\lambda_{n}/\lambda_{1,n})^{d}(\log n)^{2}(\log p)^{3} \over n^{2}\lambda_{n}^{-d}}} + \lambda_{1,n}^{d-1}\lambda_{2,n}\lambda_{n}^{-d/2}(\log n)(\log p)^{2}D_{n} \!\!\right)\\
&= K \left(D_{n}(\log n + n\lambda_{n}^{-d})\sqrt{{\overline{\beta}_{q}\lambda_{2,n} \over n^{2}\lambda_{n}^{-2d}\lambda_{1,n}}}(\log n)(\log p)^{3/2} + \lambda_{1,n}^{d-1}\lambda_{2,n}\lambda_{n}^{-d/2}(\log n)(\log p)^{2}D_{n}\right),
\end{align*}
where $K$ is universal. The left-hand side is bounded by $Cn^{-c'}\log^{-1/2} p$ under Assumption \ref{ass: design}. Thus, taking $\tau = n^{-c'/2}$, we have  $t \leq Cn^{-c'/2}\log^{-1/2}p$, which yields (\ref{asy-neg-S2-S3}) for $k = 2$, $\bm{\epsilon} \neq \bm{\epsilon}_{0}$. Likewise, 
since $\breve{S}_{n}(\bm{\ell};\bm{\epsilon})$, $\bm{\ell} \in L_{2,n}$ are independent for each $\bm{\epsilon} \neq \bm{\epsilon}_{0}$ and 
\begin{align*}
E_{\cdot \mid \bZ}\left[\max_{\bm{\ell} \in L_{2,n}}\|\breve{S}_{n}(\bm{\ell};\bm{\epsilon})\|_{\infty}^{2}\right] \lesssim \lambda_{1,n}^{2d-2}\lambda_{2,n}^{2}n^{2}\lambda_{n}^{-2d}(\log n)^{2}(\log p)^{2}D_{n}^{2},\\ 
\Var_{\cdot \mid \bZ}(\breve{S}_{n}^{(j)}(\bm{\ell};\bm{\epsilon})) \lesssim \overline{\beta}_{q}\lambda_{1,n}^{d-1}\lambda_{2,n}(\log n + n\lambda_{n}^{-d})^{2}(\log n)^{2}(\log p)^{2}D_{n}^{2}
\end{align*}
for $\bm{\ell} \in L_{2,n}$, $\bm{\epsilon} \neq \bm{\epsilon}_{0}$, $1 \leq j \leq p$, Lemma \ref{lem: maximal inequality} ahead and the boundary condition on $R_{n}$ yield that 
\begin{align*}
&E_{\cdot \mid \bZ}\left[r_{n}^{-1/2}\|\breve{S}_{3,n}(\bm{\epsilon})\|_{\infty}\right]\\
&\leq \! K \!\! \left( \!\! D_{n}(\log n + n\lambda_{n}^{-d})\sqrt{{\overline{\beta}_{q}\lambda_{1,n}^{d-1}\lambda_{2,n}(\lambda_{n}/\lambda_{1,n})^{d-1}(\log n)^{2}(\log p)^{3} \over n^{2}\lambda_{n}^{-d}}} + \lambda_{1,n}^{d-1}\lambda_{2,n}\lambda_{n}^{-d/2}(\log n)(\log p)^{2}D_{n}\!\! \right)\\
&= K \left(D_{n}(\log n + n\lambda_{n}^{-d})\sqrt{{\overline{\beta}_{q}\lambda_{2,n} \over n^{2}\lambda_{n}^{-2d+1}}}(\log n)(\log p)^{3/2} + \lambda_{1,n}^{d-1}\lambda_{2,n}\lambda_{n}^{-d/2}(\log n)(\log p)^{2}D_{n}\right),
\end{align*}
where $K$ is universal. Then, under Assumption \ref{ass: design}, the left-hand side is bounded by $Cn^{-c'}\log^{-1/2} p$, which yields (\ref{asy-neg-S2-S3}) for $k=3$, $\bm{\epsilon} \neq \bm{\epsilon}_{0}$. Conclude that $IV \lesssim n^{-c'/2}$ for some $t \leq Cn^{-c'/2}\log^{-1/2}p$. 

\medskip

\underline{Step 2} (High-dimensional CLT applied to the sum of independent blocks).
Recall that $\bm{V}_{n} = (V_{1,n},\dots, V_{p,n})'$ be a centered Gaussian random vector under $P_{\cdot \mid \bZ}$ with covariance 
\[
E_{\cdot \mid \bZ}\left[\bm{V}_{n}\bm{V}'_{n}\right] = {1 \over n^{2}\lambda_{n}^{-d}}\sum_{\bm{\ell} \in L_{n}}E_{\cdot \mid \bZ}\left[S_{n}(\bm{\ell};\bm{\epsilon}_{0})S_{n}(\bm{\ell};\bm{\epsilon}_{0})'\right].
\]
We wish to show that
\begin{align*}
\sup_{A \in \mathcal{A}}\left|P_{\cdot \mid \bZ}\!\left(r_{n}^{-1/2}\sum_{\bm{\ell} \in L_{n}}\!\!\breve{S}_{n}(\bm{\ell};\bm{\epsilon}_{0}) \in A\right)\! -\! P_{\cdot \mid \bZ}\!\left(\bm{V}_{n} \in A\right)\right| = O(n^{-c'/6}).
\end{align*}
Since $\breve{S}_{n}(\bm{\ell};\bm{\epsilon})$, $\bm{\ell} \in L_{n}$ are independent, we may apply Proposition 2.1 in \cite{ChChKa17}. We wish to verify Conditions (M.1), (M.2) and (E.2) of the proposition to this case. We first verify Condition (M.1).  Lemma \ref{variance-rate} shows that there exist constants $0<c_{1} \leq c_{2}<\infty$ such that 
\begin{align}\label{bound-S-block}
c_{1} \leq {1 \over [\![L_{n}]\!]}\sum_{\bm{\ell} \in L_{n}}\Var_{\cdot \mid \bZ}\left(\breve{S}_{n}^{(j)}(\bm{\ell};\bm{\epsilon}_{0})/\sqrt{n^{2}\lambda_{n}^{-d}[\![L_{n}]\!]^{-1}}\right) \leq c_{2}
\end{align}
uniformly over $1 \leq j \leq p$. This implies Condition (M.1) of Proposition 2.1 in \cite{ChChKa17}. 

Next we verify Conditions (M.2) and (E.2). Recall that $[\![L_{n}]\!] \sim (\lambda_{n}/\lambda_{1,n})^{d}$. We  have that  for $k=3,4$,
\begin{align*}
&E_{\cdot \mid \bZ}\left[\left\|\breve{S}_{n}(\bm{\ell};\bm{\epsilon}_{0})/\sqrt{r_{n}[\![L_{n}]\!]^{-1}}\right\|_{\infty}^{k}\right] \\
&\leq (n\lambda_{n}^{-d}\lambda_{1,n}^{d})^{k}E_{\cdot \mid \bZ}\left[\max_{1 \leq i \leq n}\|\bm{X}(\bs_{i})\|_{\infty}^{k}\right]/(n\lambda_{n}^{-d}\lambda_{1,n}^{d/2})^{k}\\
&\lesssim  (n\lambda_{n}^{-d}\lambda_{1,n}^{d})^{k}\left((\log n)(\log p)\max_{1 \leq i \leq n}\max_{1\leq j \leq p}\left\|X_{j}(\bs_{i})\right\|_{\psi_{1}}\right)^{k}/(n\lambda_{n}^{-d}\lambda_{1,n}^{d/2})^{k}\\
& \leq (n\lambda_{n}^{-d}\lambda_{1,n}^{d})^{k}(\log n)^{k}(\log p)^{k}D_{n}^{k}/(n\lambda_{n}^{-d}\lambda_{1,n}^{d/2})^{k} = \lambda_{1,n}^{kd/2}(\log n)^{k}(\log p)^{k}D_{n}^{k}.
\end{align*}
Thus, Conditions (M.2) and (E.2) with $q = 4$ of Proposition 2.1 in \cite{ChChKa17} are verified with $B_{n} = \lambda_{1,n}^{3d/2}(\log n)^{3}(\log p)^{3}D_{n}^{3}$ in their notation, which leads to the assertion of this step. 

\medskip

\underline{Step 3} (Conclusion). 
Pick any rectangle $A = \prod_{j=1}^{p}[a_{j},b_{j}]$ with $a = (a_{1},\dots, a_{p})'$ and $b = (b_{1},\dots, b_{p})'$. Applying similar arguments to Step 1, we have 
\begin{align*}
&P_{\cdot \mid \bZ}\left(\sqrt{\lambda_{n}^{d}}\overline{\bY}_{n}\in A\right) = P_{\cdot \mid \bZ}\left(\left\{-\sqrt{\lambda_{n}^{d}}\overline{\bY}_{n} \leq -a\right\} \cap \left\{\sqrt{\lambda_{n}^{d}}\overline{\bY}_{n} \leq b \right\}  \right)\\  
&\leq P_{\cdot \mid \bZ}\left(\left\{-\sqrt{\lambda_{n}^{d}}\overline{\bY}_{n} \leq -a\right\} \cap \left\{\sqrt{\lambda_{n}^{d}}\overline{\bY}_{n} \leq b \right\} \cap \left\{r_{n}^{-1/2}\left\|\sum_{i=1}^{n}\bm{\Upsilon}(\bs_{i})\right\|_{\infty} \leq n^{-\zeta}(\log p)^{-1/2}\right\} \right) \\
&\quad + P_{\cdot \mid \bZ}\left(r_{n}^{-1/2}\left\|\sum_{i=1}^{n}\bm{\Upsilon}(\bs_{i})\right\|_{\infty} > n^{-\zeta}(\log p)^{-1/2}\right)\\
&\leq P_{\cdot \mid \bZ}\left(\left\{-r_{n}^{-1/2}S_{n} \leq -a + n^{-\zeta}(\log p)^{-1/2}\right\} \cap \left\{r_{n}^{-1/2}S_{n}  \leq b + n^{-\zeta}(\log p)^{-1/2}\right\}\right) + O(n^{-\zeta})\\
&\leq P_{\cdot \mid \bZ}\!\left(\{-\bm{V}_{n} \leq -a + 2n^{-({c' \over 2}\wedge \zeta)}\log^{-1/2}p\}\cap \{\bm{V}_{n} \leq b + 2n^{-({c' \over 2} \wedge \zeta)}\log^{-1/2}p\}\right)\\
&\quad +  3C\left({\lambda_{n} \over \lambda_{1,n}}\right)^{d}\beta(\lambda_{2,n}; \lambda_{n}^{d}) + O(n^{-c'/6}) + O(n^{-\zeta}).
\end{align*}
Combining  Nazarov's inequality (see Lemma \ref{lem: Nazarov} ahead), we have 
\begin{align*}
&P_{\cdot \mid \bZ}\left(\sqrt{\lambda_{n}^{d}}\overline{\bY}_{n}\in A\right)\\
&\leq P_{\cdot \mid \bZ}\left(\bm{V}_{n} \in A\right) + O(n^{-({c' \over 2} \wedge \zeta)}) + 3C\left({\lambda_{n} \over \lambda_{1,n}}\right)^{d}\beta(\lambda_{2,n}; \lambda_{n}^{d}) + O(n^{-c'/6}) + O(n^{-\zeta})\\
&=  P_{\cdot \mid \bZ}\left(\bm{V}_{n} \in A\right) + 3C\left({\lambda_{n} \over \lambda_{1,n}}\right)^{d}\beta(\lambda_{2,n}; \lambda_{n}^{d}) + O(n^{-({c' \over 6}\wedge \zeta)}).
\end{align*}
Likewise, we have
\begin{align*}
P_{\cdot \mid \bZ}\left(\sqrt{\lambda_{n}^{d}}\overline{\bY}_{n} \in A\right) &\geq P\left(\bm{V}_{n} \in A\right) - 3C\left({\lambda_{n} \over \lambda_{1,n}}\right)^{d}\beta(\lambda_{2,n}; \lambda_{n}^{d}) - O(n^{-({c'\over 6}\wedge \zeta)}).
\end{align*}
Conclude that
\begin{align*}
\sup_{A \in \mathcal{A}}\left|P_{\cdot \mid \bZ}\left(\sqrt{\lambda_{n}^{d}}\overline{\bY}_{n} \in A\right) - P_{\cdot \mid \bZ}\left(\bm{V}_{n} \in A\right)\right| 
&\leq 3C\left({\lambda_{n} \over \lambda_{1,n}}\right)^{d}\beta(\lambda_{2,n}; \lambda_{n}^{d}) + O(n^{-({c'\over 6}\wedge \zeta)}).
\end{align*}
This completes the overall proof.
\end{proof}

\subsection{Proof of Theorem \ref{Thm: GA_hyperrec-m}}\label{pf-Thm4-2}

The proof is similar to that of Theorem \ref{Thm: GA_hyperrec}, so we only point out required modifications in each step. 

\underline{Step 1} (Reduction to independence). 
By Lemma 2.2.2 in \cite{vaWe96}, for $\bm{\epsilon} \neq \bm{\epsilon}_{0}$, we have
\begin{align*}
E_{\cdot \mid \bZ}\left[\max_{\bm{\ell} \in L_{1,n}}\|\breve{S}_{n}(\bm{\ell};\bm{\epsilon})\|_{\infty}^{2}\right] &\lesssim [\![L_{1,n}]\!]^{1/2}\max_{\ell \in L_{1,n}}\left(E_{\cdot \mid \bZ}\left[\|\breve{S}_{n}(\bm{\ell};\bm{\epsilon})\|_{\infty}^{4}\right]\right)^{1/2}\\
&\lesssim (\lambda_{n}^{d}\lambda_{1,n}^{-d})^{1/2}\max_{\ell \in L_{1,n}}\left(E_{\cdot \mid \bZ}\left[\|S_{n}(\bm{\ell};\bm{\epsilon})\|_{\infty}^{4}\right]\right)^{1/2}.
\end{align*}
Note that 
\begin{align*}
\|S_{n}(\bm{\ell};\bm{\epsilon})\|_{\infty}^{4} &\leq \sum_{\bs_{i_{1}},\bs_{i_{2}},\bs_{i_{3}},\bs_{i_{4}} \in \Gamma(\bm{\ell};\bm{\epsilon})\cap R_{n}, \bm{\epsilon} \neq \bm{\epsilon}_{0}}\prod_{k=1}^{4}\|\bm{X}(\bs_{i_{k}})\|_{\infty}.
\end{align*}
Then we have $E_{\cdot \mid \bZ}\left[\|S_{n}(\bm{\ell};\bm{\epsilon})\|_{\infty}^{4}\right] \lesssim \lambda_{1,n}^{4(d-1)}\lambda_{2,n}^{4}n^{4}\lambda_{n}^{-4d}M_{n}^{4}$. This yields that 
\begin{align}
E_{\cdot \mid \bZ}\left[\max_{\bm{\ell} \in L_{1,n}}\|\breve{S}_{n}(\bm{\ell};\bm{\epsilon})\|_{\infty}^{2}\right] &\lesssim \lambda_{n}^{d/2}\lambda_{1,n}^{-d/2}\lambda_{1,n}^{2d-2}\lambda_{2,n}^{2}n^{2}\lambda_{n}^{-2d}M_{n}^{2} \nonumber \\
&= \lambda_{1,n}^{3d/2}(\lambda_{2,n}/\lambda_{1,n})^{2}\lambda_{n}^{-3d/2}n^{2}M_{n}^{2}. \label{moment-bound1}
\end{align}
Similar arguments to those used in the proof of (\ref{Var_ineq_S2}) yield that when $\bm{\epsilon} \neq \bm{\epsilon}_{0}$,
\begin{align}
\Var_{\cdot \mid \bZ}(\breve{S}_{n}^{(j)}(\bm{\ell};\bm{\epsilon})) \lesssim \overline{\beta}_{q}\lambda_{1,n}^{d-1}\lambda_{2,n}(\log n + n\lambda_{n}^{-d})^{2}, \label{moment-bound2}
\end{align}
where $\overline{\beta}_{q} =  1 + \sum_{k=1}^{\lambda_{1,n}}k^{d-1}\beta_{1}^{1-2/q}(k)$. Thus, from (\ref{moment-bound1}), (\ref{moment-bound2}) and Lemma \ref{lem: maximal inequality} ahead, we have 
\begin{align*}
&E_{\cdot \mid \bZ}\left[r_{n}^{-1/2}\|\breve{S}_{2,n}(\bm{\epsilon})\|_{\infty}\right]\\
&\lesssim \left((\log n + n\lambda_{n}^{-d})\sqrt{{\overline{\beta}_{q}\lambda_{1,n}^{d-1}\lambda_{2,n}(\lambda_{n}/\lambda_{1,n})^{d}\log p \over n^{2}\lambda_{n}^{-d}}} + \lambda_{1,n}^{3d/4-1}\lambda_{2,n}\lambda_{n}^{-d/4}M_{n}\log p\right)\\
&= \left((\log n + n\lambda_{n}^{-d})\sqrt{{\overline{\beta}_{q}\lambda_{2,n}\log p \over n^{2}\lambda_{n}^{-2d}\lambda_{1,n}}} + \lambda_{1,n}^{3d/4-1}\lambda_{2,n}\lambda_{n}^{-d/4}M_{n}\log p\right) \\
&\lesssim n^{-c'}\log^{-1/2} p.
\end{align*}
This implies that $S_{2,n}(\bm{\epsilon})$ and $S_{3,n}(\bm{\epsilon})$  for $\bm{\epsilon} \neq \bm{\epsilon}_{0}$ are asymptotically negligible from the argument in Step 1 of the proof of Theorem \ref{Thm: GA_hyperrec}.

\underline{Step 2} (High-dimensional CLT applied to  the sum of independent blocks). 
As in the proof of Theorem \ref{Thm: GA_hyperrec}, we have that uniformly over $j = 1,\dots, p$,
\begin{align*}
&E_{\cdot \mid \bZ}\left[\left|{\breve{S}^{(j)}_{n}(\bm{\ell};\bm{\epsilon}_{0}) \over \sqrt{r_{n}(\lambda_{1,n}/\lambda_{n})^{d}}}\right|^{k}\right] \lesssim {\lambda_{1,n}^{kd}n^{k}\lambda_{n}^{-kd}M_{n}^{k-2} \over n^{k}\lambda_{n}^{-kd}\lambda_{1,n}^{kd/2}} = \lambda_{1,n}^{kd/2}M_{n}^{k-2},\ k=3,4, \\
&E_{\cdot \mid \bZ}\left[\left\|{\breve{S}_{n}(\bm{\ell};\bm{\epsilon}_{0}) \over \sqrt{r_{n}(\lambda_{1,n}/\lambda_{n})^{d}}}\right\|_{\infty}^{q}\right] \lesssim \lambda_{1,n}^{qd/2}M_{n}^{q}.
\end{align*}
Therefore, the conditions of Proposition 2.1 (i) in \cite{ChChKa17} are verified with $B_{n} = \lambda_{1,n}^{3d/2}M_{n}$ in their notation, which leads to the conclusion of Step 2. Step 3 is completely analogous. This completes the proof. \qed

\subsection{Proof of Corollary \ref{cor41}}\label{pf-Cor4-1}
Observe that 
\begin{align}
     &\sup_{A \in \mathcal{A}}\left|P_{\cdot \mid \bZ}\left(\sqrt{\lambda_{n}^{d}}\overline{\bY}_{n} \in A\right) - P\left(\breve{\bm{V}}_{n} \in A\right)\right| \nonumber \\
     &\leq \sup_{A \in \mathcal{A}}\left|P_{\cdot \mid \bZ}\left(\sqrt{\lambda_{n}^{d}}\overline{\bY}_{n} \in A\right) - P_{\cdot \mid \bZ}\left(\bm{V}_{n} \in A\right)\right| + \sup_{A \in \mathcal{A}}\left|P_{\cdot \mid \bZ}\left(\bm{V}_{n} \in A\right) - P\left(\breve{\bm{V}}_{n} \in A \right)\right| \nonumber \\
     & =: \rho_{1,n} + \rho_{2,n} \label{cor41-decomp}
\end{align}
From Theorem \ref{Thm: GA_hyperrec}, we have $\rho_{1,n}=o(1)$. 
Since both $\bm{V}_{n} = (V_{1,n},\dots, V_{p,n})'$ and $\breve{\bm{V}}_{n} = (\breve{V}_{1,n},\dots,\breve{V}_{p,n})'$ are Gaussian (under $P_{\cdot \mid \bZ}$ for $\bm{V}_{n}$), by Lemma \ref{lem: Gaussian comparison} ahead, $\rho_{2,n}$ is bounded by $C\Xi_{n}^{1/3}\log^{2/3} p$, where 
\begin{align*}
\Xi_{n} &= \max_{1 \leq j,k \leq p}\left|E_{\cdot \mid \bZ}[V_{j,n}V_{k,n}] - E[\breve{V}_{j,n}\breve{V}_{k,n}]\right|\\
&= \max_{1 \leq j,k \leq p}\left|r_{n}^{-1}E_{\cdot \mid \bZ}[\breve{S}_{1,n}^{(j)}\breve{S}_{1,n}^{(k)}] - E[\breve{V}_{j,n}\breve{V}_{k,n}]\right|. 
\end{align*}
The second equality follows from the conclusion of Theorem \ref{Thm: GA_hyperrec}. From Condition (i) in Corollary \ref{cor41} and similar arguments to the proof of Lemma \ref{variance-rate}, we have  $\Xi_{n} \lesssim n^{-({c' \over 2} \wedge {1 \over 2r_{0}})} + o((\log n)^{-2})$, which implies that $\rho_{2,n} = o(1)$. This completes the proof. \qed

\section{Proof of Theorem \ref{DWBvalidity: HR}}\label{Append_B}

To simplify our analysis, we set the bandwidth  as $b_{n} = \lambda_{2,n}$. We prove Theorem \ref{DWBvalidity: HR} separately under Assumptions \ref{ass: design} and \ref{ass: design2}.

\subsection{Proof of Theorem \ref{DWBvalidity: HR} under Assumption \ref{ass: design}}

We divide the proof into several subsections. 

\subsubsection{Approximation of $\sqrt{\lambda_{n}^{d}}(\overline{\bY}^{\ast}_{n} - \overline{\bY}_{n})$ by $r_{n}^{-1/2}\sum_{i=1}^{n}W(\bs_{i})\bX(\bs_{i})$}
\label{B11-BootHD}
Write
\begin{align*}
U_{n} &:= \sum_{i=1}^{n}W(\bs_{i})(\bY(\bs_{i}) - \overline{\bY}_{n})\\
&= \sum_{i=1}^{n}W(\bs_{i})\bX(\bs_{i}) - \sum_{i=1}^{n}W(\bs_{i})\overline{\bY}_{n} + \sum_{i=1}^{n}W(\bs_{i})\bm{\Upsilon}(\bs_{i})\\
& =: U_{1,n} + U_{2,n} + U_{3,n}. 
\end{align*}

We will show the asymptotic negligibility of $U_{2,n}$ and $U_{3,n}$ in two steps. Namely, in Step 1, we will show that with $P_{\cdot \mid \bZ}$-probability at least $1 - O(n^{-c'/6})$, 
\begin{align}\label{B.1.1-(1)}
P_{\cdot|\bY, \bZ}\left(r_{n}^{-1/2}\left\|U_{2,n}\right\|_{\infty} > Cn^{-c'/2}\log^{-1/2}p\right) \leq n^{-c'/4},
\end{align}
while in Step 2, we will show  that with $P_{\cdot \mid \bZ}$-probability at least $1-O(n^{-c'/6})$,
\begin{align}\label{B.1.1-(2)}
P_{\cdot|\bY, \bZ}\left(r_{n}^{-1/2}\left\|U_{3,n}\right\|_{\infty} > Cn^{-c'/2}\log^{-1/2}p\right) \leq n^{-c'/4}.
\end{align}

Given (\ref{B.1.1-(1)}) and (\ref{B.1.1-(2)}), we have that with $P_{\cdot \mid \bZ}$-probability at least $1 - O(n^{-c'/6})$,   for any rectangle $A = \prod_{j=1}^{p}[a_{j},b_{j}]$ with $a = (a_{1},\dots, a_{p})'$ and $b = (b_{1},\dots, b_{p})'$, 
\begin{align}
    &P_{\cdot \mid \bY,\bZ}\left(r_{n}^{-1/2}U_{n} \in A\right) = P_{\cdot \mid \bY,\bZ}\left(\left\{-r_{n}^{-1/2}U_{n} \leq -a\right\} \cap \left\{r_{n}^{-1/2}U_{n} \leq b\right\}\right) \nonumber \\
    &\leq P_{\cdot \mid \bY,\bZ}\left(\left\{-r_{n}^{-1/2}U_{n} \leq -a\right\} \cap \left\{r_{n}^{-1/2}U_{n} \leq b\right\} \cap \left\{r_{n}^{-1/2}\left\|U_{2,n} + U_{3,n}\right\|_{\infty} \leq 2Cn^{-c'/2}\log^{-1/2}p\right\}\right) \nonumber \\
    &\quad + P_{\cdot \mid \bY,\bZ}\left(r_{n}^{-1/2}\left\|U_{2,n}\right\|_{\infty} > Cn^{-c'/2}\log^{-1/2}p\right) + P_{\cdot \mid \bY,\bZ}\left(r_{n}^{-1/2}\left\|U_{3,n}\right\|_{\infty} > Cn^{-c'/2}\log^{-1/2}p\right) \nonumber \\
    &\leq P_{\cdot|\bm{X},\bZ}\left(\left\{-r_{n}^{-1/2}U_{1,n} \leq -a + 2Cn^{-c'/2}\log^{-1/2}p\right\} \right. \nonumber \\ 
    &\left. \quad \quad \cap \left\{r_{n}^{-1/2}U_{1,n} \leq b + 2Cn^{-c'/2}\log^{-1/2}p\right\} \right) + O(n^{-c'/4}). \label{asy-neg-U2-U3-max}
\end{align}
Likewise, we have that with $P_{\cdot \mid \bZ}$-probability at least $1 - O(n^{-c'/6})$,
\begin{align*}
    P_{\cdot \mid \bY,\bZ}\left(r_{n}^{-1/2}U_{n} \in A\right) &\geq P_{\cdot|\bm{X},\bZ}\left(\left\{-r_{n}^{-1/2}U_{1,n} \leq -a + 2Cn^{-c'/2}\log^{-1/2}p\right\} \right.  \\ 
    &\left. \quad \quad \cap \left\{r_{n}^{-1/2}U_{1,n} \leq b + 2Cn^{-c'/2}\log^{-1/2}p\right\} \right) - O(n^{-c'/4}).
\end{align*}
Here $P_{\cdot \mid \bX,\bZ}$ denotes the conditional probability given $\sigma\left(\{\bX(\bs): \bs \in \R^{d}\} \cup \{\bZ_{i}\}_{i \geq 1}\right)$. We can replace $P_{\cdot \mid \bY,\bZ}$ with $P_{\cdot \mid \bX,\bZ}$ by interpreting them as the probability with respect to $W$.

\underline{Step 1} (Proof of (\ref{B.1.1-(1)})). 
We have  $E_{\cdot \mid \bZ}\left[\left(\sum_{i=1}^{n}W(\bs_{i})\right)^{2}\right] = \sum_{i_{1}, i_{2}=1}^{n}a(\|(\bs_{i_{1}} - \bs_{i_{2}})/\lambda_{2,n}\|)$, which is $O(n^{2}(\lambda_{2,n}/\lambda_{n})^{d})$. Since $\bm{V}_{n}$ is Gaussian under $P_{\cdot \mid \bZ}$, the maximal inequality for Gaussian random variables (Lemma 2.2.10 in \cite{vaWe96}) yields that $E_{\cdot \mid \bZ}[\|\bm{V}_{n}\|_{\infty}] = O(\sqrt{\log p})$. Together with Theorem \ref{Thm: GA_hyperrec}, we have
\begin{align}\label{max-rate-Xmean}
E_{\cdot \mid \bZ}\left[\|\overline{\bY}_{n}\|_{\infty}\right]= O(\lambda_{n}^{-d/2}\sqrt{\log p }).
\end{align}
Observe that for $t_{1}, t_{2}>0$, 
\begin{align}
P_{\cdot \mid \bZ}\left(P_{\cdot|\bm{\bY, \bZ}}\left(r_{n}^{-1/2}\left\|U_{2,n}\right\|_{\infty} > t_{1}\right) > t_{2}\right) &\leq {1 \over t_{2}}E_{\cdot \mid \bZ}\left(P_{\cdot|\bm{\bY, \bZ}}\left(r_{n}^{-1/2}\left\|U_{2,n}\right\|_{\infty} > t_{1}\right)\right) \nonumber \\
&\leq {1 \over t_{1}t_{2}}E_{\cdot \mid \bZ}\left[E_{\cdot|\bY, \bZ}\left[r_{n}^{-1/2}\left\|U_{2,n}\right\|_{\infty}\right]\right] \nonumber \\
&\leq {r_{n}^{-1/2} \over t_{1}t_{2}}E_{\cdot \mid \bZ}\left[\|\overline{\bY}_{n}\|_{\infty}\left|\sum_{i=1}^{n}W(\bm{s}_{i})\right|\right].\label{asy-neg-U2-max}
\end{align}
Here we have
\begin{align*}
E_{\cdot \mid \bZ}\left[\|\overline{\bY}_{n}\|_{\infty}\left|\sum_{i=1}^{n}W(\bm{s}_{i})\right|\right] &= E_{\cdot \mid \bZ}\left[\|\overline{\bY}_{n}\|_{\infty}\right]E_{\cdot \mid \bZ}\left[\left|\sum_{i=1}^{n}W(\bm{s}_{i})\right|\right]\\
&\leq E_{\cdot \mid \bZ}\left[\|\overline{\bY}_{n}\|_{\infty}\right]\left(E_{\cdot \mid \bZ}\left[\left(\sum_{i=1}^{n}W(\bm{s}_{i})\right)^{2}\right]\right)^{1/2}\\
&\lesssim (\lambda_{n}^{-d/2}\sqrt{\log p})n(\lambda_{2,n}/\lambda_{n})^{d/2}\\
& = r_{n}^{1/2}(\lambda_{2,n}/\lambda_{n})^{d/2}\log^{1/2} p \leq r_{n}^{1/2}n^{-c'}. 
\end{align*}
Setting $t_{1} = Cn^{-c'/2}\log^{-1/2}p$ and $t_{2} = n^{-c'/4}$ in (\ref{asy-neg-U2-max}), we have
\begin{align*}
P_{\cdot \mid \bZ}\left(P_{\cdot|\bm{\bY, \bZ}}\left(r_{n}^{-1/2}\left\|U_{2,n}\right\|_{\infty} > Cn^{-c'/2}\log^{-1/2}p\right) > n^{-c'/4}\right) = O(n^{-c'/6}).
\end{align*}

\underline{Step 2} (Proof of (\ref{B.1.1-(2)})). 
It suffices to show that
\begin{align}\label{BC-conv}
P_{\cdot \mid \bZ}\left(P_{\cdot|\bY, \bZ}\left(r_{n}^{-1/2}\left\|\sum_{i=1}^{n}W(\bs_{i})\bm{\Upsilon}(\bs_{i})\right\|_{\infty} \geq Cn^{-c'/2}(\log p)^{-1/2}\right) > n^{-c'/4}\right) = O(n^{-c'/6}).
\end{align}

Observe that
\begin{align*}
    &P_{\cdot \mid \bZ}\left(P_{\cdot|\bY, \bZ}\left(r_{n}^{-1/2}\left\|\sum_{i=1}^{n}W(\bs_{i})\bm{\Upsilon}(\bs_{i})\right\|_{\infty} \geq Cn^{-c'/2}(\log p)^{-1/2}\right) > n^{-c'/4}\right)\\
    &\quad \leq n^{c'/4}E_{\cdot \mid \bZ}\left[P_{\cdot|\bY, \bZ}\left(r_{n}^{-1/2}\left\|\sum_{i=1}^{n}W(\bs_{i})\bm{\Upsilon}(\bs_{i})\right\|_{\infty} \geq Cn^{-c'/2}(\log p)^{-1/2}\right)\right]\\
    &\quad \leq C^{-1}n^{3c'/4}(\log p)^{1/2}E_{\cdot \mid \bZ}\left[E_{\cdot|\bY, \bZ}\left[r_{n}^{-1/2}\left\|\sum_{i=1}^{n}W(\bs_{i})\bm{\Upsilon}(\bs_{i})\right\|_{\infty} \right]\right]\\
    &\quad = C^{-1}r_{n}^{-1/2}n^{3c'/4}(\log p)^{1/2}E_{\cdot \mid \bZ}\left[\left\|\sum_{i=1}^{n}W(\bs_{i})\bm{\Upsilon}(\bs_{i})\right\|_{\infty} \right].
\end{align*}

Note that 
\begin{align*}
    E_{\cdot \mid \bZ}\left[\left\|\sum_{i=1}^{n}W(\bs_{i})\bm{\Upsilon}(\bs_{i})\right\|_{\infty} \right] &\leq E_{\cdot \mid \bZ}\left[\sum_{i=1}^{n}|W(\bs_{i})|\|\bm{\Upsilon}(\bs_{i})\|_{\infty}\right] \\
    &\leq nE_{\cdot \mid \bZ}\left[\max_{1\leq i\leq n}|W(\bs_{i})|\right]E_{\cdot \mid \bZ}\left[\max_{1\leq i\leq n}\|\bm{\Upsilon}(\bs_{i})\|_{\infty}\right]\\
    &\lesssim n(\log n)\delta_{n,\Upsilon}\max_{1\leq i \leq n}\|W(\bs_{i})\|_{\psi_{1}}\\
    &\lesssim n(\log n)\delta_{n,\Upsilon},
\end{align*}
where we used an inequality for the Orlicz norm (see \cite{vaWe96}, p. 95) to obtain the third inequality. Then we have 
\begin{align*}
    &P_{\cdot \mid \bZ}\left(P_{\cdot|\bY, \bZ}\left(r_{n}^{-1/2}\left\|\sum_{i=1}^{n}W(\bs_{i})\bm{\Upsilon}(\bs_{i})\right\|_{\infty} \geq Cn^{-c'/2}(\log p)^{-1/2}\right) > n^{-c'/4}\right)\\
    &\quad \lesssim \lambda_{n}^{d/2}n^{3c'/4}\delta_{n,\Upsilon}(\log n)^{3/2}. 
\end{align*}

Therefore, Condition (\ref{asy-neg-Upsilon}) in Assumption \ref{ass: design} yields (\ref{BC-conv}), which implies (\ref{B.1.1-(2)}). 

\subsubsection{Approximation of $r_{n}^{-1/2}U_{1,n}$ by large blocks}\label{Boot-U-large-block-approx}

 Define $U_{n}(\bm{\ell}; \bm{\epsilon}) = \sum_{i: \bs_{i} \in \Gamma_{n}(\bm{\ell};\bm{\epsilon})\cap R_{n}}W(\bs_{i})\bX(\bs_{i}) = (U_{n}^{(1)}(\bm{\ell}; \bm{\epsilon}),\dots,U_{n}^{(p)}(\bm{\ell}; \bm{\epsilon}))'$ and 
\begin{align*}
U_{1,n} &= \sum_{i=1}^{n}W(\bs_{i})\bX(\bs_{i}) = \sum_{\bm{\ell} \in L_{n}}U_{n}(\bm{\ell};\bm{\epsilon}_{0}) + \sum_{\bm{\epsilon} \neq \bm{\epsilon}_{0}}\underbrace{\sum_{\bm{\ell} \in L_{1,n} }U_{n}(\bm{\ell};\bm{\epsilon})}_{=:U_{12,n}(\bm{\epsilon})} + \sum_{\bm{\epsilon} \neq \bm{\epsilon}_{0}}\underbrace{\sum_{\bm{\ell} \in L_{2,n} }U_{n}(\bm{\ell};\bm{\epsilon})}_{=:U_{13,n}(\bm{\epsilon})}\\
&=: U_{11,n} + \sum_{\bm{\epsilon} \neq \bm{\epsilon}_{0}}U_{12,n}(\bm{\epsilon}) + \sum_{\bm{\epsilon} \neq \bm{\epsilon}_{0}}U_{13,n}(\bm{\epsilon}).
\end{align*}

In this subsection, we will show below that $U_{12,n}(\bm{\epsilon})$ and $U_{13,n}(\bm{\epsilon})$ ($\bm{\epsilon} \neq \bm{\epsilon}_0$) are asymptotically negligible, i.e., with $P_{\cdot \mid \bZ}$-probability at least $1 - O(n^{-c'/6}) - C(\lambda_{n}/\lambda_{1,n})^{d}\beta(\lambda_{2,n};\lambda_{n}^{d})$,
\begin{align}\label{U12-U13-asy-neg}
    P_{\cdot|\bm{X},\bZ}\left(r_{n}^{-1/2}\left\|U_{1k,n}(\bm{\epsilon})\right\|_{\infty} > Cn^{-c'/2}\log^{-1/2}p\right) \leq n^{-c'/4}
\end{align}
for $k=2,3$.

Given (\ref{U12-U13-asy-neg}), we have that with $P_{\cdot \mid \bZ}$-probability at least $1 - O(n^{-c'/6}) - 2C(\lambda_{n}/\lambda_{1,n})^{d}\beta(\lambda_{2,n};\lambda_{n}^{d})$, for any rectangle $A = \prod_{j=1}^{p}[a_{j},b_{j}]$ with $a = (a_{1},\dots, a_{p})'$ and $b = (b_{1},\dots, b_{p})'$,  
\begin{align}
    &P_{\cdot|\bm{X},\bZ}\left(r_{n}^{-1/2}U_{1,n} \in A\right) = P_{\cdot|\bm{X},\bZ}\left(\left\{-r_{n}^{-1/2}U_{1,n} \leq -a\right\} \cap \left\{r_{n}^{-1/2}U_{1,n} \leq b\right\}\right) \nonumber \\
    &\leq P_{\cdot|\bm{X},\bZ}\left(\left\{-r_{n}^{-1/2}U_{1,n} \leq -a\right\} \cap \left\{r_{n}^{-1/2}U_{1,n} \leq b\right\} \right. \nonumber \\
    & \left. \quad \cap \left\{r_{n}^{-1/2}\left\|{\textstyle \sum}_{\bm{\epsilon} \ne \bm{\epsilon}_0} U_{12,n}(\bm{\epsilon}) + U_{13,n}(\bm{\epsilon})\right\|_{\infty} \leq (2^{d+1}-2)Cn^{-c'/2}\log^{-1/2}p\right\}\right) \nonumber \\
    &\quad + P_{\cdot|\bm{X},\bZ}\left(r_{n}^{-1/2}\left\| {\textstyle \sum}_{\bm{\epsilon} \neq \bm{\epsilon}_{0}}U_{12,n}(\bm{\epsilon}) + U_{13,n}(\bm{\epsilon})\right\|_{\infty} > (2^{d+1}-2)Cn^{-c'/2}\log^{-1/2}p\right) \nonumber \\
    &\leq P_{\cdot|\bm{X},\bZ}\left(\left\{-r_{n}^{-1/2}U_{11,n} \leq -a + (2^{d+1}-2)Cn^{-c'/2}\log^{-1/2}p\right\} \right. \nonumber \\
    &\left. \quad \quad \cap \left\{r_{n}^{-1/2}U_{11,n} \leq b + (2^{d+1}-2)Cn^{-c'/2}\log^{-1/2}p\right\}\right) \nonumber \\
    &\quad + \sum_{\bm{\epsilon} \neq \bm{\epsilon}_{0}}P_{\cdot|\bm{X},\bZ}\left(r_{n}^{-1/2}\left\|U_{12,n}(\bm{\epsilon})\right\|_{\infty}>Cn^{-c'/2}\log^{-1/2}p\right) \nonumber \\
    &\quad + \sum_{\bm{\epsilon} \neq \bm{\epsilon}_{0}}P_{\cdot|\bm{X},\bZ}\left(r_{n}^{-1/2}\left\|U_{13,n}(\bm{\epsilon})\right\|_{\infty}>Cn^{-c'/2}\log^{-1/2}p\right).\nonumber \\
    &\leq P_{\cdot|\bm{X},\bZ}\left(\left\{-r_{n}^{-1/2}U_{11,n} \leq -a + (2^{d+1}-2)Cn^{-c'/2}\log^{-1/2}p\right\} \right. \nonumber \\
    &\left. \quad \quad \cap \left\{r_{n}^{-1/2}U_{11,n} \leq b + (2^{d+1}-2)Cn^{-c'/2}\log^{-1/2}p\right\}\right) + O(n^{-c'/4}) \label{U1n-U11n-approx-ineq}. 
\end{align}

Likewise, we have that with $P_{\cdot \mid \bZ}$-probability at least $1 - O(n^{-c'/6}) - 2C(\lambda_{n}/\lambda_{1,n})^{d}\beta(\lambda_{2,n};\lambda_{n}^{d})$,
\begin{align*}
    &P_{\cdot|\bm{X},\bZ}\left(r_{n}^{-1/2}U_{1,n} \in A\right) \\
    &\quad \geq P_{\cdot|\bm{X},\bZ}\left(\left\{-r_{n}^{-1/2}U_{11,n} \leq -a + (2^{d+1}-2)Cn^{-c'/2}\log^{-1/2}p\right\} \right.  \\
    &\quad \left. \quad \quad \cap \left\{r_{n}^{-1/2}U_{11,n} \leq b + (2^{d+1}-2)Cn^{-c'/2}\log^{-1/2}p\right\}\right) - O(n^{-c'/4}).
\end{align*}

Now we move on to proving (\ref{U12-U13-asy-neg}). Recall that $[\![L_{n}]\!] = O((\lambda_{n}/\lambda_{3,n})^{d}) \lesssim (\lambda_{n}/\lambda_{1,n})^{d}$ and for $\bm{\ell}_{1}, \bm{\ell}_{2} \in L_{n}$ with $\bm{\ell}_{1} \neq \bm{\ell}_{2}$, let 
\begin{align*}
\mathcal{J}_{1}(\bm{\epsilon}) &= \{1 \leq i_{1} \leq n: \bs_{i_{1}} \in \Gamma_{n}(\bm{\ell}_{1};\bm{\epsilon})\},\ \mathcal{J}_{2}(\bm{\epsilon}) = \{1 \leq i_{2} \leq n: \bs_{i_{2}} \in \Gamma_{n}(\bm{\ell}_{2};\bm{\epsilon})\}. 
\end{align*} 
Pick any $\bs_{i_k}=(s_{1,i_k},\dots,s_{d,i_k})$, $k=1,2$ such that $i_{1} \in \mathcal{J}_{1}(\bm{\epsilon})$ and $i_{2} \in \mathcal{J}_{2}(\bm{\epsilon})$. Then we have $\max_{1 \leq u \leq d}|s_{u,i_{1}} - s_{u,i_{2}}| \geq \lambda_{2,n}$ from the definition of $\Gamma(\bm{\ell};\bm{\epsilon})$. This implies $\|\bs_{i_1}-\bs_{i_2}\| \geq \lambda_{2,n}$ and $|\bs_{i_1}-\bs_{i_2}|\geq \lambda_{2,n}$, so that the dependence structure of the Gaussian random field $W$ does not contribute to the dependence between $U_{n}(\bm{\ell}_{1};\bm{\epsilon})$ and $U_{n}(\bm{\ell}_{2};\bm{\epsilon})$ for $\bm{\ell}_{1} \neq \bm{\ell}_{2}$ since $W(\bm{s}_{1})$ and $W(\bm{s}_{2})$ are independent if $\|\bm{s}_{i_1} - \bm{s}_{i_2}\| \geq \lambda_{2,n}$ from the definition of $W$. 

For $\bm{\epsilon} \in \{1,2\}^{d}$, consider an arrangement $U_{n}(\bm{\ell}_{1};\bm{\epsilon}), \dots, U_{n}(\bm{\ell}_{[\![L_{n}]\!]};\bm{\epsilon})$ of $\{U_{n}(\bm{\ell};\bm{\epsilon}): \bm{\ell} \in L_{n}\}$.
Let $Q_{\cdot \mid \bZ}^{(a)}$ be the marginal distribution of $U_{n}(\bm{\ell}_{a};\bm{\epsilon})$ and let $Q_{\cdot \mid \bZ}^{(a:b)}$ be the joint distribution of $\{U_{n}(\bm{\ell}_{k};\bm{\epsilon}): a \leq k \leq b\}$. The $\beta$-mixing property of $\bm{X}$ implies that for $1 \leq k \leq [\![L_{n}]\!]-1\!$,
\[
\left \|Q_{\cdot \mid \bZ} - Q_{\cdot \mid \bZ}^{(1:k)} \times Q_{\cdot \mid \bZ}^{(k+1:[\![L_{n}]\!])} \right\|_{\mathrm{TV}} \lesssim \beta(\lambda_{2,n};\lambda_{n}^{d}),
\]
where the inequality does not depend on the arrangement of $\{U_{n}(\bm{\ell};\bm{\epsilon}): \bm{\ell} \in L_{n}\}$. 
Hence, the assumption (\ref{prod_bound})  in Lemma \ref{lem: indep_lemma} is satisfied for $\{U_{n}(\bm{\ell};\bm{\epsilon}) : \bm{\ell} \in L_{n}\}$ ($\bm{\epsilon} \neq \bm{\epsilon}_{0}$) with $\tau \sim \beta(\lambda_{2,n};\lambda_{n}^{d})$ and $m \lesssim (\lambda_{n}/\lambda_{1,n})^{d}$. This implies that for $\bm{\epsilon} \neq \bm{\epsilon}_{0}$, we can construct a sequence of independent random vectors $\{\breve{U}_{n}(\bm{\ell};\bm{\epsilon}): \bm{\ell} \in L_{n}\}$ with
$\breve{U}_{n}(\bm{\ell};\bm{\epsilon}) \stackrel{d}{=} U_{n}(\bm{\ell};\bm{\epsilon})$  under $P_{\cdot \mid \bZ}$. 
Define
\[
\breve{U}_{12,n}(\bm{\epsilon}) = \sum_{\bm{\ell} \in L_{1,n}}\breve{U}_{n}(\bm{\ell};\bm{\epsilon}),\ \breve{U}_{13,n}(\bm{\epsilon}) = \sum_{\bm{\ell} \in L_{2,n}}\breve{U}_{n}(\bm{\ell};\bm{\epsilon}),\ \bm{\epsilon}\neq \bm{\epsilon}_{0}.
\]
Then for any $t_{1}>0$,
\begin{align}
    &\sup_{t \geq 0}\left|P_{\cdot \mid \bZ}\left(P_{\cdot|\bm{X},\bZ}\left(\left\|U_{1k,n}(\bm{\epsilon})\right\|_{\infty}>t_{1}\right)>t\right) - P_{\cdot \mid \bZ}\left(P_{\cdot|\bm{X},\bZ}\left(\left\|\breve{U}_{1k,n}(\bm{\epsilon})\right\|_{\infty}>t_{1}\right)>t\right)\right| \nonumber \\
    &\lesssim (\lambda_{n}/\lambda_{1,n})^{d}\beta(\lambda_{2,n};\lambda_{n}^{d}),\ k=2,3.\label{indep-block-U12-U13-max}
\end{align}
By Markov's inequality, we have
\begin{align*}
    P_{\cdot \mid \bZ}\left(P_{\cdot|\bm{X},\bZ}\left(\left\|\breve{U}_{1k,n}(\bm{\epsilon})\right\|_{\infty}>t_{1}\right)>t\right) &\leq {1 \over t_{1}t}E_{\cdot \mid \bZ}\left[\left\|\breve{U}_{1k,n}(\bm{\ell};\bm{\epsilon})\right\|_{\infty}\right].
\end{align*}
Similarly to  the proofs of (\ref{mean_ineq_S2}) and (\ref{Var_ineq_S2}), we have 
\begin{align*}
    E_{\cdot \mid \bZ}\left[\left\|\breve{U}_{1k,n}(\bm{\ell};\bm{\epsilon})\right\|_{\infty}\right]\lesssim r_{n}^{1/2}n^{-c'}\log^{-1/2}p,\ k=2,3.
\end{align*}
Thus, for $t_{1} = C(n^{2}\lambda_{n}^{d})^{1/2}n^{-c'/2}\log^{-1/2}p$ and $t = n^{-c'/4}$, we have 
\begin{align}\label{asy-neg-U12-U13-XZ}
    P_{\cdot \mid \bZ}\left(P_{\cdot|\bm{X},\bZ}\left(r_{n}^{-1/2}\left\|\breve{U}_{1k,n}(\bm{\epsilon})\right\|_{\infty}>Cn^{-c'/2}\log^{-1/2}p\right)>n^{-c'/4}\right) &\lesssim n^{-c'/6},\ k=2,3.  
\end{align}
Together with (\ref{indep-block-U12-U13-max}) and (\ref{asy-neg-U12-U13-XZ}), we have
\begin{align*}
    &P_{\cdot \mid \bZ}\left(P_{\cdot|\bm{X},\bZ}\left(r_{n}^{-1/2}\left\|U_{1k,n}(\bm{\epsilon})\right\|_{\infty}>Cn^{-c'/2}\log^{-1/2}p\right)>n^{-c'/4}\right)\\
    &\quad \lesssim (\lambda_{n}/\lambda_{1,n})^{d}\beta(\lambda_{2,n};\lambda_{n}^{d}) + O(n^{-c'/6}),
\end{align*}
which leads to (\ref{U12-U13-asy-neg}).

\subsubsection{Conditional Gaussian approximation to $r_{n}^{-1/2}U_{11,n}$}\label{condti-GA-proof}
We will show that with $P_{\cdot \mid \bZ}$-probability at least $1 - O(n^{-c'/6})$, 
\begin{align}\label{DWB: block}
&\sup_{A \in \mathcal{A}}\left|P_{\cdot|\bX,\bZ}\left(r_{n}^{-1/2}U_{11,n} \in A\right) - P_{\cdot \mid \bZ}\left(\bm{G}_{n} \in A\right)\right| = O(n^{-c'/6}). 
\end{align}
where $\bm{G}_{n} = (G_{1,n},\dots, G_{p,n})'$ is a centered Gaussian random vector under $P_{\cdot \mid \bZ}$ with covariance 
\[
E_{\cdot \mid \bZ}\left[\bm{G}_{n}\bm{G}'_{n}\right] = {1 \over n^{2}\lambda_{n}^{-d}}\sum_{\bm{\ell} \in L_{n}}E_{\cdot \mid \bZ}\left[U_{n}(\bm{\ell};\bm{\epsilon}_{0})U_{n}(\bm{\ell};\bm{\epsilon}_{0})'\right].
\]

Recall $U_{11,n}=\sum_{\bm{\ell}\in L_{n}}U_{n}(\bm{\ell};\bm{\epsilon}_{0})$ and $U_{n}(\bm{\ell}; \bm{\epsilon}_{0}) = \sum_{i: \bs_{i} \in \Gamma_{n}(\bm{\ell};\bm{\epsilon}_{0})\cap R_{n}}W(\bs_{i})\bX(\bs_{i})$. Note that $\{U_{n}(\bm{\ell};\bm{\epsilon}_{0}): \bm{\ell}\in L_{n}\}$ is a sequence of independent random vectors under $P_{\cdot \mid \bX,\bZ}$ since $W(\bs_{1})$ and $W(\bs_{2})$ are independent for any $\bs_{1}\in \Gamma_{n}(\bm{\ell}_{1};\bm{\epsilon}_{0})$ and $\bs_{2}\in \Gamma_{n}(\bm{\ell}_{2};\bm{\epsilon}_{0})$ such that $\bm{\ell}_{1} \neq \bm{\ell}_{2}$ from similar arguments to Subsection \ref{Boot-U-large-block-approx}. 

Since both $r_{n}^{-1/2}U_{11,n}$ and $\bm{G}_{n}$ are (conditionally) Gaussian, by Lemma \ref{lem: Gaussian comparison} ahead, the left-hand side on (\ref{DWB: block}) is bounded by $C\breve{\Delta}^{1/3}\log^{2/3} p$, where 
\begin{align*}
\breve{\Delta}&= {1 \over n^{2}\lambda_{n}^{-d}}\max_{1 \leq j_{1},j_{2} \leq p}\left|\sum_{\bm{\ell} \in L_{n}}\!\! \left(E_{\cdot|\bX,\bZ}[U_{n}^{(j_{1})}(\bm{\ell};\bm{\epsilon}_{0})U_{n}^{(j_{2})}(\bm{\ell};\bm{\epsilon}_{0})] - E_{\cdot \mid \bZ}[U_{n}^{(j_{1})}(\bm{\ell};\bm{\epsilon}_{0})U_{n}^{(j_{2})}(\bm{\ell};\bm{\epsilon}_{0})] \right)\right|,
\end{align*}
and $E_{\cdot|\bX,\bZ}$ denotes the conditional expectation given $\sigma(\{\bX(\bs) : \bs \in \R^{d}\} \cup \{\bZ_{i}\}_{i \geq 1})$. Hence it suffices to prove that $P_{\cdot \mid \bZ}(\breve{\Delta}>C'n^{-5c'/6}\log^{-2}p) \leq Cn^{-c'/6}$ with suitable constants $C$ and $C'$ that are independent of $n$. Observe that 
\begin{align*}
|E_{\cdot|\bX,\bZ}[U_{n}^{(j_{1})}(\bm{\ell};\bm{\epsilon}_{0})U_{n}^{(j_{2})}(\bm{\ell};\bm{\epsilon}_{0})]| &= \left|E_{\cdot|\bX,\bZ}\left[\sum_{\bs_{i_{1}}, \bs_{i_{2}} \in \Gamma_{n}(\bm{\ell};\bm{\epsilon}_{0})\cap R_{n}}W(\bs_{i_{1}})W(\bs_{i_{2}})X_{j_{1}}(\bs_{k_{1}})X_{j_{2}}(\bs_{i_{2}})\right]\right|\\
&\lesssim \sum_{\bs_{i_{1}}, \bs_{i_{2}} \in \Gamma_{n}(\bm{\ell};\bm{\epsilon}_{0})\cap R_{n}}a\left({\|\bs_{i_{1}} - \bs_{i_{2}}\| \over \lambda_{2,n}}\right)|X_{j_{1}}(\bs_{i_{1}})||X_{j_{2}}(\bs_{i_{2}})|.
\end{align*}
Then by Lemma \ref{n summands} and inequalities for Orlicz norm (\cite{vaWe96}, p.95), we have 
\begin{align*}
&E_{\cdot \mid \bZ}\left[\max_{\bm{\ell} \in L_{n}}\max_{1 \leq j_{1},j_{2} \leq p}|E_{\cdot|\bX,\bZ}[U_{n}^{(j_{1})}(\bm{\ell};\bm{\epsilon}_{0})U_{n}^{(j_{2})}(\bm{\ell};\bm{\epsilon}_{0})]|^{2}\right]\\
&\quad \lesssim \lambda_{1,n}^{2d}\lambda_{2,n}^{2d}n^{4}\lambda_{n}^{-4d}E_{\cdot \mid \bZ}\left[\max_{1 \leq i \leq n}\|\bm{X}(\bs_{i})\|_{\infty}^{4}\right]\\
&\quad \lesssim \lambda_{1,n}^{2d}\lambda_{2,n}^{2d}n^{4}\lambda_{n}^{-4d}(\log n)^{4}(\log p)^{4}D_{n}^{4}.
\end{align*}
Similarly, we have $E_{\cdot \mid \bZ}[(E_{\cdot|\bX,\bZ}[U_{n}^{(j_{1})}(\bm{\ell};\bm{\epsilon}_{0})U_{n}^{(j_{2})}(\bm{\ell};\bm{\epsilon}_{0})])^{2}] \lesssim \lambda_{1,n}^{2d}\lambda_{2,n}^{2d}n^{4}\lambda_{n}^{-4d}(\log n)^{4}D_{n}^{4}$ uniformly over $1 \le j \le p$. 
Hence by Lemma \ref{lem: maximal inequality} ahead, we have
 \begin{align*}
 E_{\cdot \mid \bZ}[\breve{\Delta}] &\lesssim {\lambda_{1,n}^{d}\lambda_{2,n}^{d}n^{2}\lambda_{n}^{-2d}(\lambda_{n}/\lambda_{1,n})^{d/2}(\log n)^{2}(\log p)^{1/2}D_{n}^{2} \over n^{2}\lambda_{n}^{-d}}\\
 &\quad + {\lambda_{1,n}^{d}\lambda_{2,n}^{d}n^{2}\lambda_{n}^{-2d}(\log n)^{2}(\log p)^{3}D_{n}^{2} \over n^{2}\lambda_{n}^{-d}}\\
&\lesssim D_{n}^{2}\lambda_{1,n}^{d/2}\lambda_{2,n}^{d}\lambda_{n}^{-d/2}(\log n)^{2}(\log p)^{1/2} + D_{n}^{2}\lambda_{1,n}^{d}\lambda_{2,n}^{d}\lambda_{n}^{-d}(\log n)^{2}(\log p)^{3} \\
&\leq C'n^{-c'}\log^{-2}p. 
 \end{align*}
The conclusion (\ref{DWB: block}) follows from an application of Markov's inequality.

\subsubsection{Gaussian comparison}\label{B13-BootHD}
We wish to show that
\begin{align}\label{DWB: comparison}
&\sup_{A \in \mathcal{A}}\left|P_{\cdot \mid \bZ}\left(\bm{V}_{n} \in A\right) - P_{\cdot \mid \bZ}\left(\bm{G}_{n} \in A\right)\right| = O\left(n^{-c'/6}\right).
\end{align}
By Lemma \ref{lem: Gaussian comparison} ahead, the left-hand side on (\ref{DWB: comparison}) is bounded by $C\hat{\Delta}^{1/3}\log^{2/3} p$, where 
\begin{align*}
\hat{\Delta} &= {1 \over n^{2}\lambda_{n}^{-d}}\max_{1 \leq j_{1},j_{2} \leq p}\left|\sum_{\bm{\ell} \in L_{n}}\!\left(E_{\cdot \mid \bZ}[U_{n}^{(j_{1})}(\bm{\ell};\bm{\epsilon}_{0})U_{n}^{(j_{2})}(\bm{\ell};\bm{\epsilon}_{0})] - E_{\cdot \mid \bZ}[S_{n}^{(j_{1})}(\bm{\ell};\bm{\epsilon}_{0})S_{n}^{(j_{2})}(\bm{\ell};\bm{\epsilon}_{0})]\right)\right|.
\end{align*}
Hence it suffices to prove that $\hat{\Delta} \lesssim n^{-c'}\log^{-2}p$. Observe that
\begin{align*}
&\left|E_{\cdot|\bZ}[U_{n}^{(j_{1})}(\bm{\ell};\bm{\epsilon}_{0})U_{n}^{(j_{2})}(\bm{\ell};\bm{\epsilon}_{0})] - E_{\cdot|\bm{Z}}[S_{n}^{(j_{1})}(\bm{\ell};\bm{\epsilon}_{0})S_{n}^{(j_{2})}(\bm{\ell};\bm{\epsilon}_{0})]\right|\\
&\quad \leq \sum_{m_{1}=1}^{n}\sum_{m_{2}=1}^{n}\left(1 - a\left(\left\|{\bm{s}_{m_{1}} - \bm{s}_{m_{2}} \over \lambda_{2,n}}\right\|\right)\right)|\Sigma_{j_{1},j_{2}}(\bm{s}_{m_{1}} - \bm{s}_{m_{2}})|1\{\bm{s}_{m_{1}}, \bm{s}_{m_{2}} \in \Gamma_{n}(\bm{\ell};\bm{\epsilon}) \cap R_{n}\}\\
&\quad =: Q_{n}.
\end{align*}
Let $R_{D} = \{\bm{z}_{1} - \bm{z}_{2}: \bm{z}_{1},\bm{z}_{2} \in {\Gamma_{n}(\bm{\ell};\bm{\epsilon}) \over \lambda_{n}} \cap R_{0}\}$, $R(\bm{z}) = (R_{0} + \bm{z})\cap R_{0}$ for all $\bm{z} \in R_{D}$, and let $M>0$. Then we have 
\begin{align*}
Q_{n} &=\sum_{m_{1}=1}^{n}\sum_{m_{2}=1}^{n}\left(1 - a\left(\left\|{\bm{s}_{m_{1}} - \bm{s}_{m_{2}} \over \lambda_{2,n}}\right\|\right)\right)|\Sigma_{j_{1},j_{2}}(\bm{s}_{m_{1}} - \bm{s}_{m_{2}})|\\
&\quad \times 1\{\bm{s}_{m_{1}} -\bm{s}_{m_{2}} \in \lambda_{n}R_{D}\}1\{\bm{s}_{m_{2}}/\lambda_{n} \in R(\bm{s}_{m_{1}}/\lambda_{n})\}\\
&=\sum_{m_{1}=1}^{n}\sum_{m_{2}=1}^{n}\left(1 - a\left(\left\|{\bm{s}_{m_{1}} - \bm{s}_{m_{2}} \over \lambda_{2,n}}\right\|\right)\right)|\Sigma_{j_{1},j_{2}}(\bm{s}_{m_{1}} - \bm{s}_{m_{2}})|\\
&\quad \times 1\{\bm{s}_{m_{1}} -\bm{s}_{m_{2}} \in \lambda_{n}R_{D}\}1\{\bm{s}_{m_{2}}/\lambda_{n} \in R(\bm{s}_{m_{1}}/\lambda_{n})\}1\{\|\bm{s}_{m_{1}} -\bm{s}_{m_{2}}\|\leq M\}\\
&\quad + \sum_{m_{1}=1}^{n}\sum_{m_{2}=1}^{n}\left(1 - a\left(\left\|{\bm{s}_{m_{1}} - \bm{s}_{m_{2}} \over \lambda_{2,n}}\right\|\right)\right)|\Sigma_{j_{1},j_{2}}(\bm{s}_{m_{1}} - \bm{s}_{m_{2}})|\\
&\quad \times 1\{\bm{s}_{m_{1}} -\bm{s}_{m_{2}} \in \lambda_{n}R_{D}\}1\{\bm{s}_{m_{2}}/\lambda_{n} \in R(\bm{s}_{m_{1}}/\lambda_{n})\}1\{\|\bm{s}_{m_{1}} -\bm{s}_{m_{2}}\| > M\}\\
&=: Q_{1,n} + Q_{2,n}. 
\end{align*}
When $\{\|\bm{s}_{m_{1}}-\bm{s}_{m_{2}}\| \leq M\}$ and $|M/\lambda_{2,n}|\leq c_{W}$, we have $1-a(\|(\bm{s}_{m_{1}}-\bm{s}_{m_{2}})/\lambda_{2,n}\|) \lesssim |M/\lambda_{2,n}|$ uniformly over $\|\bm{s}_{m_{1}}-\bm{s}_{m_{2}}\| \leq M$, which implies that 
\begin{align*}
Q_{1,n} &\lesssim \left({M \over \lambda_{2,n}}\right)\sum_{m_{1}=1}^{n}\sum_{m_{2}=1}^{n}|\Sigma_{j_{1},j_{2}}(\bm{s}_{m_{1}} - \bm{s}_{m_{2}})|\\
&\quad \times 1\{\bm{s}_{m_{1}} -\bm{s}_{m_{2}} \in \lambda_{n}R_{D}\}1\{\bm{s}_{m_{2}}/\lambda_{n} \in R(\bm{s}_{m_{1}}/\lambda_{n})\}1\{\|\bm{s}_{m_{1}} -\bm{s}_{m_{2}}\| > M\}\\
& \lesssim \max_{1\leq j\leq p}|\Sigma_{j,j}(\bm{0})|\left({M \over \lambda_{2,n}}\right)(n\lambda_{n}^{-d}+\log n)\left({\lambda_{1,n} \over \lambda_{n}}\right)^{d}n  \\
&\lesssim Mn\lambda_{n}^{-d}(n\lambda_{n}^{-d}+\log n)\lambda_{2,n}^{-1}\lambda_{1,n}^{d}.
\end{align*}
Observe that 
\begin{align*}
    E_{\bZ}[Q_{2,n}] &= n^{2}\int_{R_{D}\cap \{\|\lambda_{n}\bm{z}\|>M\}}\int_{R(\bm{z})}\left(1 - a\left(\left\|{\lambda_{n}\bm{z}\over \lambda_{2,n}}\right\|\right)\right)|\Sigma_{j_{1},j_{2}}(\lambda_{n}\bm{z})|f(\bm{z}_{1})f(\bm{z}_{1}-\bm{z})d\bm{z}_{1}d\bm{z}\\
    &= n^{2}\lambda_{n}^{-d}\!\!\!\int_{\lambda_{n}R_{D}\cap \{\|\bm{w}\|>M\}}\int_{R(\bm{w}/\lambda_{n})}\!\!\left(\!1 - a\!\left(\left\|{\bm{w}\over \lambda_{2,n}}\right\|\right)\right)\!|\Sigma_{j_{1},j_{2}}(\bm{w})|f(\bm{z}_{1})f(\bm{z}_{1}-\bm{w}/\lambda_{n})d\bm{z}_{1}d\bm{w}\\
    &\lesssim n^{2}\lambda_{n}^{-d}\int_{\{\|\bm{w}\|>M\}}|\Sigma_{j_{1},j_{2}}(\bm{w})|d\bm{w}.
\end{align*}
Then for any $t>0$, 
\begin{align*}
    P_{\bZ}\left(Q_{2,n}> n^{-c'}(\log p)^{-2}n^{2}\lambda_{n}^{-d}(\lambda_{1,n}/\lambda_{n})^{d}\right) &\leq {E_{\bZ}[Q_{2,n}] \over n^{-c'}(\log p)^{-2}n^{2}\lambda_{n}^{-d}(\lambda_{1,n}/\lambda_{n})^{d}}\\
    &\lesssim n^{c'}(\log p)^{2}\lambda_{n}^{d}\lambda_{1,n}^{-d}\int_{\{\|\bm{w}\|>M\}}|\Sigma_{j_{1},j_{2}}(\bm{w})|d\bm{w}
\end{align*}
If we set $M = \lambda_{2,n}^{1/2}$, then from (\ref{cov_matrix_tail}) and the Borel-Cantelli lemma, we have ${Q_{2,n} \over n^{2}\lambda_{n}^{-d}(\lambda_{1,n}/\lambda_{n})^{d}} = O(n^{-c'}\log^{-2}p)$. This yields that  
\[
\begin{split}
\hat{\Delta} &\lesssim {n\lambda_{n}^{-d}(n\lambda_{n}^{-d}+\log n)\lambda_{1,n}^{d}\lambda_{2,n}^{-1/2}\over n^{2}\lambda_{n}^{-d}(\lambda_{1,n}/\lambda_{n})^{d}} + n^{-c'}\log^{-2}p \\
&\lesssim \lambda_{2,n}^{-1/2}\log n + n^{-c'}\log^{-2}p \\
&\lesssim n^{-c'}\log^{-2}p.
\end{split}
\]

\subsubsection{Conclusion}
Pick any rectangle $A = \prod_{j=1}^{p}[a_{j},b_{j}]$ with $a = (a_{1},\dots, a_{p})'$ and $b = (b_{1},\dots, b_{p})'$. Combining (\ref{asy-neg-U2-U3-max}), (\ref{U1n-U11n-approx-ineq}), (\ref{DWB: block}) and (\ref{DWB: comparison}) in Subsections \ref{B11-BootHD}-\ref{B13-BootHD}, we have that with $P_{\cdot \mid \bZ}$-probability at least $1 - O(n^{-c'/6}) - 2C(\lambda_{n}/\lambda_{1,n})^{d}\beta(\lambda_{2,n};\lambda_{n}^{d})$,
\begin{align*}
    &P_{\cdot \mid \bY,\bZ}\left(r_{n}^{-1/2}U_{n} \in A\right)\\ 
    &\leq P_{\cdot|\bm{X},\bZ}\left(\left\{-r_{n}^{-1/2}U_{1,n} \leq -a + 2Cn^{-c'/2}\log^{-1/2}p\right\} \right.  \\ 
    &\left. \quad \quad \cap \left\{r_{n}^{-1/2}U_{1,n} \leq b + 2Cn^{-c'/2}\log^{-1/2}p\right\} \right) + O(n^{-c'/4})\quad(\text{from (\ref{asy-neg-U2-U3-max})})\\
    &\leq P_{\cdot|\bm{X},\bZ}\left(\left\{-r_{n}^{-1/2}U_{11,n} \leq -a + 2^{d+1}Cn^{-c'/2}\log^{-1/2}p\right\} \right. \nonumber \\
    &\quad \left. \quad \quad \cap \left\{r_{n}^{-1/2}U_{11,n} \leq b + 2^{d+1}Cn^{-c'/2}\log^{-1/2}p\right\}\right) + O(n^{-c'/4}) \quad (\text{from (\ref{U1n-U11n-approx-ineq})})\\
    &\leq P_{\cdot \mid \bZ}\left(\left\{-\bm{G}_{n} \leq -a + 2^{d+1}Cn^{-c'/2}\log^{-1/2}p\right\} \right. \nonumber \\
    &\quad \left. \quad \quad \cap \left\{\bm{G}_{n} \leq b + 2^{d+1}Cn^{-c'/2}\log^{-1/2}p\right\}\right) + O(n^{-c'/6}) + O(n^{-c'/4})\quad (\text{from (\ref{DWB: block})})\\
    &\leq P_{\cdot \mid \bZ}\left(\left\{-\bm{V}_{n} \leq -a + 2^{d+1}Cn^{-c'/2}\log^{-1/2}p\right\} \right. \nonumber \\
    &\quad \left. \quad \quad \cap \left\{\bm{V}_{n} \leq b + 2^{d+1}Cn^{-c'/2}\log^{-1/2}p\right\}\right) + O(n^{-c'/6}) + O(n^{-c'/4}) \quad (\text{from (\ref{DWB: comparison})}).
\end{align*}
Combining Nazarov's inequality (see Lemma \ref{lem: Nazarov} ahead), we have that with $P_{\cdot \mid \bZ}$-probability at least $1 - O(n^{-c'/6}) - 2C(\lambda_{n}/\lambda_{1,n})^{d}\beta(\lambda_{2,n};\lambda_{n}^{d})$,
\begin{align*}
    P_{\cdot \mid \bY,\bZ}\left(r_{n}^{-1/2}U_{n} \in A\right) \leq P_{\cdot \mid \bZ}\left(\bm{V}_{n} \in A\right) + O(n^{-c'/6})
\end{align*}
uniformly over $A \in \mathcal{A}$. Likewise, we have that with $P_{\cdot \mid \bZ}$-probability at least $1 - O(n^{-c'/6}) - 2C(\lambda_{n}/\lambda_{1,n})^{d}\beta(\lambda_{2,n};\lambda_{n}^{d})$, $P_{\cdot \mid \bY,\bZ}\left(r_{n}^{-1/2}U_{n} \in A\right) \geq P_{\cdot \mid \bZ}\left(\bm{V}_{n} \in A\right) - O(n^{-c'/6})$ uniformly over $A \in \mathcal{A}$. This completes the proof of Theorem \ref{DWBvalidity: HR}. \qed

\subsection{Proof of Theorem \ref{DWBvalidity: HR} under Assumption \ref{ass: design2}}  
The proof is similar to that of the previous case, so we only point out required modifications in each step.

%

\subsubsection{Conditional Gaussian approximation to $r_{n}^{-1/2}U_{11,n}$}
By Lemma 2.2.2 in \cite{vaWe96}, we have 
\begin{align*}
&E_{\cdot \mid \bZ}\left[\max_{\ell \in L_{n}}\max_{1 \leq j_{1},j_{2} \leq p}\left(E_{\cdot|\bX,\bZ}\left[|U_{n}^{(j_{1})}(\bm{\ell};\bm{\epsilon}_{0})U_{n}^{(j_{2})}(\bm{\ell};\bm{\epsilon}_{0})|\right]\right)^{2}\right]\\
&\lesssim [\![L_{n}]\!]^{1/2}\max_{\ell \in L_{n}}\left(E_{\cdot \mid \bZ}\left[\max_{1 \leq j_{1},j_{2} \leq p}\left(E_{\cdot|\bX,\bZ}\left[|U_{n}^{(j_{1})}(\bm{\ell};\bm{\epsilon}_{0})U_{n}^{(i_{2})}(\bm{\ell};\bm{\epsilon}_{0})|\right]\right)^{4}\right]\right)^{1/2}.
\end{align*}
Note that 
\begin{align*}
&\left(E_{\cdot|\bX,\bZ}\left[|U_{n}^{(j_{1})}(\bm{\ell};\bm{\epsilon}_{0})U_{n}^{(j_{2})}(\bm{\ell};\bm{\epsilon}_{0})|\right]\right)^{4}\\
&= \left(\sum_{\bs_{i_{1}}, \bs_{i_{2}} \in \Gamma(\bm{\ell};\bm{\epsilon}_{0}) \cap R_{n}}a\left({\|\bs_{i_{1}}-\bs_{i_{2}}\| \over \lambda_{2,n}}\right)X_{j_{1}}(\bs_{i_{1}})X_{j_{2}}(\bs_{i_{2}})\right)^{4}\\
&\lesssim \sum_{\bs_{i_{1}},\dots, \bs_{i_{8}} \in \Gamma(\bm{\ell};\bm{\epsilon}_{0}) \cap R_{n}}\!\!\!a\left({\|\bs_{i_{1}}-\bs_{i_{2}}\| \over \lambda_{2,n}}\right)a\left({\|\bs_{i_{3}}-\bs_{i_{4}}\| \over \lambda_{2,n}}\right)a\left({\|\bs_{i_{5}}-\bs_{i_{6}}\| \over \lambda_{2,n}}\right)a\left({\|\bs_{i_{7}}-\bs_{i_{8}}\| \over \lambda_{2,n}}\right)\\
&\quad \quad \times \prod_{k=1}^{8}\|\bm{X}(\bs_{i_{k}})\|_{\infty}.
\end{align*}
Then we have
\begin{align*}
\max_{\bm{\ell} \in L_{n}}E_{\cdot \mid \bZ}\left[\max_{1\leq j_{1},j_{2}\leq p}\left(E_{\cdot|\bX,\bZ}\left[|U_{n}^{(j_{1})}(\bm{\ell};\bm{\epsilon}_{0})U_{n}^{(j_{2})}(\bm{\ell};\bm{\epsilon}_{0})|\right]\right)^{4}\right] &\lesssim \lambda_{1,n}^{4d}\lambda_{2,n}^{4d}n^{8}\lambda_{n}^{-8d}E_{\cdot \mid \bZ}[\|\bm{X}(\bs)\|_{\infty}^{8}]\\
&\lesssim \lambda_{1,n}^{4d}\lambda_{2,n}^{4d}n^{8}\lambda_{n}^{-8d}M_{n}^{8}.
\end{align*}
This yields that 
\begin{align}
&E_{\cdot \mid \bZ}\left[\max_{\ell \in L_{n}}\max_{1 \leq j_{1},j_{2} \leq p}\left(E_{\cdot|\bX,\bZ}\left[|U_{n}^{(j_{1})}(\bm{\ell};\bm{\epsilon}_{0})U_{n}^{(j_{2})}(\bm{\ell};\bm{\epsilon}_{0})|\right]\right)^{2}\right] \nonumber \\
&\lesssim [\![L_{n}]\!]^{1/2}\lambda_{1,n}^{2d}\lambda_{2,n}^{2d}n^{4}\lambda_{n}^{-4d}M_{n}^{4} \lesssim \lambda_{1,n}^{3d/2}\lambda_{2,n}^{2d}n^{4}\lambda_{n}^{-7d/2}M_{n}^{4}. \label{moment-bound2-1}
\end{align}
Likewise,  we have 
\begin{align}
E_{\cdot \mid \bZ}\left[E_{\cdot|\bX,\bZ}\left[|U_{n}^{(j_{1})}(\bm{\ell};\bm{\epsilon}_{0})U_{n}^{(j_{2})}(\bm{\ell};\bm{\epsilon}_{0})|^{2}\right]\right] &\lesssim \lambda_{1,n}^{d}\lambda_{2,n}^{3d}n^{4}\lambda_{n}^{-4d}M_{n}^{4} \label{moment-bound2-2}
\end{align}
uniformly over $1 \leq j_{1}, j_{2} \leq p$.
Hence, combining Lemma \ref{lem: maximal inequality} ahead, we have
 \begin{align*}
 E_{\cdot \mid \bZ}[\breve{\Delta}] &\lesssim {\lambda_{1,n}^{d/2}\lambda_{2,n}^{3d/2}n^{2}\lambda_{n}^{-2d}(\lambda_{n}/\lambda_{1,n})^{d/2}M_{n}^{2}\sqrt{\log p} + \lambda_{1,n}^{d}\lambda_{2,n}^{d}n^{2}\lambda_{n}^{-2d}M_{n}^{2}\log p \over n^{2}\lambda_{n}^{-d}}\\
&= M_{n}^{2}(\lambda_{2,n}^{3d/2}\lambda_{n}^{-d/2}\sqrt{\log p} + \lambda_{1,n}^{d}\lambda_{2,n}^{d}\lambda_{n}^{-d}\log p) \\
&\lesssim n^{-c'}\log^{-2}p. 
 \end{align*}
Therefore, the same result as in Subsection \ref{condti-GA-proof} holds. The rest of the proof is completely analogous. This completes the proof. \qed

\section{Proof of Proposition \ref{m-approx-CARMA}}\label{Appendix: CARMA}
We focus on the case that $\bm{g}$ is a diagonal matrix since the proof for the general case is similar. 
\subsection{Proof of Proposition \ref{m-approx-CARMA} (i)}
Consider the following decomposition: 
\begin{align*}
\bY(\bs) &= \int_{\R^{d}}\bm{g}(\|\bs - \bm{u}\|)\psi_{0}\left(\|\bs - \bm{u}\| : m_{n}\right)\bm{L}(d\bm{u}) + \int_{\R^{d}}\bm{g}(\|\bs - \bm{u}\|)\left(1 - \psi_{0}\left(\|\bs - \bm{u}\| : m_{n}\right)\right)\bm{L}(d\bm{u})\\
&=: \bX^{(m_{n})}(\bs) + \bm{\Upsilon}^{(m_{n})}(\bs).
\end{align*}
Recall that $\bX^{(m_{n})}$ is $m_{n}$-dependent (with respect to the $\ell^2$-norm), i.e., $\bX^{(m_n)}(\bs_{1})$ and $\bX^{(m_n)}(\bs_{2})$ are independent when $\|\bs_{1}-\bs_{2}\|\geq m_n$. For any even integer $q \geq 8$, we have 
\begin{align*}
E_{\cdot \mid \bZ}[(\Upsilon^{(m_n)}_{j}(\bs_{i}))^{q}] &\leq \overline{M}\int_{\R^{d}}g_{j}^{q}(\|\bm{u}\|)\left(1- \psi_{0}\left(\|\bm{u}\| : m_{n}\right)\right)^{q}d\bm{u}\\
& \leq \overline{M}\xi_{0}\int_{\|\bm{u}\| \geq m_{n}/4}e^{-qr_{0}\|\bm{u}\|}\left|1 + {4 \over m_{n}}\left(\|\bm{u}\| - {m_{n} \over 2}\right)\right|^{q}d\bm{u}\\
& \leq \overline{M}\xi_{0}\int_{\|\bm{u}\| \geq m_{n}/4}e^{-qr_{0}\|\bm{u}\|}\left|1 + {4\|\bm{u}\| \over m_{n}}\right|^{q}d\bm{u}\\
&\leq 2^{q-1}\overline{M}\xi_{0}\int_{\|\bm{u}\| \geq m_{n}/4}e^{-qr_{0}\|\bm{u}\|}\left(1 + {4^{q}\|\bm{u}\|^{q} \over m_{n}^{q}}\right)d\bm{u}\\
&\lesssim \overline{M}\xi_{0}\int_{m_{n}/4}^{\infty}e^{-qr_{0}t}\left(1 + {4^{q}t^{q} \over m_{n}^{q}}\right)t^{d-1}dt \\
&\lesssim \overline{M}\xi_{0}m_{n}^{d-1}e^{-{qr_{0}m_{n} \over 4}}.
\end{align*}
By Markov's inequality and Lemma 2.2.2 in \cite{vaWe96}, we have 
\begin{align*}
P_{\cdot \mid \bZ}\left(\left\|\sum_{i=1}^{n}\bm{\Upsilon}^{(m_{n})}(\bs_{i})\right\|_{\infty} > \varrho \right) &\leq \varrho^{-1}E_{\cdot \mid \bZ}\!\left[\left\|\sum_{i=1}^{n}\bm{\Upsilon}^{(m_{n})}(\bs_{i})\right\|_{\infty}\right]\\ 
&\lesssim \varrho^{-1}p^{2/q}\max_{1 \leq j \leq p}\!\left(\!E_{\cdot \mid \bZ}\!\left[\left|\sum_{i=1}^{n}\Upsilon^{(m_n)}_{j}(\bs_{i})\right|^{q/2}\right]\right)^{2/q}  \\
&\leq \varrho^{-1}p^{2/q}\max_{1 \leq j \leq p}\!\left(\!E_{\cdot \mid \bZ}\left[n^{q/2}\max_{1 \leq i \leq n}\left|\Upsilon^{(m_n)}_{j}(\bs_{i})\right|^{q/2}\right]\right)^{2/q}\\
&\lesssim \varrho^{-1}p^{2/q}n\max_{1 \leq j \leq p}\!\!\left(\!n^{1/2}\!\max_{1 \leq i \leq n}\!\left(E_{\cdot \mid \bZ}\left[\left|\Upsilon^{(m_n)}_{j}(\bs_{i})\right|^{q}\right]\right)^{1/2}\!\right)^{2/q}\\
&\lesssim \varrho^{-1}p^{2/q}n^{(q+1)/q}\!\!\max_{1 \leq j \leq p}\!\left(E_{\cdot \mid \bZ}[|\Upsilon^{(m_n)}_{j}(\bs_{1})|^{q}]\right)^{1/q}\\
&\lesssim \varrho^{-1}p^{2/q}n^{(q+1)/q}m_{n}^{(d-1)/q}e^{-{r_{0}m_{n} \over 4}}.
\end{align*}
To obtain the fifth inequality, we used stationarity of $\bm{\Upsilon}^{(m_{n})}$. Let $\varrho = r_{n}^{1/2}$. Since $\lambda_{n}^{d/2} \lesssim n^{1/2}$, we have  
\begin{align*}
P_{\cdot \mid \bZ}\left(r_{n}^{-1/2}\left\|\sum_{i=1}^{n}\bm{\Upsilon}^{(m_{n})}(\bs_{i})\right\|_{\infty} >1 \right) &\lesssim p^{2/q}n^{1/q}\lambda_{n}^{d/2}m_{n}^{(d-1)/q}e^{-{r_{0}m_{n} \over 4}}\\
& \lesssim p^{2/q}n^{1/q +1/2}m_{n}^{(d-1)/q}e^{-{r_{0}m_{n} \over 4}}.
\end{align*}
Since $p = O(n^{\alpha})$ and $m_{n} = \delta \log n$, we have 
\begin{align*}
p^{2/q}n^{1/q+1/2}m_{n}^{(d-1)/q}e^{-{r_{0}m_{n} \over 4}} = n^{(\alpha+1)/q + 1/2}\delta^{(d-1)/q}(\log n)^{(d-1)/q}n^{-\delta r_{0}/4}.
\end{align*}
Thus, when $\delta > {2(2\alpha +q + 2) \over r_{0}q}$, we have
\begin{align*}
n^{(\alpha+1)/q + 1/2 - \delta r_{0}/4}(\log n)^{(d-1)/q} &\lesssim n^{-2\zeta}(\log n)^{-1/2}
\end{align*}
for some $\zeta>0$, i.e., 
\[
p^{2/q}n^{1/q+1/2}m_{n}^{(d-1)/q}e^{-{r_{0}m_{n} \over 4}} \lesssim n^{-2\zeta}(\log p)^{-1/2}. 
\]
This yields the desired result. \qed

\subsection{Proof of Proposition \ref{m-approx-CARMA} (ii)} 

Recall that diagonal components of $\bm{g}(\|\bm{v}\|)$ are of the form $g_{j,j}(\|\bm{v}\|) = \xi_{j,j}e^{-r_{j,j}\|\bm{v}\|}$. In this case, $\bY(\bs) = (Y_{1}(\bs), \dots, Y_{p}(\bs))'$ is given by
\[
Y_{j}(\bs) = \int_{\R^{d}}g_{j,j}(\|\bs - \bm{v}\|)L_{j}(d\bm{v}),\ j=1,\dots, p. 
\]
Then the marginal L\'evy density of $Y_{j}(\bs)$ is given by
\begin{align*}
\int_{\R^{d}}{1 \over \left|g_{j,j}(\|\bm{v}\|)\right|}\nu_{0,j}\left({x \over g_{j,j}(\|\bm{v}\|)}\right)d\bm{v}.
\end{align*}
See \cite{KaRoSpWa19} and \cite{Sa06} for more details. In the following proofs of (a) and (b), we will verify that we can take $M_{n} = n^{M_{0}}$ with small $M_{0}>0$. 

\subsubsection{Proof of Part (a)} The representation of the characteristic function of $\bY(\bs)$ given in Equation (\ref{CF-CARMA-RF}) implies that if $\bm{L}(\cdot)$ is Gaussian with triplet $(\bm{\gamma}_{0}, \Sigma_{0}, 0)$, then $\bY(\bs)$ is also Gaussian with mean $\int_{\R^{d}}\bm{g}(\|\bm{v}\|)\bm{\gamma}_{0}d\bm{v}$ and the covariance matrix
\[
\int_{\R^{d}}\bm{g}(\|\bm{v}\|)\Sigma_{0}\bm{g}(\|\bm{v}\|)'d\bm{v}. 
\]
Thus, $\|\bm{\gamma}_{0}\|_{\infty}\leq \overline{M}<\infty$ and $\max_{1 \leq j \leq p}\sigma_{0,j,j} \leq \overline{M} <\infty$ imply $\max_{1 \leq j \leq p}\|Y_{j}(\bs)\|_{\psi_{1}} \leq D_{0}$ for some $D_{0}<\infty$. This also implies that $\max_{1 \leq j \leq p}\|X^{(m_n)}_{j}(\bs)\|_{\psi_{1}} \leq D_{0}$, where $X^{(m_n)}_{j}$ is the $j$-th component of $\bX^{(m_{n})}$. Moreover, by Lemma 2.2.2 in \cite{vaWe96}, we have
\begin{align*}
E[\|\bm{X}^{(m_{n})}(\bs)\|_{\infty}^{q}] &\lesssim (\log p)^{q}\!\left(\max_{1 \leq j \leq p}\|X^{(m_n)}_{j}(\bs)\|_{\psi_{1}}\!\right)^{q}\!\!\!\\ 
&= \alpha^{q} (\log n)^{q}\!\left(\max_{1 \leq j \leq p}\|X^{(m_n)}_{j}(\bs)\|_{\psi_{1}}\!\right)^{q}\!\!\! = O((\log n)^{q}).
\end{align*}
Conclude that we can take $M_{n} = n^{M_{0}}$ with small $M_{0}>0$.

Now we move on to the proof of Part (b).
\subsubsection{Verification of Assumption \ref{ass: design2} Condition (i')}
Since $Y_{j}(\bs)$ is infinitely divisible, for $q \geq 1$, $E[|Y_{j}(\bs)|^{q}] = O(1)$ if and only if $\int_{|x|>1}|x|^{q}\nu_{0,j}(x)dx = O(1)$ (cf. Theorem 25.3 in \cite{Sa99}). Set $\underline{r} = r_{0} \wedge r_{0}^{\overline{M}}$ and $q > 1 \vee (\xi_{0}/\underline{r}-1)$. Observe that
\begin{align*}
\int_{\R^{d}}\int_{|z_{1}|\geq 1}\!\!\!|z_{1}|^{q}\nu_{0,1}\left({z_{1} \over g_{1,1}(\|\bm{v}\|)}\right)dz_{1}{1 \over \left|g_{1,1}(\|\bm{v}\|)\right|}d\bm{v} \lesssim \int_{\R^{d}}\!\!e^{\xi_{0}\|\bm{v}\|}\int_{|z_{1}| \geq 1}\!\!\!\!|z_{1}|^{q-\beta_{1}-1}e^{-\underline{M}e^{\underline{r}\|\bm{v}\|}|z_{1}|^{\alpha_{0,1}}}dz_{1}d\bm{v}.
\end{align*}
Define $c'_{1} = \underline{M}e^{\underline{r}\|\bm{v}\|}$. Then we have
\begin{align*}
\int_{|z_{1}| \geq 1}|z_{1}|^{q-\beta_{1}-1}e^{-c'_{1}|z_{1}|^{\alpha_{0,1}}}dz_{1} &\lesssim \int_{\R}|z_{1}|^{q}e^{-c'_{1}|z_{1}|}dz_{1} = 2\int_{0}^{\infty}z_{1}^{q}e^{-c'_{1}z_{1}}dz_{1} = {2q! \over (c'_{1})^{q+1}}.
\end{align*}
Thus, 
\begin{align*}
\int_{\R^{d}}e^{\xi_{0}\|\bm{v}\|}\int_{\R}|z_{1}|^{q}e^{-c'_{1}|z_{1}|}dz_{1}d\bm{v} &\lesssim {2q! \over \underline{M}^{q+1}}\int_{\R^{d}}e^{-((q+1)\underline{r}-\xi_{0})\|\bm{v}\|}d\bm{v}<\infty.
\end{align*}
This implies that $E[|Y_{j}(\bs)|^{q}] = O(1)$ for any $q > 1 \vee (\xi_{0}/\underline{r}-1)$. Moreover, when $\bm{g}$ is non-diagonal, we can also show that for any $q > 1 \vee (\xi_{0}/\underline{r}-1)$, $\max_{1 \leq j \leq p}E[|Y_{j}(\bs)|^{q}] \leq M_{0,q}<\infty$ where $M_{0,q}$ is independent of $n$, provided that
\begin{align*}
\xi_{j,j} &\neq 0, \max_{1 \leq j \leq p}[\![\{ k  : \xi_{j,k} \neq 0,k \neq j\}]\!] \leq \overline{M}, \\
r_{0}&\leq r_{j,k}\leq \xi_{0},\ |\xi_{j,k}| \leq \xi_{0},\ |\gamma_{0,j}|\leq \overline{M},\\
1 &\leq \alpha_{0,j} \leq \overline{M},\ -1\leq \beta_{j}\leq M_{\beta},\ \underline{M} \leq \overline{c}_{j}, C_{j} \leq \overline{M},
\end{align*}
for all $1 \leq j,k \leq p$. Indeed, define $Y_{j,k}(\bm{s}) = \xi_{j,k}\int_{\mathbb{R}^{d}}e^{-r_{j,k}\|\bm{s}\|}L_{k}(d\bm{s})$ and then we have $Y_{j}(\bm{s}) = \sum_{k=1,\xi_{j,k} \neq 0}^{p}Y_{j,k}(\bm{s})$ and 
\begin{align*}
 |Y_{j}(\bm{s})|^{q} &\lesssim \sum_{k=1,\xi_{j,k} \neq 0}^{p}|Y_{j,k}(\bm{s})|^{q} \leq \overline{M}\max_{1 \leq j,k \leq p,\xi_{j,k} \neq 0}|Y_{j,k}(\bm{s})|^{q}   
\end{align*}
for $q > 1 $. Similarly to the case when $\bm{g}$ is diagonal, we have $\max_{1 \leq j,k \leq p,\xi_{j,k} \neq 0}E[|Y_{j,k}(\bm{s})|^{q}] \leq M'_{0,q}<\infty$ where $M'_{0,q}$ is independent of $n$. 
Thus, for any $q > 1 \vee (\xi_{0}/\underline{r}-1)$ and $M'_{0}>1$, 
\begin{align*}
E[\|\bY(\bs)\|_{\infty}^{q}] &\leq \left(E[\|\bY(\bs)\|_{\infty}^{qM'_{0}}]\right)^{1/M'_{0}} \\
&\leq \left(E\left[\sum_{j=1}^{p}|Y_{j}(\bs)|^{qM'_{0}}\right]\right)^{1/M'_{0}} \\
&\lesssim \left(E\left[\sum_{j,k=1,\xi_{j,k}\neq 0}^{p}|Y_{j,k}(\bs)|^{qM'_{0}}\right]\right)^{1/M'_{0}}\\
&\leq \overline{M}^{1/M'_{0}}p^{1/M'_{0}}\max_{1 \leq j,k \leq p,\xi_{j,k} \neq 0}\left(E[|Y_{j,k}(\bs)|^{qM'_{0}}]\right)^{1/M'_{0}} \\
&= O(n^{\alpha /M'_{0}}).
\end{align*}
Likewise, we have $E[\|\bm{X}^{(m_{n})}(\bs)\|_{\infty}^{q}] = O(n^{\alpha /M'_{0}})$. Conclude that we can take $M_{n} = n^{M_{0}}$ with small $M_{0} = \alpha/M'_{0}>0$. From  the proof of Proposition \ref{m-approx-CARMA} (i) we also have $\phi_{n,\Upsilon} = e^{-M_{0,1}\lambda_{2,n}^{1/(\theta d)}}$ for some $M_{0,1}>0$. Hence we have verified Assumption \ref{ass: design2} Condition (i').

In the following proof, we will verify other regularity conditions, i.e. Condition (iv') in Assumption \ref{ass: design2}, Condition (\ref{cov_matrix_tail}), and Condition (\ref{HDGA_cov_ass}) in Assumption \ref{ass: design} (v).
\subsubsection{Verification of (\ref{high-level-condi-HDCLT}) in Assumption \ref{ass: design2} Condition (iv')}\label{C23}
Proposition \ref{m-approx-CARMA} (i) implies that we can assume that $\bX^{(m_{n})}$ is $\lambda_{2,n}^{1/(\theta d)}$-dependent with respect to the $\ell^2$-norm, i.e., $m_{n} = \lambda_{2,n}^{1/(\theta d)}$. This implies that 
\begin{align*}
\overline{\beta}_{q} &=1+\sum_{k=1}^{\lambda_{1,n}}k^{d-1}\beta_{1}^{1-2/q}(k) \sim \sum_{k=1}^{\lambda_{2,n}^{1/\theta d}}k^{d-1}\beta_{1}^{1-2/q}(k) = O(\lambda_{2,n}^{1/(\theta d)} \times \lambda_{2,n}^{(d-1)/(\theta d)}) = O(\lambda_{2,n}^{1/\theta}).
\end{align*}
Observe that
\begin{align*}
{1 \over \sqrt{\lambda_{2,n}}} &\sim {1 \over \lambda_{1,n}^{c_{1}/2}},\ M_{n}\sqrt{\overline{\beta}_{q}\lambda_{2,n} \over \lambda_{1,n}} \sim {n^{M_{0}} \over \lambda_{1,n}^{(1-c_{1}(\theta+1)/\theta)/2}},
\end{align*}
\begin{align*}
{\lambda_{1,n}^{d} \over \lambda_{n}^{d/2}}M_{n}^{2} \sim {n^{2M_{0}} \over \lambda_{n}^{d(1-c_{0})/2}},\ {M_{n}^{2}\lambda_{1,n}^{3d} \over (n^{2}\lambda_{n}^{-d})^{1-2/q}} \lesssim {n^{2M_{0}}\lambda_{n}^{3dc_{0}/2} \over n^{1-2/q}}.
\end{align*}
Since we can take $M_{0}$ sufficiently small, we can find $0<c'<c<1/2$ that satisfies Assumption \ref{ass: design2} Condition (iv') with $p$, $\lambda_{1,n}$, $\lambda_{2,n}$, $\lambda_{n}$ and $M_{n}$ in the assumption of Proposition \ref{m-approx-CARMA}.

\subsubsection{Verification of (\ref{asy-neg-Upsilon2}) in Assumption \ref{ass: design2} Condition (iv')}\label{C24}
The proof of Proposition \ref{m-approx-CARMA} (i) implies that we can take $\phi_{n,\Upsilon} = e^{-M_{0,1}\lambda_{2,n}^{1/(\theta d)}}$ for some $M_{0,1}>0$. Since $\phi_{n,\Upsilon}$ decays exponentially fast as $n \to \infty$, Assumption \ref{ass: design2} Condition (iv') is satisfied. 

\subsubsection{Verification of Conditions (\ref{HDGA_cov_ass}) and  (\ref{cov_matrix_tail})}\label{C25}
We first verify Condition (\ref{cov_matrix_tail}). Observe that
\begin{align*}
\left|\Cov(Y_{j_{1}}(\bs), Y_{j_{2}}(\bm{0}))\right| &= \left|\sigma_{j_{1},j_{2}}\int_{\R^{d}}g_{j_{1},j_{1}}(\|\bm{v}\|)g_{j_{2},j_{2}}(\|\bs - \bm{v}\|)d\bm{v}\right|\\
&\quad\leq \xi_{0}^{2} |\sigma_{j_{1},j_{2}}|\int_{\R^{d}}e^{-r_{j_{1},j_{1}}\|\bm{v}\|}e^{-r_{j_{2},j_{2}}\|\bs - \bm{v}\|}d\bm{v},
\end{align*} 
where $\sigma_{j_{1},j_{2}} = \Cov(L_{j_{1}}([0,1]^{d}), L_{j_{2}}([0,1]^{d})) \leq \sqrt{\Var(L_{j_{1}}([0,1]^{d}))}\sqrt{\Var(L_{j_{2}}([0,1]^{d}))} \leq \overline{M}$.
Without loss of generality, we can assume that $r_{j_{1},j_{1}} \geq r_{j_{2},j_{2}}$. Note that $(\|\bs\| - \|\bm{v}\|)/2 \leq \|\bs - \bm{v}\|$. Then we have 
\begin{align*}
\left|\Cov(Y_{j_{1}}(\bs), Y_{j_{2}}(\bm{0}))\right| &\leq \xi_{0}^{2}\overline{M}\int_{\R^{d}}e^{-r_{j_{1},j_{1}}\|\bm{v}\|}e^{-r_{j_{2},j_{2}}(\|\bs\| - \|\bm{v}\|)/2}d\bm{v}\\
&= \xi_{0}^{2}\overline{M}e^{-r_{j_{2},j_{2}}\|\bs\|/2}\int_{\R^{d}}e^{-(r_{j_{1},j_{1}} - r_{j_{2},j_{2}}/2)\|\bm{v}\|}d\bm{v} \lesssim \xi_{0}^{2}\overline{M}e^{-r_{0}\|\bs\|/2}.
\end{align*}
This yields that $\left|\Cov(Y_{j_{1}}(\bs), Y_{j_{2}}(\bm{0}))\right|$ decays exponentially fast as $\|\bs\|\to\infty$. Likewise, we have 
\[
|\Sigma_{j_{1},j_{2}}(\bs)| = \left|\Cov(X_{j_{1}}(\bs), X_{j_{2}}(\bm{0}))\right| \lesssim \xi_{0}^{2}\overline{M}e^{-r_{0}\|\bs\|/2}.
\]
Thus, (\ref{cov_matrix_tail}) is satisfied. Condition (\ref{HDGA_cov_ass}) immediately follows from the assumption of Part (b) and the argument in the proof of Proposition \ref{m-approx-CARMA} (i). This completes the overall proof. \qed

\begin{remark}[Non-Gaussian L\'evy-driven random fields with exponential moments]\label{bdd-supp-jump}
From Theorem 25.3 in \cite{Sa99}, $E[e^{|Y_{j}(\bs)|}]<\infty$ if and only if $\int_{|x|>1}e^{|x|}\nu_{0,j}(x)dx<\infty$. Assume that the supports $\text{Supp}(\nu_{0,j})$ of $\nu_{0,j} \ (1 \leq j \leq p)$ are uniformly bounded, i.e. there exists a constant $M>0$ independent of $n$ such that $\text{Supp}(\nu_{0,j}) \subset [-M,M]$ and 
\begin{align*}
\max_{1 \leq j \leq p}|\nu_{0,j}(x)| \leq \overline{M}e^{-\overline{c}_{0}|x|},\ x \in \R,
\end{align*}
for some $0<\overline{M}<\infty$ and $\overline{c}_{0}>1$ that are independent of $n$. 
This implies that the L\'evy measures $L_{j}(\cdot)$ have bounded jumps. Let $c'_{1} = \overline{c}_{0}e^{r_{1,1}\|\bm{v}\|}$. Then we have 
\begin{align*}
&\int_{\R^{d}}\int_{|z_{1}|\geq 1}e^{|z_{1}|}\nu_{0,1}\left({z_{1} \over g_{1,1}(\|\bm{v}\|)}\right)dz_{1}{1 \over \left|g_{1,1}(\|\bm{v}\|)\right|}d\bm{v}\\
&\lesssim \overline{M}\int_{\R^{d}}e^{r_{1,1}\|\bm{v}\|}\int_{1 \leq |z_{1}| \leq (M \vee 1)}e^{-(c'_{1}-1)|z_{1}|}dz_{1}d\bm{v} \lesssim \overline{M}\int_{\R^{d}}e^{-(\overline{c}_{0}e^{r_{0}\|\bm{v}\|} - \xi_{0}\|\bm{v}\|-1)}d\bm{v} < \infty,
\end{align*}
which implies that $E[e^{|Y_{j}(\bs)|}]$ (and $E[e^{|X^{(m_n)}_{j}(\bs)|}]$) is bounded uniformly over $1 \leq j \leq p$.
By Lemma 2.2.2 in \cite{vaWe96}, we have
\begin{align*}
E[\|\bm{X}^{(m_{n})}(\bs)\|_{\infty}^{q}] &\lesssim (\log p)^{q}\!\left(\max_{1 \leq j \leq p}\|X^{(m_n)}_{j}(\bs)\|_{\psi_{1}}\!\right)^{q} \\
&= \alpha^{q} (\log n)^{q}\!\left(\max_{1 \leq j \leq p}\|X^{(m_n)}_{j}(\bs)\|_{\psi_{1}}\!\right)^{q} \\
&= O((\log n)^{q}).
\end{align*}
Then, arguing as in Subsection \ref{C23}, \ref{C24} and \ref{C25}, we can verify that the conditions in Assumption \ref{ass: design} are satisfied  with $\lambda_{n}$, $\lambda_{1,n}$, $\lambda_{2,n}$, $D_{n}$ and $\delta_{n,\Upsilon}$ given in Remark \ref{Ass4.1CARMA}.
\end{remark}

\section{Proof of Proposition \ref{asy_valid_SDWB_CI}}\label{Appendix: CB_proof}

We wish to show that $P_{\cdot \mid \bZ}(\bm{\mu} \in \hat{C}^{\ast}_{1-\tau}) = 1-\tau + o(1)$.
Note that 
\[
\bm{\mu} \in \hat{C}^{\ast}_{1- \tau}\Leftrightarrow \sqrt{\lambda_{n}^{d}}\max_{1 \leq j \leq p}\left|{\overline{Y}_{j,n}- \mu_{j} \over \sqrt{\hat{\Sigma}_{j,j}^{\bm{V}_{n}}}}\right| \leq  \hat{q}_{n}(1-\tau).
\]
From  Subsections \ref{B11-BootHD}-\ref{B13-BootHD} in Appendix \ref{Append_B}, we have that with $P_{\cdot \mid \bZ}$-probability at least $1- O(n^{-c'/6}) - 2C(\lambda_{n}/\lambda_{1,n})^{d}\beta(\lambda_{2,n};\lambda_{n}^{d})$,
\begin{align*}
    \max_{1 \leq j,k \leq p}\left|\hat{\Sigma}_{j,k}^{\bm{V}_{n}} - \Sigma_{j,k}^{\bm{V}_{n}}\right| &\leq \breve{\Delta} + \hat{\Delta} + O(n^{-c'/6}\log^{-2} p)\\
    &= O(n^{-5c'/6}\log^{-2}p) + O(n^{-c'/6}\log^{-2} p) \\
    &= O(n^{-c'/6}\log^{-2} p).
\end{align*}
Since $\Sigma_{j,j}^{\bm{V}_{n}}$ are  bounded and  bounded away from $0$ uniformly over $1 \leq j \leq p$ (this follows from Lemma \ref{variance-rate}), we have that with $P_{\cdot \mid \bZ}$-probability at least $1- O(n^{-c'/6}) - 2C(\lambda_{n}/\lambda_{1,n})^{d}\beta(\lambda_{2,n};\lambda_{n}^{d})$,
\begin{align}
\max_{1 \leq j \leq p}\left|\sqrt{\Sigma_{j,j}^{\bm{V}_{n}} \over \hat{\Sigma}_{j,j}^{\bm{V}_{n}}} - 1\right| &\leq \max_{1 \leq j \leq p}\left|\sqrt{\Sigma_{j,j}^{\bm{V}_{n}} \over \hat{\Sigma}_{j,j}^{\bm{V}_{n}}} - 1\right|\left|\sqrt{\Sigma_{j,j}^{\bm{V}_{n}} \over \hat{\Sigma}_{j,j}^{\bm{V}_{n}}} + 1\right| = \max_{1 \leq j \leq p}\left|{\Sigma_{j,j}^{\bm{V}_{n}} \over \hat{\Sigma}_{j,j}^{\bm{V}_{n}}} - 1\right| \nonumber \\
&\leq \max_{1 \leq j \leq p}\left|{1 \over \hat{\Sigma}_{j,j}^{\bm{V}_{n}}}\right|\max_{1 \leq j \leq p}\left|\hat{\Sigma}_{j,j}^{\bm{V}_{n}} - \Sigma_{j,j}^{\bm{V}_{n}}\right|  \leq {O\left(n^{-c'/6}\log^{-2}p\right) \over \min_{1 \leq j \leq p}\Sigma_{j,j}^{\bm{V}_{n}} - \max_{1 \leq j \leq p}\left|\hat{\Sigma}_{j,j}^{\bm{V}_{n}} - \Sigma_{j,j}^{\bm{V}_{n}}\right|}  \nonumber \\
&= O(1) \times O\left(n^{-c'/6}\log^{-2}p\right) = O\left(n^{-c'/6}\log^{-2}p\right). \label{SDWB-var-rate}
\end{align}

Observe that 
\begin{align*}
&{\lambda_{n}^{d/2}(\overline{Y}_{j,n}- \mu_{j}) \over \sqrt{\hat{\Sigma}_{j,j}^{\bm{V}_{n}}}} = \sqrt{{\Sigma_{j,j}^{\bm{V}_{n}} \over \hat{\Sigma}_{j,j}^{\bm{V}_{n}}}} \cdot {\lambda_{n}^{d/2}(\overline{Y}_{j,n}- \mu_{j}) \over \sqrt{\Sigma_{j,j}^{\bm{V}_{n}}}}\\
&\quad = {\lambda_{n}^{d/2}(\overline{Y}_{j,n}- \mu_{j}) \over \sqrt{\Sigma_{j,j}^{\bm{V}_{n}}}} + \underbrace{\left(\sqrt{{\Sigma_{j,j}^{\bm{V}_{n}} \over \hat{\Sigma}_{j,j}^{\bm{V}_{n}}}}-1\right)}_{=O_{P_{\cdot \mid \bZ}}(n^{-c'/6}\log^{-2}p)} \cdot \underbrace{{\lambda_{n}^{d/2}(\overline{Y}_{j,n}- \mu_{j}) \over \sqrt{\Sigma_{j,j}^{\bm{V}_{n}}}}}_{=O_{P_{\cdot \mid \bZ}}(\log^{1/2} p)} \quad (\text{from (\ref{SDWB-var-rate}) and (\ref{max-rate-Xmean})})\\
&\quad = {\lambda_{n}^{d/2}(\overline{Y}_{j,n}- \mu_{j}) \over \sqrt{\Sigma_{j,j}^{\bm{V}_{n}}}} + O_{P_{\cdot \mid \bZ}}\left( n^{-c'/6}\log^{-3/2}p \right)
\end{align*}
uniformly over $1 \leq j \leq p$. Thus, Theorems \ref{Thm: GA_hyperrec} and \ref{Thm: GA_hyperrec-m} imply that there exists a sequence
of constants $\rho_{n,1} \to 0$ such that
\[
\sup_{t \in\R}\left|P_{\cdot \mid \bZ}\left( \sqrt{\lambda_{n}^{d}}\max_{1 \leq j \leq p}\left|{\overline{Y}_{j,n}- \mu_{j} \over \sqrt{\hat{\Sigma}_{j,j}^{\bm{V}_{n}}}}\right| \leq t\right) - P_{\cdot \mid \bZ}\left( \max_{1 \leq j \leq p}\left|{V_{j,n} \over \sqrt{\Sigma_{j,j}^{\bm{V}_{n}}}}\right| \leq t\right) \right|\leq \rho_{n,1}.
\]
Likewise, (\ref{SDWB-var-rate}) and Theorem \ref{DWBvalidity: HR} yield that there exists a sequence of constants $\rho_{n,2} \to 0$
such that
\[
\sup_{t\in\R}\left|P_{\cdot|\bY,\bZ}\left(\sqrt{\lambda_{n}^{d}}\max_{1 \leq j \leq p}\left|{\overline{Y}_{j,n}^{\ast}- \overline{Y}_{j,n} \over \sqrt{\hat{\Sigma}_{j,j}^{\bm{V}_{n}}}}\right| \leq t \right) - P_{\cdot \mid \bZ}\left(  \max_{1 \leq j \leq p}\left|{V_{j,n} \over \sqrt{\Sigma_{j,j}^{\bm{V}_{n}}}}\right| \leq t \right) \right| \leq \rho_{n,2}.
\]
Let $\Omega_{n}$ denote the event on which these inequalities hold
and let $q_{n}(1-\tau)$ denote the $(1-\tau)$-th quantile of $\max_{1 \leq j \leq p}|V_{j,n}/\sqrt{\Sigma_{j,j}^{\bm{V}_{n}}}|$.
Note that $P_{\cdot \mid \bZ}\left(\Omega_{n}\right) \to 1$ as $n \to \infty$. Define $\rho'_{n}= \rho_{n,1} \vee \rho_{n,2}$ (or we take $\rho'_{n}$ sufficiently slow if necessary).
Then on $\Omega_{n}$, we have
\begin{align*}
&P_{\cdot|\bY,\bZ}\!\left(\sqrt{\lambda_{n}^{d}}\max_{1 \leq j \leq p}\left|{\overline{Y}_{j,n}^{\ast}- \overline{Y}_{j,n} \over \sqrt{\hat{\Sigma}_{j,j}^{\bm{V}_{n}}}}\right| \leq q_{n}(1-\tau + \rho'_{n}) \! \right) \\
&\geq P_{\cdot \mid \bZ}\! \left(  \max_{1 \leq j \leq p}\left|{V_{j,n} \over \sqrt{\Sigma_{j,j}^{\bm{V}_{n}}}}\right| \leq q_{n}(1-\tau + \rho'_{n}) \! \right) -\rho'_{n}\\
&=1-\tau.
\end{align*}
We used the continuity of the distribution of $\max_{1 \leq j \leq p}|V_{j,n}/\sqrt{\Sigma_{j,j}^{\bm{V}_{n}}}|$
to obtain the last equation. 
This yields that on $\Omega_{n}$, $\hat{q}_{n}(1-\tau) \leq q_{n}(1-\tau+\rho'_{n})$.
Likewise, we have $q_{n}(1-\tau-\rho'_{n}) \leq \hat{q}_{n}(1-\tau)$ on $\Omega_{n}$. Then we have 
\begin{align*}
 & P_{\cdot \mid \bZ}\!\! \left( \sqrt{\lambda_{n}^{d}}\max_{1 \leq j \leq p}\left|{\overline{Y}_{j,n}- \mu_{j} \over \sqrt{\hat{\Sigma}_{j,j}^{\bm{V}_{n}}}}\right| \leq \hat{q}_{n}(1-\tau)\!\! \right) \\
 &\leq P_{\cdot \mid \bZ}\!\! \left(\sqrt{\lambda_{n}^{d}}\max_{1 \leq j \leq p}\left|{\overline{Y}_{j,n}- \mu_{j} \over \sqrt{\hat{\Sigma}_{j,j}^{\bm{V}_{n}}}}\right| \leq q_{n}(1-\tau + \rho'_{n})\!\! \right) +o(1)\\
 & \quad= P_{\cdot \mid \bZ}\left(\max_{1 \leq j \leq p}\left|{V_{j,n} \over \sqrt{\Sigma_{j,j}^{\bm{V}_{n}}}}\right| \leq q_{n}(1-\tau + \rho'_{n})\right) +o(1)\\
 & \quad =1-\tau+\rho'_{n}+o(1) =1-\tau+o(1).
\end{align*}
To obtain the third equation, we used the continuity of the distribution
of $\max_{1 \leq j \leq p}|V_{j,n}/\sqrt{\Sigma_{j,j}^{\bm{V}_{n}}}|$. Likewise, we
have
\[
P_{\cdot \mid \bZ}\left(\sqrt{\lambda_{n}^{d}}\max_{1 \leq j \leq p}\left|{\overline{Y}_{j,n}- \mu_{j} \over \sqrt{\hat{\Sigma}_{j,j}^{\bm{V}_{n}}}}\right| \leq \hat{q}_{n}(1-\tau)\right) \geq 1 - \tau - o(1).
\]
Further, by the Borel-Sudakov-Tsirelson inequality (Lemma A.2.2 in \cite{vaWe96}), we have
\[
q_{n}(1-\tau+\rho'_{n})\lesssim E_{\cdot \mid \bZ}\left[\max_{1 \leq j \leq p}|V_{j,n}/\sqrt{\Sigma_{j,j}^{\bm{V}_{n}}}|\right]+\sqrt{\log(1/(\tau-\rho'_{n}))} \lesssim \sqrt{\log p},
\]
which implies that $\hat{q}_{n}(1-\tau)=O_{P_{\cdot \mid \bZ}}(\sqrt{\log p})$.
Conclude that the maximum width of the joint confidence interval $\hat{C}_{1-\tau}$
is $2\lambda_{n}^{-d/2}\hat{q}_{n}(1-\tau)=O_{P_{\cdot \mid \bZ}}\left(\lambda_{n}^{-d/2}\sqrt{\log p}\right)$. \qed

\section{Technical tools}

Here we collect technical tools we used in the proofs. 

\begin{lemma}[A useful maximal inequality]
\label{lem: maximal inequality}
Let $\bZ_1,\dots,\bZ_n$ be independent random vectors in $\R^p$ with $p \ge 2$. Define $M=\max_{1 \le i \le n}\|\bZ_i\|_{\infty}$ and $\sigma^2 = \max_{1 \le j \le p}\sum_{i=1}^n E[Z_{ij}^2]$. Then we have 
\[
E \left [ \left \| {\textstyle \sum}_{i=1}^n (\bZ_i - E[\bZ_i]) \right \|_{\infty}\right] \le K\left ( \sigma \sqrt{\log p} + \sqrt{E[M^2]} \log p \right ),
\]
where $K$ is a universal constant.
\end{lemma} 
\begin{proof}
See Lemma 8 in \cite{ChChKa15}.
\end{proof}

The following anti-concentration inequality for Gaussian measures (called Nazarov's inequality in \cite{ChChKa17}), together with the Gaussian comparison inequality,  played  crucial roles in proving the high dimensional CLTs and asymptotic validity of the SWDB. 

\begin{lemma}[Nazarov's inequality]
	\label{lem: Nazarov}
	Let $\bm{G}=(G_1,\dots,G_p)'$ be a centered Gaussian vector in $\R^p$ such that $E[G^2_j]\ge \underline{\sigma}^{2}$ for all $j=1,\dots,p$ and some constant $\underline{\sigma}>0$. Then for every $\bm{t} = (t_1,\dots,t_p)' \in \R^p$ and $\delta >0$,
	\[
	P(\bm{G} \le \bm{t}+\delta )-P(\bm{G} \le \bm{t})\le \frac{\delta}{\underline{\sigma}} (\sqrt{2\log p} + 2). 
	\]
	Here $\bm{t}+\delta = (t_1+\delta,\dots,t_p+\delta)'$. 
\end{lemma}
\begin{proof}
	See Lemma A.1 in \cite{ChChKa17}.
\end{proof}

\begin{lemma}[Gaussian comparison]
	\label{lem: Gaussian comparison}
	Let $\bm{G}$ and $\bm{V}$ be centered Gaussian random vectors in $\R^{p}$ with covariance matrices $\Sigma^{G} = (\Sigma_{j,k}^{G})_{1 \le j,k \le d}$ and $\Sigma^{V} = (\Sigma_{j,k}^{V})_{1 \le j,k \le p}$, respectively, and let $\Delta = \| \Sigma^{G} - \Sigma^{V} \|_{\infty} := \max_{1 \le j,k \le p} |\Sigma_{j,k}^{G} - \Sigma_{j,k}^{V}|$. 
	Suppose that $\min_{1 \le j \le p} \Sigma_{j,j}^{G} \bigvee \min_{1 \le j \le p} \Sigma_{j,j}^{V} \ge \underline{\sigma}^{2}$ for some constant $\underline{\sigma} > 0$. 
	Then
	\[
	\sup_{A \in \mathcal{A}} | P (\bm{G} \in A) - P (\bm{V} \in A) | \le C \Delta^{1/3} \log^{2/3} p,
	\]
	where $\mathcal{A} = \{ \prod_{j=1}^{p}[a_j,b_j] : -\infty \le a_j \le b_j \le \infty, 1 \le j \le p \}$ is the collection of closed rectangles in $\R^{p}$ and $C$ is a constant that depends only on $\underline{\sigma}$.
\end{lemma}

\begin{proof}
	Implicit in the proof of Theorem 4.1 in \cite{ChChKa17}.
\end{proof}

\clearpage

\section{Real data analysis}\label{real-data-US-ppt}

We illustrate the usefulness of  our SDWB by conducting a change-point analysis of U.S. precipitation data. Our objective here is not to perform a detailed climatological analysis,
but rather to illustrate main aspects of the SDWB-based change-point procedure in Section~\ref{sec: applications}. The U.S. precipitation data are monthly precipitation (in millimeters) observed at weather stations all over U.S.A from 1895 through 1997, which is available from Institute for Mathematics Applied to Geosciences (IMAG): 
\url{http://www.image.ucar.edu/Data/U.S..monthly.met/U.S.monthlyMet.shtml.} In our data analysis, we focus on  annual total precipitation observed at 101 weather stations in 6 midwest states (Iowa, Minnesota, North Dakota, Nebraska, South Dakota, Wisconsin)  from 1919 till 1994. Note that these six midwest states have similar K\"oppen climate classification Dfa or Dfb (see e.g., \cite{BeZiMcVeBeWo18}), which facilitates the use of our method as approximate spatial stationarity is needed.
To remove the effects due to the heavy tail distribution, we apply the following transformation considered in \cite{GrKoRe17} and \cite{MaYa18}:
\[
Y(\bm{s},t) = \log(1+M(\bm{s},t)),
\]
where $M(\bm{s},t)$ is the original annual precipitation record of year $t$ at location $\bm{s}$.

Figure \ref{Fig: sites 6 states} shows a map of weather stations in the 6 midwest states. In this case, we have $101$ sampling sites that are irregularly scattered  over the sampling region. We set $\lambda_{n} = 15$, which corresponds to setting $R_0$ as approximately a rectangle of $65$km $\times$ $100$km (length$\times$side), to identify the locations of weather stations but any other positive value of $\lambda_n$ would suffice too. 
Note that the only tuning parameter involved in our procedure is the bandwidth $b_n$, which can be a function of $\lambda_n$ in practice.

\subsubsection*{Results}

Let $\mu_{j}=E[Y(\bs,t_{j})]$. We report the results of the stepdown procedure for change-point detection described in Example \ref{cp-test-temporal}. 
Specifically, we test the set of null hypotheses on the differences of adjacent transformed annual precipitation
 \begin{align*}
H_{j}: \mu_{j+1} - \mu_{j} = 0,\ 1\leq j \leq p-1
\end{align*}
against the alternatives $H'_{j}: \mu_{j+1} - \mu_{j} \neq 0,\ 1\leq j \leq p-1$.
We generate $1,500$ bootstrap observations and use the Bartlett kernel for the covariance function of the Gaussian random field $W = \{W(\bm{s}): \bm{s} \in \mathbb{R}^{2}\}$. To examine the sensitivity with respect to the choice of the bandwidth parameter, we set $b_n \in \{1,2,\dots,10\}$.



Figure \ref{Fig: sum-CP-SN} shows 21 detected change points for every $b \in \{1,2,\dots,10\}$. In particular, We found 12 upward and 9 downward change points in Step 1 and no change points are detected in Step 2 at a significance level of $0.01$. Figure \ref{Fig: Change-point-sum-SN-b} shows the results of Step 1 in change-point analysis with joint 95\% (dark gray) and 99\% (gray) confidence intervals when $b=1$ (top left), $b=4$ (top right), $b=7$ (bottom left) and $b=10$ (bottom right). It appears that the interval widths vary as we move from $b=1$ to $b=10$, with the smallest interval widths for $b=1$ and the largest ones corresponding to $b=7$.  
These results imply the temporal nonstationarity of annual precipitations of 6 midwest states, since the spatial averages have been experiencing significant shifts over time. 
Note that several change-point detection procedures have been applied to U.S. precipitation data in the statistics literature; see \cite{GaLuRo12}, \cite{GrKoRe17} and \cite{ZhMaNgYa19}, among others. However, due to the use of different data sets (with different time periods and weather locations), it is a bit difficult to compare our results to theirs directly. Furthermore, the modeling assumption, statistical framework and the change-point alternatives detected by these procedures are all different, which render a direct comparison even harder. 
Nevertheless, our results shed some new light on the U.S. precipitation data, and indicate its strong mean non-stationarity over time.  

\begin{figure}[H]
  \begin{center}
     \includegraphics[clip, width=9cm]{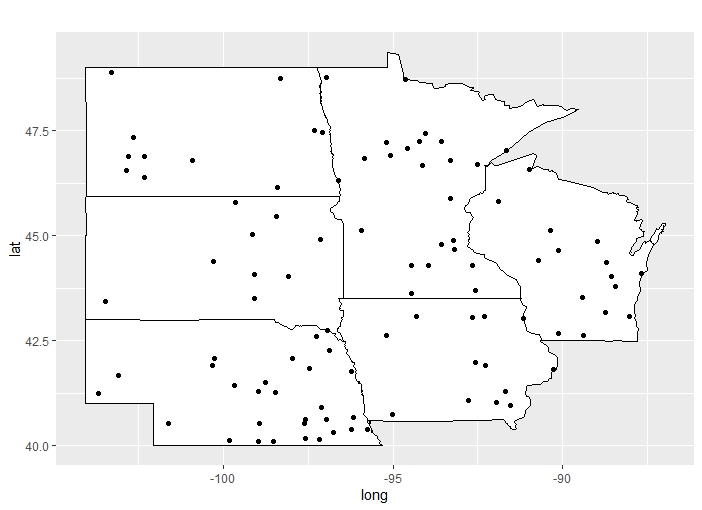}
       \caption{Map of weather stations (sampling sites) in 6 midwest states.}
       \label{Fig: sites 6 states}
     \end{center}
\end{figure}

\begin{figure}[H]
  \begin{center}
     \includegraphics[clip, width=9cm]{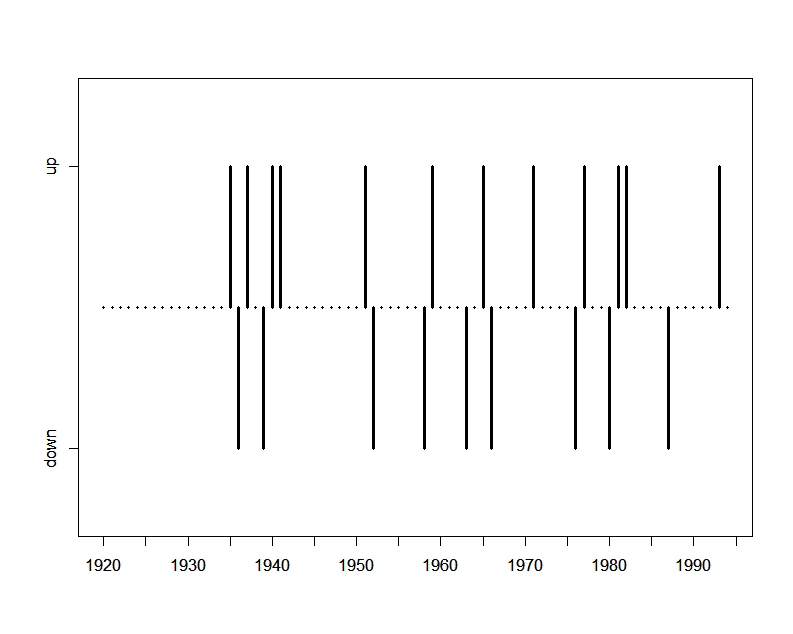}
        \caption{Change points detected in every $b \in \{1,2\dots,10\}$.}
    \label{Fig: sum-CP-SN}
  \end{center}
\end{figure}

\begin{figure}[H]
  \begin{center}
    \begin{tabular}{cc}

      \begin{minipage}{0.4\hsize}
        \begin{center}
          \includegraphics[clip, width=8cm]{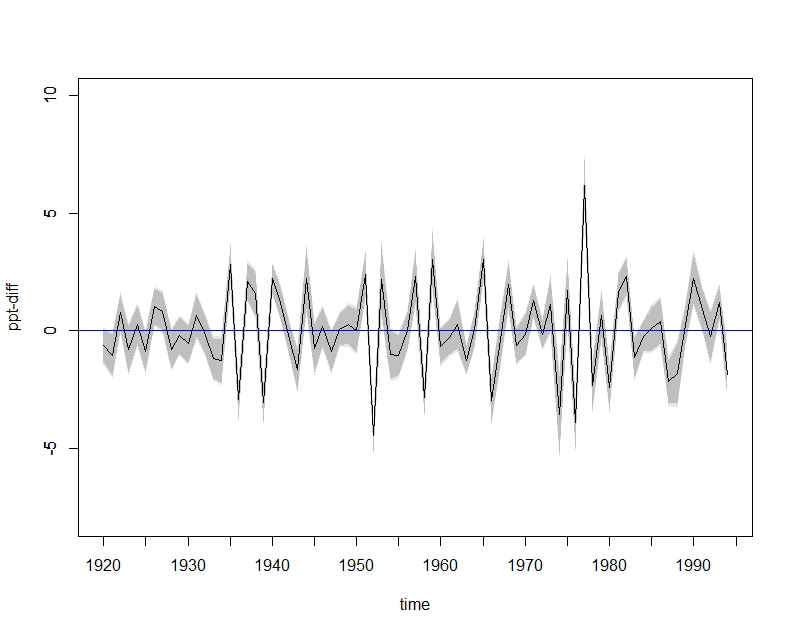}
        \end{center}
      \end{minipage}

       \begin{minipage}{0.6\hsize}
        \begin{center}
          \includegraphics[clip, width=8cm]{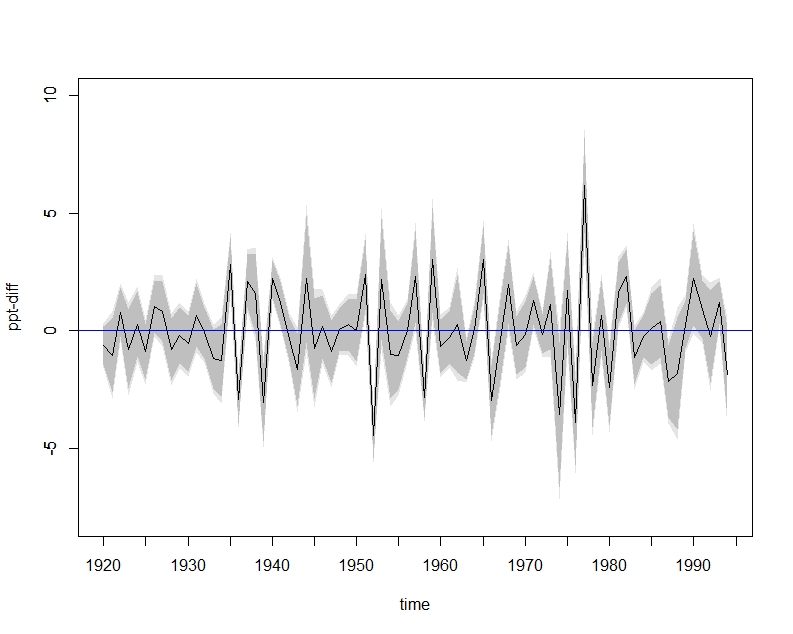}
        \end{center}
       \end{minipage} \\
       
       \begin{minipage}{0.4\hsize}
        \begin{center}
          \includegraphics[clip, width=8cm]{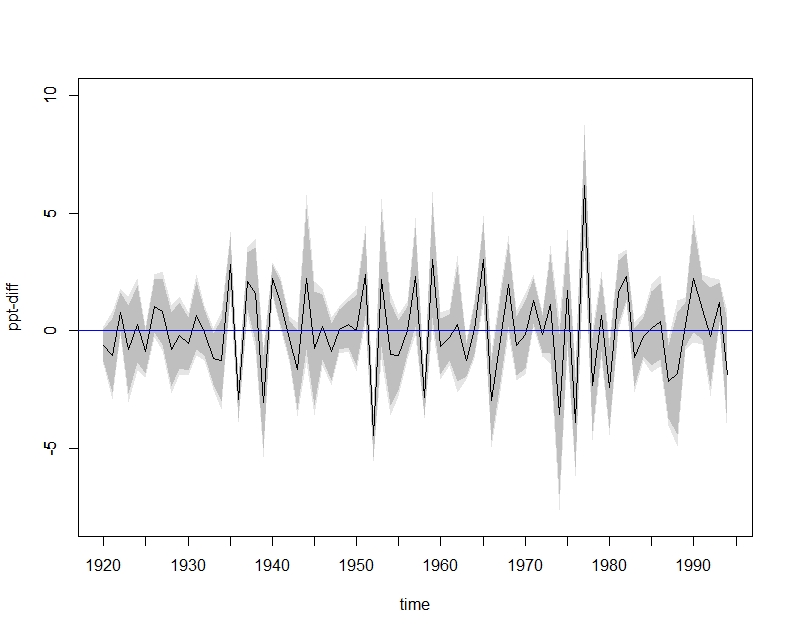}
        \end{center}
      \end{minipage}

       \begin{minipage}{0.6\hsize}
        \begin{center}
          \includegraphics[clip, width=8cm]{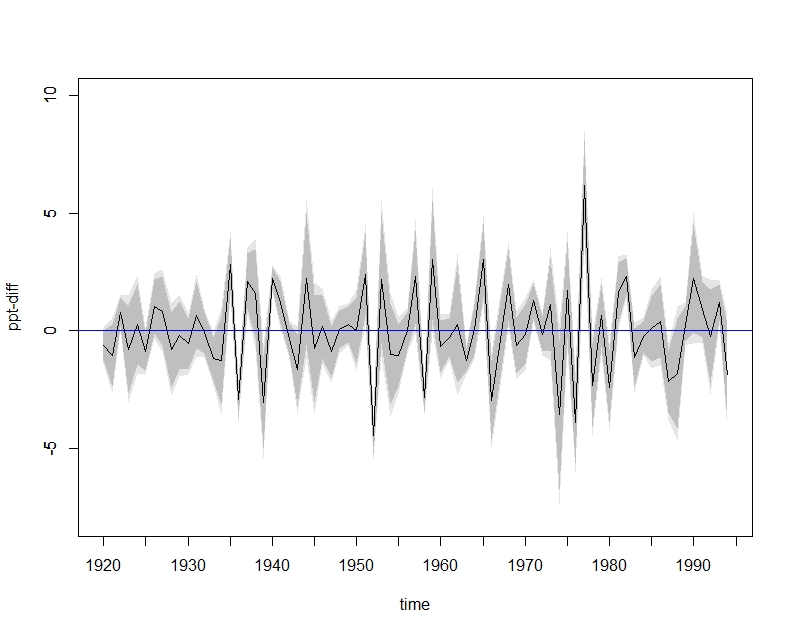}
        \end{center}
       \end{minipage} 

    \end{tabular}
    \caption{Results of the stepdown procedure for change-point analysis with joint 95\% (dark gray) and 99\% (gray) confidence intervals of differences of adjacent transformed annual precipitation data from 1920 till 1994. Each figure corresponds to the results of Step 1 with $b_{n}=1$ (top left), $b_{n} = 4$ (top right), $b_{n}=7$ (bottom left) and $b_{n}=10$ (bottom right). \label{Fig: Change-point-sum-SN-b}}
  \end{center}
\end{figure}

\clearpage

\end{document}